\documentclass[11pt]{amsart}

\usepackage{amssymb,latexsym}
\usepackage{amsmath,amscd}
\usepackage{graphicx,epsfig}
\usepackage{color}
\usepackage{verbatim}
\usepackage[justification=centering]{caption}

\usepackage{pdfsync} 

\usepackage{enumerate}

\def \R {\mathbb R}
\def \C {\mathbb C}
\def \N {\mathbb N}

\def \N {\mathbb N}
\def \F {\mathcal F}

\def \Ri {\mathcal R}
\def \S {\mathcal S}
\def \E {\mathcal E}
\def \A {\mathcal A}
\def \L {\mathcal L}
\def \P {\mathcal P}

\def \cC {\mathcal C}  
\def \cH {\mathcal H}  
\def \D {\mathcal D}  
\def \V {\mathcal V}
\def \cV {\mathcal V}
\def \W {\mathcal W}

\def \h {\mathfrak h}

\def \RE {\Re\text{\rm e}\,}
\def \IM {\Im\text{\rm m}\,}
\def \al {\alpha}
\def \ga {\gamma}
\def \la {\lambda}

\def \del {\delta}
\def \eps {\varepsilon}
\def \om {\omega}
\def \lan {\langle}
\def \ran {\rangle}
\def \de {\partial} 
\def \Box {\square}
\def \Boxbar {\overline{\square}}
\def \barBox {\overline\square}

\def \rtrans{\!{}^r}

\def \half{\frac12}

\def \inv{^{-1}}

\def \dom {\text{\rm dom\,}}

\def \span {\text{\rm span\,}}
\def \sgn {\text{\rm sgn\,}}
\def \range {\text{\rm ran\,}}
\def \tr {\text{\rm tr\,}}
\def \id {I}
\def \ex {\text{\rm -ex}}
\def \cl {\text{\rm -cl}}

\def\cD{{\mathcal D}} 
\def\cE{{\mathcal E}}   
\def\cA{{\mathcal A}}

\def\bpm{\begin{pmatrix}} 
\def\epm{\end{pmatrix}}
\def\Piu{+}
\def\Meno{-}

\def\bee{\begin{enumerate}}
\def\ee{\end{enumerate}}

\def\qed{\smallskip\hfill Q.E.D.\medskip}

\def\osum{\mathop{{\sum}^\oplus}}

\newtheorem{thm}{Theorem}[section]
\newtheorem{prop}[thm]{Proposition}
\newtheorem{cor}[thm]{Corollary}
\newtheorem{lemma}[thm]{Lemma}
\newtheorem{remark}[thm]{Remark}
\newtheorem{defn}[thm]{Definition}

\begin{document}

\title [Hodge Laplacian on the Heisenberg group]{Analysis of the Hodge
  Laplacian on the Heisenberg group}

\author[D. M\"uller]{Detlef M\"uller}
\address{Mathematisches Seminar, Christian-Albrechts-Universit\"at zu Kiel,
Ludewig-Meyn-Stra\ss{}e 4, D-24098 Kiel, Germany} 
\email{{\tt mueller@math.uni-kiel.de}}
\urladdr{{http://analysis.math.uni-kiel.de/mueller/}}

\author[M. M. Peloso]{Marco M. Peloso}
\address{Dipartimento di Matematica, Universit\`a degli Studi di Milano,
Via C. Saldini 50,
20133 Milano, Italy } 
\email{{\tt marco.peloso@unimi.it}}
\urladdr{{http://users.mat.unimi.it/~peloso/}}

\author[F. Ricci]{Fulvio Ricci}
\address{Scuola Normale Superiore, Piazza dei Cavalieri
7, 56126 Pisa, Italy } 
\email{{\tt fricci@sns.it}}
\urladdr{{http://www.math.sns.it/HomePages/Ricci/}}


\thanks{This work has been supported by the IHP network HARP ``Harmonic  
Analysis and Related problems'' of the European Union and by the project PRIN 2007
 ``Analisi Armonica''.
Part of this work was done during a visit at the Centro De Giorgi,
Pisa.  We thank this institution for the generous hospitality and the 
support provided.}

\begin{abstract}
We consider the Hodge Laplacian $\Delta$ on  the
Heisenberg group $H_n$, endowed with a left-invariant and
$U(n)$-invariant Riemannian metric.  For $0\le k\le 2n+1$, let $\Delta_k$ denote the Hodge
Laplacian restricted to $k$-forms.  

Our first main result shows 
that $L^2\Lambda^k(H_n)$ 
decomposes
into finitely many mutually orthogonal
subspaces $\V_\nu$ with the properties: 
\begin{itemize}
\item $\dom \Delta_k$ splits along the $\V_\nu$'s as $\sum_\nu(\dom\Delta_k\cap \V_\nu)$;
\item $\Delta_k:(\dom\Delta_k\cap \V_\nu)\longrightarrow \V_\nu$ for every $\nu$;
\item for each $\nu$, there is a Hilbert space $\cH_\nu$ of
  $L^2$-sections of a $U(n)$-homogeneous vector bundle over $H_n$ such
  that the restriction of $\Delta_k$ to $\V_\nu$ is unitarily
  equivalent to an explicit scalar operator. 
\end{itemize}

Next, we consider $L^p\Lambda^k$, $1<p<\infty$, and
prove that the same kind of decomposition holds true.
More precisely we show that:
\begin{itemize}
\item the Riesz transforms $d\Delta_k^{-\half}$ are $L^p$-bounded;
\item  the orthogonal projection onto $\cV_\nu$ extends
from $(L^2\cap L^p)\Lambda^k$ to  a bounded operator from
  $L^p\Lambda^k$ to the the $L^p$-closure $\cV_\nu^p$ of $\cV_\nu\cap
  L^p\Lambda^k$. 
\end{itemize}
  
We then use this decomposition to prove a Mihlin--H\"ormander multiplier theorem for each
$\Delta_k$.  We show that the operator $m(\Delta_k)$ is bounded on
$L^p\Lambda^k(H_n)$ for all $p\in(1,\infty)$ and all $k=0,\dots,2n+1$, provided $m$  satisfies a Mihlin--H\"ormander condition of order 
$\rho>(2n+1)/2$.
We also prove that this restriction on $\rho$ is optimal and extend this result to the Dirac operator.
\end{abstract} 

\maketitle

\makeatletter
\renewcommand\l@subsection{\@tocline{2}{0pt}{3pc}{5pc}{}}
\makeatother
\tableofcontents

\thispagestyle{empty}

\setcounter{equation}{0}

\section*{Introduction}

The theory of the Hodge Laplacian $\Delta$ on a complete Riemannian
manifold $M$ shows deep connections between 
geometry, topology and analysis
on $M$. While this theory is well developed in the case of
functions, i.e., for the  Laplace--Beltrami operator, much less is
known for forms of higher degree on a non-compact manifold. In
particular, one basic question that one would like to answer is
whether the Riesz transform $d\Delta^{-\half}$ is $L^p$-bounded in the
range $1<p<\infty$. According to \cite{Str}, this property is relevant
for establishing the Hodge decomposition in $L^p$ for differential
forms, cf. \cite{ACDH, Li, Loh} and references therein.

In a similar way, functional calculus on self-adjoint, left-invariant
Laplacians and sublaplacians $L$ on Lie groups and more general
manifolds has been widely studied, cf. \cite{A, Ank, AnkLoh, Ch, CM, ClS, CKS, Hebish, HZ, H,Mar, LuM, LuMS, MM, 
  MS, Sik, Taylor}. A key
question concerns the possibility that, for a given  $L$, a
Mihlin--H\"ormander condition of finite order on the multiplier
$m(\la)$ implies that the operator $m(L)$ is bounded on $L^p$ for
$1<p<\infty$. A second fundamental question is the $L^p$-boundedness,
in the same range of $p$, of the Riesz trasforms $XL^{-\half}$ for
appropriate  left-invariant vector fields $X$ \cite{CD, CMZ, GS, Loh2, LohMu}. 

Also in these situations, not much is known for operators which act on
sections of some homogeneous linear bundle over a given group. The
most notable case is that of sublaplacians associated to  the
$\bar\de_b$-complex on homogeneous CR-manifold \cite{CKS, FS} 
\medskip

In this  paper we  consider the Hodge Laplacian $\Delta$ on  the
Heisenberg group $H_n$, endowed with a left-invariant and
$U(n)$-invariant Riemannian metric, and give answers to the above
questions.   

The rich structure of the Heisenberg group makes it a natural model to
explore such questions in detail.  
First of all, it has a natural CR-structure, with a well-understood
Kohn Laplacian \cite{FS}, and nice interactions with the Riemannian
structure \cite{MPR}.

For  operators on $H_n$ which act on scalar-valued functions and are left-
and $U(n)$-invariant, the methods of Fourier analysis are quite handy
to study  spectral resolution, and sharp multiplier theorems for
differential operators of this kind are known \cite{MRS1, MRS2}.   
This class of operators is based on two commuting differential operators,
namely the {\it sublaplacian} $L$ and the {\it central derivative}
$T$, in the following sense: 
\begin{itemize}
\item[--] the left- and $U(n)$-invariant  differential operators on 
$H_n$ are the polynomials in $L$ and~$T$;
\item[--] the left- and $U(n)$-invariant self-adjoint operators on
  $L^2(H_n)$ containing the Schwartz space in their domain are the
  operators $m(L,i\inv T)$, with $m$ a real spectral multiplier. 
\end{itemize}

The same methods also allow to study operators acting on
  differential forms,  like  the Kohn Laplacian, which have the
  property of acting 
componentwise with respect to a canonical basis of left-invariant
forms, cf. \eqref{zeta-j-def}.

\smallskip

On the other hand, the Hodge Laplacian restricted to $k$-forms, which
we denote by  $\Delta_k$ and whose explicit expression is given in \eqref{1.12} below, is far from acting componentwise.

Nevertheless, we are able to reduce the spectral analysis of
$\Delta_k$  to that of a finite family of explicit
{\it scalar} operators. We call scalar an operator on some  space of
differential forms which can be expressed as $D\otimes I$, i.e., which
acts separately on each  scalar component of a given form by the same
operator $D$. 

We do so
by introducing a decomposition
of $L^2\Lambda^k(H_n)$ into finitely many mutually orthogonal
subspaces $\V_\nu$ with the following properties: 
\begin{enumerate}
\item[(i)] $\dom \Delta_k$ splits along the $\V_\nu$'s as $\sum_\nu(\dom\Delta_k\cap \V_\nu)$;
\item[(ii)] $\Delta_k:(\dom\Delta_k\cap \V_\nu)\longrightarrow \V_\nu$ for every $\nu$;
\item[(iii)] for each $\nu$, there is a Hilbert space $\cH_\nu$ of
  $L^2$-sections of a $U(n)$-homogeneous vector bundle over $H_n$ such
  that the restriction of $\Delta_k$ to $\V_\nu$ is unitarily
  equivalent to a scalar operator $m_\nu(L,i\inv T)$ acting
  componentwise on $\cH_\nu$; 
\item[(iv)] there exist unitary operators
  $U_\nu:\cH_\nu\longrightarrow \cV_\nu$ intertwining $m_\nu(L,i\inv
  T)$ and $\Delta_k$ which are either bounded multiplier operators
  $u_\nu(L,i\inv T)$, or compositions of such operators with  the
  Riesz transforms 
$$
R=d\Delta^{-\half}\ ,\qquad \Ri=\de\Box^{-\half}\ ,\qquad \bar\Ri=\bar\de\Boxbar^{-\half}\ .
$$
\end{enumerate}

This is done in the first part of the paper (Sections
\ref{cores}-\ref{Hodge}), and we refer to this part as to the
``$L^2$-theory''.  The main results in this context are Theorems \ref{dec-k-le-n}
and \ref{dec-k>n}, resp., where we obtain the decomposition of $L^2\Lambda^k$
into the  $\Delta_k$-invariant subspaces  $\V_\nu$, respectively for $0\le k\le n$ and $n+1\le
k\le 2n+1$.

This decomposition is fundamental for all the second
  part of the paper, which we are going to describe next. A quick
  description of the logic and the basic ideas in the 
construction of the $\V_\nu$ is postponed to the last part of this introduction.
\medskip

In Sections \ref{Lp}-\ref{appendix},  we
develop the ``$L^p$-theory''\footnote{There is an
  unfortunate notational conflict, due to the fact that the  letter
  $p$ is the commonly used symbol  for both Lebesgue spaces and
  bi-degrees of forms. In this introduction and in the titles of
  sections we keep the notation $L^p$, while in
the body of the paper we will
  denote by $L^r$ the generic Lebesgue space.}. 
We prove that, for $1<p<\infty$, the same kind of decomposition also
takes place  in $L^p\Lambda^k$. Precisely: 
\begin{enumerate}
\item[(a)] the intertwining operators $U_\nu$ in~(iv) have $L^p$-bounded extensions;
\item[(b)] consequently, the orthogonal projections $U_\nu^*U_\nu$
  from $L^2\Lambda^k$ to $\cV_\nu$ extend to bounded operators from
  $L^p\Lambda^k$ to the the $L^p$-closure $\cV_\nu^p$ of $\cV_\nu\cap
  L^p\Lambda^k$;
\item[(c)] the Riesz transforms $R_k=d\Delta_k^{-\half}$ are $L^p$-bounded;
\item[(d)] the $L^p$-strong Hodge decomposition holds true for
  $k=0,\dots,2n+1$, and more precisely $L^p\Lambda^k$ is direct sums
  of the subspaces $V_\nu^p$'s.
\end{enumerate}

 The (much simpler) case of 1-forms was already considered in \cite{MPR}.
We expect that our results can be applied to the study on conformal
invariants on quotients of $H_n$, along the lines of \cite{MPR2},
cf. \cite{L, Lu}.  
\medskip

In our last main result, as consequence of the $L^p$-theory we develop, 
we prove  a Mihlin--H\"ormander multiplier theorem for
$\Delta_k$, for all $k=0,\dots,2n+1$. 
 We show that, if  $m:\R\to \C$ is a bounded, continuous
function  satisfying a Mihlin--H\"ormander condition of order 
$\rho>(2n+1)/2$, then, for $1<p<\infty$,
the operator $m(\Delta_k)$ is bounded on
$L^p(H_n)\Lambda^k$, with norm bounded by the appropriate norm of
$m$  (cf. Theorem \ref{hodgethm}). 

\medskip

We briefly comment on some interesting aspects of the proof and on some
consequences and applications. It is always assumed that
$1<p<\infty$. 
\begin{itemize}
\item  Our inductive strategy requires that two statements be proved
  simultaneously at each step:  property~(a) above for the given $k$
  and $L^p$-boundedness of the Riesz transform
  $R_k=d\Delta_k^{-1/2}$. Precisely, the validity of (a) for a given
  $k$ implies $L^p$-boundedness of $R_k$, and this, in turn, is
  required to prove (a) for $k+1$. 
\item In order to handle the complicated expressions of the
  intertwining operators $U_\nu$, we identify certain symbol classes, denoted by $\Psi^{\rho,
    \sigma}_\tau$, which satisfy simple composition properties,
  contain all the scalar components of the $U_\nu$, and, when bounded,
  give $L^p$-bounded operators (cf. Subsection \ref{Psi-classes}).  
\item Taking as the initial definition of  ``exact $L^p$-form''  a
  form $\omega$ which is the $L^p$-limit  of a sequence of exact test
  forms (cf. Proposition \ref{s4.5} for $p=2$), we prove in Subsection
  \ref{subsec-exact} that this condition is equivalent to saying  that
  $\omega$ is in $L^p$ and  a differential in the sense of
  distributions. Incidentally, this allows to prove that the reduced
  $L^p$-cohomology of $H_n$ is trivial for every $k$. 
\item The Mihlin--H\"ormander theorem for spectral multipliers of
  $\Delta_k$, proved in \cite{MPR} for $k=1$, extends to every $k$. 
\item Our  analysis of $\Delta$ easily yields analogous results for
  the Dirac operator $d+d^*$. Studying the Hodge laplacian first has
  the advantage of isolating one order of forms at a time.  
Corollary \ref{diracthm} is a multiplier theorem  for the Dirac
operator completely analogous to Theorem \ref{hodgethm}. 
\end{itemize}
\medskip

\noindent{\it Outline of the decomposition of $L^2\Lambda^k$.}
\medskip

We go back now to the construction of the subspaces $\V_\nu$ of $L^2\Lambda^k(H_n)$. 
\medskip

First of all, by Hodge duality, we may restrict ourselves to form of degree $k\le n$. 
We start from the primary decomposition into exact and $d^*$-closed forms:
$$
L^2\Lambda^k(H_n)=(L^2\Lambda^k)_{d\ex}\oplus (L^2\Lambda^k)_{d^*\cl}\ ,
$$
where each summand is $\Delta_k$-invariant.

Since the Riesz transform $R_{k-1}=d\Delta_{k-1}^{-\half}$ commutes
with $\Delta$ and transforms $(L^2\Lambda^{k-1})_{d^*\cl}$ onto
$(L^2\Lambda^k)_{d\ex}$ unitarily, any $\Delta$-invariant subspace
$\V_\nu$ of  $(L^2\Lambda^{k-1})_{d^*\cl}$ has a twin
$\Delta$-invariant subspace $R_{k-1}\V_\nu$ inside
$(L^2\Lambda^k)_{d\ex}$. 

The analysis is so reduced to the space of $d^*$-closed
forms. Associated with the CR-structure of $H_n$, there is a natural
notion of {\it horizontal $(p,q)$-form} as a section of the bundle 
$$
\Lambda^{p,q}=\Lambda^{p,q}T^*_\C H_n\ ,
$$
and of {\it horizontal $k$-form} as a section of
$\Lambda_H^k=\osum_{\!\!\!\!p+q=k}\Lambda^{p,q}$. 

Every differential form $\omega$ decomposes uniquely as
$$
\om=\omega_1+\theta\wedge\omega_2\ ,
$$
where $\om_1,\om_2$ are horizontal and $\theta$ is the contact
form. Moreover, a $d^*$-closed form $\om$ is uniquely determined by
its horizontal component $\om_1$. 

From now on it is very convenient to introduce a special ``test
space''  $\S_0$, contained in the Schwartz space, together with
its corresponding spaces of forms, $\S_0\Lambda^k$,
$\S_0\Lambda^{p,q}$ etc., which are cores  for $\Delta_k$ and the
other self-adjoint operators that will appear.  

For forms in the core, we have enough flexibility to perform all the
required operations in a rather formal way, leaving the extensions to
$L^2$-closures for the very end. 
For instance, we can say that to every horizontal form $\om_1$ in the
core we can associate a ``vertical component'' $\theta\wedge\om_2$,
also in the core, to form a $d^*$-closed form
$\om_1+\theta\wedge\om_2$ in the core. 

Setting $\Phi(\om_1)=\om_1+\theta\wedge\om_2$, we can replace
$\Delta_k$ by the conjugated (but no longer differential) operator
$D_k=\Phi\inv\circ\Delta_k\circ\Phi$, which acts now  on the space of
horizontal $k$-forms in the core and globally defined. 

Here comes into play another invariance property of $\Delta_k$, which
is easily read as a property of $D_k$ and involves the horizontal
symplectic form $d\theta$. The following identity holds (cf. Lemma
\ref{D-e}) for a horizontal form $\om$ of degree $k$: 
\begin{equation}\label{conj-dtheta}
D_k(d\theta\wedge\om)=d\theta\wedge (D_{k-2}+n-k+1)\om\ .
\end{equation}

This brings in the Lefschetz decomposition of the space of horizontal
forms, as adapted in \cite{MPR} from the classical context of K\"ahler
manifolds \cite{W}. Denoting by $e(d\theta)$ the operator of exterior
multiplication by $d\theta$ and by $i(d\theta)$ its adjoint, it is
then natural to think of the core $\S_0\Lambda^{p,q}$ in the space of
horizontal $(p,q)$-forms as the direct sum 
$$
\S_0\Lambda^{p,q}=\sum_{j=0}^{\min\{p,q\}}e(d\theta)^j\ker i(d\theta)\ .
$$

Here each summand is $D_k$-invariant,  and the conjugation formula
\eqref{conj-dtheta} allows us to focus our attention on $\ker
i(d\theta)$. 

Nevertheless, $\ker i(d\theta)$ still is too big a space to allow a
reduction of $D_k$ to scalar operators. It is  however easy to
identify, for each pair $(p,q)$ with $p+q=k$, a proper $D_k$-invariant
subspace of $\ker i(d\theta)\cap \S_0\Lambda^{p,q}$, namely 
$$
W_0^{p,q}=\{\om\in \S_0\Lambda^{p,q}:\de^*\om=\bar\de^*\om=0\}\ .
$$

It turns out that $D_k$ acts as a scalar operator on $W_0^{p,q}$, so
that the $L^2$-closure $\V_0^{p,q}$ of $\Phi(W_0^{p,q}\big)$ will be
one of the spaces $\V_\nu$ we are looking for\footnote{$W_0^{p,q}$ is
  nontrivial, unless $p+q=n$ and $0<p<n$,  cf. Proposition \ref{non-triviality-lemma}.} 

Next, we take to the orthogonal complement of 
$$
W_0^k=\osum_{p+q=k}W_0^{p,q}
$$ 
in $\S_0\Lambda_H^k$. We have
$$
(W_0^k)^\perp=\big\{\de\xi+\bar\de\eta:\xi,\eta\in \S_0\Lambda_H^{k-1}\big\}\ ,
$$
and we can telescopically expand this splitting to obtain that
$$
\begin{aligned}
 \S_0\Lambda_H^k
&=W_0^k\oplus \big\{\de\xi+\bar\de\eta:\xi,\eta\in \S_0\Lambda_H^{k-1}\big\}\\
 &=W_0^k\oplus \big\{\de\xi+\bar\de\eta:\xi,\eta\in
 W_0^{k-1}\big\}\oplus 
\big\{\bar\de\de\xi+\de\bar\de\eta:\xi,\eta\in \S_0\Lambda_H^{k-2}\big\}\\
 &=W_0^k\oplus \big\{\de\xi+\bar\de\eta:\xi,\eta\in
 W_0^{k-1}\big\}\oplus 
\big\{\bar\de\de\xi+\de\bar\de\eta:\xi,\eta\in W_0^{k-2}\big\}\oplus\cdots\\
 &=\osum_{p+q=k}W_0^{p,q}\oplus 
\osum_{p+q=k-1}\big\{\de\xi+\bar\de\eta:\xi,\eta\in W_0^{p,q}\big\}\\
 &\phantom{=sum_{p+q=k}W_0^{p,q}}\oplus 
\osum_{p+q=k-2}\big\{\bar\de\de\xi+\de\bar\de\eta:\xi,\eta\in
W_0^{p,q}\big\}\oplus \cdots
 \end{aligned}
 $$

The subspaces  
$$
\begin{aligned}
W_1^{p,q}&=\big\{\de\xi+\bar\de\eta:\xi,\eta\in W_0^{p,q}\big\}\qquad (p+q=k-1)\\
W_2^{p,q}&=\big\{\bar\de\de\xi+\de\bar\de\eta:\xi,\eta\in W_0^{p,q}\big\}\qquad (p+q=k-2)\\
{\rm etc.}\ &
\end{aligned}
$$
generated in this way are $D_k$-invariant and mutually orthogonal. 

Matters are simplified by the fact that, for $j\ge1$,
$$
W_{j+2}^{p,q}=e(d\theta)W_j^{p,q}\ .
$$

So only $W_0^{p,q}$, $W_1^{p,q}$ and part of $W_2^{p,q}$ are contained in $\ker i(d\theta)$. Setting
$$
W^{p,q}_{1,\ell}=e(d\theta)^\ell W^{p,q}_1\ ,\qquad W^{p,q}_{2,\ell}=e(d\theta)^\ell W^{p,q}_2\ ,
$$
we obtain that
\begin{eqnarray*}
\S_0\Lambda^k_H&=&\osum_{p+q=k}W^{p,q}_0\oplus
\osum_{p+q+2\ell=k-1}W^{p,q}_{1,\ell}\oplus \osum_{ p+q+2\ell=k-2}W^{p,q}_{2,\ell}\ .
\end{eqnarray*}

On each $W^{p,q}_0$ and $W^{p,q}_2$, $D_k$ acts as a scalar operator, as required. 

The situation is not so simple with $W^{p,q}_1$, because the best one
can obtain is a representation of $D_k$ as a $2\times 2$ matrix of
scalar operators, after parametrizing the elements of $W^{p,q}_1$ with
pairs $(\xi,\eta)$ of forms in $W^{p,q}_0$: 
$$
 \begin{pmatrix}\xi'\\
   \eta'\end{pmatrix}=\begin{pmatrix}m_{11}(L,i\inv T)&m_{12}(L,i\inv
   T)\\m_{21}(L,i\inv T)&m_{22}(L,i\inv
   T)\end{pmatrix}\begin{pmatrix}\xi\\ \eta\end{pmatrix}\ . 
$$

A formal computation can be used on the core to produce
``eigenvalues'' $\la_\pm(L,i\inv T)$ and the splitting of $W^{p,q}_1$
as the sum of the two ``eigenspaces'' $W^{p,q,\pm}_1$. 

The final decomposition is in formula \eqref{eq-dec-k}.
\bigskip

\section{Differential forms and the Hodge Laplacian on $H_n$}\label{difforms}

The Heisenberg group $H_n$ is $\C^n\times\R$ with product
\begin{equation}\label{product}
(z,t)(z',t')=\Big(z+z',t+t'-\half\IM\lan z,z'\ran\Big)\ .
\end{equation}
 
On its Lie algebra, also identified with $\C^n\times\R$, we introduce
the standard Euclidean inner product, and we consider the
left-invariant Riemannian metric on $H_n$ induced by it. The complex
vector fields 
$$
Z_j=\sqrt 2\Big(\de_{z_j}-\frac i4\bar z_j\de_t\Big)\ ,\qquad \bar
Z_j=\sqrt 2\Big(\de_{\bar z_j}+\frac i4 z_j\de_t\Big)\ ,\qquad T=\de_t 
$$
(with $1\le j\le n$) form an orthonormal basis of the complexified
tangent space at each point, and the only nontrivial commutators
involving the basis elements are  
\begin{equation}\label{1.1}
[Z_j,\bar Z_j]=iT\ .
\end{equation}

The dual basis of complex 1-forms is
\begin{equation}\label{zeta-j-def}
\zeta_j=\frac1{\sqrt2}dz_j\ ,\qquad \bar\zeta_j=\frac1{\sqrt2}d\bar
z_j\ ,\qquad \theta=dt+\frac i4\sum_{j=1}^n (\bar z_jdz_j-z_jd\bar
z_j)\ . 
\end{equation}

The differential of a function $f$ is therefore
$$
df=\sum_{j=1}^n\big(Z_jf\zeta_j+\bar Z_jf\bar\zeta_j\big)+Tf\theta\ .
$$

This formula extends to forms, once we observe that $d\zeta_j=d\bar
\zeta_j=0$ and the differential of the contact form $\theta$ is the
symplectic form on $\C^n$, 
\begin{equation}\label{1.2}
d\theta=-i\sum_{j=1}^n\zeta_j\wedge\bar\zeta_j\ .
\end{equation}

A differential form $\om$ is {\it horizontal} if
$\theta\lrcorner\om=0$, i.e. if 
\begin{equation}\label{1.3}
\om=\sum_{I,I'}f_{I,I'}\zeta^I\wedge\bar\zeta^{I'}\ .
\end{equation}

Every form $\om$ decomposes as 
\begin{equation}\label{1.4}
\om=\om_1+\theta\wedge\om_2\ ,
\end{equation}
with $\om_1,\om_2$ horizontal.

A differential operator $D$ acting on scalar-valued functions is
extended to forms by  letting $D$ act separately on each scalar component
\eqref{1.3} of each horizontal component \eqref{1.4}.
Such operators will be called {\it scalar operators}.

The partial differentials $\de$, $\bar\de$, $d_H$  (resp. {\it
  holomorphic, antiholomorphic, horizontal} differential) of a form
$\om$ are defined as 
\begin{equation}\label{1.5}
\de\om=\sum_{j=1}^n \zeta_j\wedge Z_j\om\ ,\qquad
\bar\de\om=\sum_{j=1}^n \bar\zeta_j\wedge \bar Z_j\om\ ,\qquad
d_H\om=\de\om+\bar\de\om\ . 
\end{equation}

As in \cite{MPR}, Prop. 2.2, for $\om=f\zeta^I\wedge \bar\zeta^{I'}$,
\begin{equation}\label{de-om}
\de\om=\sum_{\ell,J}\eps^J_{\ell,I}(Z_\ell f)\zeta^J\wedge\bar\zeta^{I'}\ ,\qquad 
\bar\de\om=(-1)^{|I|}\sum_{\ell,J'}\eps^{J'}_{\ell,I'}(\bar Z_\ell f)\zeta^I\wedge\bar\zeta^{J'}\ ,
\end{equation}
where $\eps^J_{\ell,I}=0$ unless $\ell\not \in I$ and $\{\ell\} \cup I = J$, in which case
$$
\eps^J_{\ell,I}=\prod_{i\in I}\sgn(i-\ell)\ .
$$

Obviously, they act separately on each horizontal component
\eqref{1.4} of $\om$,  
 and the same is true for their adjoints
$\de^*,\bar\de^*, d_H^*$, where 
\begin{equation}\label{de^*-om}
\de^*\om=-\sum_{\ell,J}\eps^I_{\ell,J}(\bar Z_\ell f)\zeta^J\wedge\bar\zeta^{I'}\ ,\qquad 
\bar\de^*\om=(-1)^{|I|+1}\sum_{\ell,J'}\eps^{I'}_{\ell,J'} (Z_\ell f)\zeta^I\wedge\bar\zeta^{J'}\ .
\end{equation}

Moreover,
$\de^2=\bar\de^2={\de^*}^2={\mathop{\bar\de}^*}^2=0$.

Two operators that  will play a fundamental role in this paper are
\begin{equation}\label{edtheta}
e(d\theta)\om=d\theta\,\wedge\om
\qquad\text{and}\quad
i(d\theta)\om=e(d\theta)^*\om=d\theta\lrcorner \om\ .
\end{equation} 
Together with $\de$ and $\bar\de$, they satisfy the following
identities: 
\begin{equation}\label{1.6}
\aligned
&\de\bar\de+\bar\de\de=d_H^2=-Te(d\theta)\ ,\\
&\de^*\bar\de^*+\bar\de^*\de^*={d_H^*}^2=Ti(d\theta)\ ,
\endaligned
\end{equation}
and
\begin{equation}\label{1.6bis}
\de\bar\de^*=-\bar\de^*\de
\qquad\text{and}\quad
\de^*\bar\de=-\bar\de\de^* \ .
\end{equation}

Other formulas involving $\de,\bar\de,e(d\theta)$ and their adjoints
are 
\begin{equation}\label{1.7}
\aligned
&\big[i(d\theta),\de\big]
 = -i\bar\de^*\ ,\qquad & \big[i(d\theta),\bar\de\big]=i\de^*\ ,\\
& \big[\de^*,e(d\theta)\big]
= i\bar\de\ ,\qquad & \big[\bar\de^*,e(d\theta)\big]=-i\de \ ,
\endaligned
\end{equation}
\begin{equation}\label{1.7bis}
\big[i(d\theta),\de^*\big] =\big[i(d\theta),\bar\de^*\big]=0
= \big[e(d\theta),\de\big]=\big[e(d\theta),\bar\de\big]
\end{equation}
and
\begin{equation}\label{1.7tris}
 \big[i(d\theta),e(d\theta)\big]=(n-k)\id\ .
\end{equation}

For these formulas and the following in this section, we refer 
to \cite{MPR}.\footnote{Perhaps we should add a few more formulas,
  while moving them into a section, later in the paper.}

We define the {\it holomorphic, antiholomorphic and horizontal
  Laplacians} as 
\begin{equation}\label{1.8}
\aligned
\Box&=\de\de^*+\de^*\de\ ,\\
\Boxbar&=\bar\de\bar\de^*+\bar\de^*\bar\de\ ,\\
\Delta_H&=d_Hd_H^*+d_H^*d_H=\Box+\Boxbar\ .
\endaligned
\end{equation}

Each of these Laplacians acts componentwise.
Calling $(p,q)$-form a horizontal form of type
$$
\om=\sum_{|I|=p\,,\,|I'|=q}f_{I,I'}\zeta^I\wedge\bar\zeta^{I'}\ ,
$$
and introducing the {\it sublaplacian}
\begin{equation}\label{1.9}
L=-\sum_{j=1}^n (Z_j\bar Z_j+\bar Z_jZ_j)\ ,
\end{equation}
the operators $\Box,\Boxbar,\Delta_H$ coincide on $(p,q)$-forms with
the following  scalar operators: 
\begin{equation}\label{1.10}
\aligned
\Box&=\half L+i\Big(\frac n2-p\Big)T\ ,\\
\Boxbar&=\half L-i\Big(\frac n2-q\Big)T\ ,\\
\Delta_H&=L+i(q-p)T\ .
\endaligned
\end{equation}

To be more explicit, we shall occasionally denote the ``box''' operators by
$\Box_p$ and~$\Boxbar_q$.  Some commutation relations that we will use
are (see \cite{MPR})
\begin{equation}\label{commutations}
\aligned
&\Box\bar\de=\bar\de(\Box-iT) \, , \qquad\Box\bar\de^*=\bar\de^*(\Box+iT) \, ,
\\
&\Boxbar\de=\de(\Boxbar+iT)\, ,\qquad
\Boxbar\de^*=\de^*(\Boxbar-iT)\, .
\endaligned
\end{equation}

The full differential $d$ of a form $\om=\om_1+\theta\wedge\om_2$ and
its adjoint $d^*$ are represented, in terms of the pair
$(\om_1,\om_2)$, by the matrices 
\begin{equation}\label{1.11}
\bpm d_H&e(d\theta)\\ T&-d_H\epm\ ,\qquad 
d^*=\bpm d_H^*&-T\\ i(d\theta)&-d_H^*\epm\ ,
\end{equation}
\smallskip

and the {\it Hodge Laplacian} $\Delta=dd^*+d^*d$ by the matrix
\begin{equation}\label{1.12}
\aligned
\Delta&=\bpm\Delta_H-T^2+e(d\theta)i(d\theta)&\big[d_H^*,e(d\theta)\big]\\  \\
\big[i(d\theta),d_H\big]&\Delta_H-T^2+i(d\theta)e(d\theta)\epm\\ \\
&=\bpm\Delta_H-T^2+e(d\theta)i(d\theta)&i\bar\de-i\de\\  \\
i\de^*-i\bar\de^*&\Delta_H-T^2+i(d\theta)e(d\theta)\epm\ .
\endaligned
\end{equation}

\smallskip
When $\Delta$ acts on $k$-forms, it will be denoted by $\Delta_k$. In particular,
$$
\Delta_0=L-T^2\ .
$$

We denote by $\Lambda^k$ the $k$-th exterior product of the dual
$\h_n^*$ of the Lie algebra of $H_n$ (identified with the linear span
of $\zeta_1,\dots,\zeta_n,\bar\zeta_1,\dots,\bar\zeta_n,\theta$), by
$\Lambda^k_H$ the $k$-th exterior product of the horizontal
distribution (i.e. the linear span of the $\zeta_j,\bar\zeta_j$), and
by $\Lambda^{p,q}$ the space of elements of bidegree $(p,q)$ in
$\Lambda^k_H$. Symbols like $L^p\Lambda^k$, $\S\Lambda^{p,q}$ etc.,
denote the space of $L^p$-sections, $\S$-sections etc., of the
corresponding bundle over $H_n$. Clearly, $L^p\Lambda^k\cong
L^p\otimes\Lambda^k$ etc..

\bigskip

\setcounter{equation}{0}
\section{Bargmann representations and sections of homogeneous bundles}\label{bargmann}

The $L^2$-Fourier analysis on the Heisenberg group involves the family of infinite dimensional
irreducible unitary representations $\{\pi_\la\}_{\la\ne0}$ such that
$\pi_\la(0,t)=e^{i\la t}\id$. These representations are most
conveniently realized for our purposes in
a modified version of
 the Bargmann form \cite{Fo}.

Let
$\F=\F(\C^n)$ be the space of entire
functions $F$ on $\C^n$ such that 
$$
\|F\|_\F^2=\int_{\C^n}|F(w)|^2e^{-\half|w|^2}\,dw<\infty\ .
$$

 The family of Bargmann representations
 $\pi_\la$ on $\F$ is defined, for $\la\ne0$, as follows:
 \begin{enumerate}
 \item[\rm(i)] for $\la=1$,
 \begin{equation}\label{pi1}
 \big(\pi_1(z,t)F\big)(w)=e^{i t}e^{-\half\lan w,z\ran-\frac14|z|^2}F(w+z)\ .
 \end{equation}
 \item[\rm(ii)] For $\la>0$,
 \begin{equation}\label{la>0}
 \pi_\la(z,t)=\pi_1(\la^\half z,\la t) \ ;
\end{equation}
\item[\rm(iii)] for $\la<0$,
\begin{equation}\label{la<0}
 \pi_{\la}(z,t)=\pi_{-\la}(\bar z,-t)\ .
 \end{equation}
 \end{enumerate}

 The unitary group $U(n)$ acts on $H_n$ through the automorphisms
 $$
 (z,t)\longmapsto (z,t)^g=(gz,t)\ ,\qquad \big(g\in U(n)\big)\ ,
 $$
 and on $L^2(H_n)$ through the representation
 $$
\big( \al(g)f\big)(z,t)=f\big((z,t)^{g\inv}\big)\ .
 $$

We also consider the pair of contragradient representations $U,\bar
 U$ of $U(n)$ on $\F$, given by  

\begin{equation}\label{2.1}
U_gF=F\circ g\inv\ ,\qquad \bar U_g=U_{\bar g}\ .
\end{equation}

Then 
\begin{equation}\label{U-intertw}
 \aligned
 \pi_\la(gz,t)&=U_g\,\pi_\la(z,t)\,U_{g\inv}\ ,&\text{for }\la>0\ ,\\
 \pi_\la(gz,t)&=\bar U_g\,\pi_\la(z,t)\,\bar U_{g\inv}\ ,&\text{for }\la<0\ .
 \endaligned
\end{equation}

The representation $U$ in \eqref{2.1} splits into irreducibles
according to the decomposition of $\F$  

\begin{equation}\label{2.2}
\F=\sum_{j\ge0}\P_j\ ,
\end{equation}
where $\P_j$ denotes the space of homogeneous polynomials of degree
$j$. 

We denote by $P_j$  the orthogonal projection of $\F$ on
$\P_j$, and by
$\F^\infty$
 the space of functions $F\in\F$ such that 
\begin{equation}\label{garding}
\|P_jF\|_\F=o(j^{-N})\ ,\qquad\forall\,N\in\N\ .
\end{equation}

 Then $\F^\infty$ is the
space of $C^\infty$-vectors for all representations $\pi_\la$.

 The differential of $\pi_\la$ is given by\footnote{ Even though $d\pi_\la(D)$ is the more standard notation, we prefer to
reduce the number of $d$'s around.} $\pi_\la(T)=i\la$ and
\begin{equation}\label{pi(Z)}
\pi_\la(Z_\ell)=\begin{cases} \sqrt{2\la}\,\de_{w_\ell}&\text{ if }\la>0\\ -\sqrt{\frac{|\la|}2}\,w_\ell&\text{ if }\la<0\ ;
\end{cases}
\qquad 
\pi_\la(\bar Z_\ell)=\begin{cases}-\sqrt{\frac\la2}\,w_\ell&\text{ if }\la>0\\  \sqrt{2|\la|}\,\de_{w_\ell}&\text{ if }\la<0\ .
\end{cases}
\end{equation}

We adopt the following definition of $\pi_\la(f)$: 

\begin{equation}\label{2.3}
\pi_\la(f)=\int_{H_n}f(x)\pi_\la(x)\inv\,dx\in \L(\F,\F)\ .
\end{equation}

Notice that $\pi_\la(f*g)=\pi_\la(g)\pi_\la(f)$, but this disadvantage
is compensated by a simpler formalism when dealing with forms or more
general vector-valued functions. 

The Plancherel formula for $f\in L^2$ is 
$$
\|f\|_2^2=c_n\int_{-\infty}^{+\infty} 
\|\pi_\la(f)\|^2_{HS}\,|\la|^n\,d\la=c_n\int_{-\infty}^{+\infty} 
\sum_{j,j'}\|P_j\pi_\la(f)P_{j'}\|^2_{HS}\,|\la|^n\,d\la\ .
$$

\vskip.2cm
 
Let $V$ be a finite dimensional Hilbert space.
Defining $\pi_\la(f)$ for $V$-valued functions $f$ by \eqref{2.3}, we have 
$$
\pi_\la(f)\in \L(\F,\F)\otimes V\cong \L(\F,\F\otimes V)\ .
$$

Suppose now that $V$ is the representation space of a unitary
representation $\rho$ of $U(n)$, and consider the two representations
$U\otimes\rho$, $\bar U\otimes\rho$ of $U(n)$ on $\F\otimes
V$.  Denote by $\Sigma^+=\Sigma^{\rho,+}$ (resp. $\Sigma^-=\Sigma^{\rho,-}$) the set of irreducible
representations $\sigma\in \widehat{U(n)}$ contained in $U\otimes\rho$
(resp. in $\bar U\otimes\rho$), and let 
\begin{equation}\label{2.4}
\F\otimes V=\bigoplus_{\sigma\in\Sigma^\pm} \E^\pm_\sigma
\end{equation}
be the corresponding orthogonal decompositions into $U(n)$-types. When
$V=\C$, the decomposition \eqref{2.4}   
  reduces to \eqref{2.2}. To indicate the dependence on $\rho,$ we shall sometime also  write
$\E^\pm_\sigma=\E^{\rho,\pm}_\sigma$.

\begin{lemma}\label{s2.1}
Each $\E_\sigma^\pm$ is finite dimensional and decomposes into
$U(n)$-invariant subspaces 
$$
\E_\sigma^\pm=\bigoplus_j \E_\sigma^\pm\cap \bigl( \P_j \otimes V\bigr) \ .
$$
In particular $\E_\sigma^\pm\subset\F\otimes V$. More precisely, $\E_\sigma^\pm\subset \F^\infty\otimes
V$, where $\F^\infty$ is defined in \eqref{garding}. 
\end{lemma}

\proof We only discuss the case of $U\otimes\rho$. For every $j$,
$\P_j\otimes V$ is an invariant subspace. Therefore, for $\sigma\in
\widehat{U(n)}$, $\E^+_\sigma=\sum_j\E^+_\sigma\cap (\P_j\otimes
V)$. Let $\chi_j$, $\chi_\rho$, $\chi_\sigma$ be the characters of
$U_{|_{\P_j}}, \rho,\sigma$ respectively. The multiplicity of $\sigma$
in $\P_j\otimes V$ is then given by 
$$
\int_{U(n)}
\chi_j(g)\chi_\rho(g)\overline{\chi_\sigma(g)}\,dg
=\lan\chi_j,\overline{\chi_\rho}\chi_\sigma\ran\
, 
$$
which is the multiplicity of $U_{|_{\P_j}}$ in $\bar\rho\otimes\sigma$
(with $\bar\rho$ denoting the contragredient of $\rho$). Since this
representation is finite dimensional, the multiplicity can be positive
only for a finite number of $j$. It follows that $\E^+_\sigma$ has
finite dimension. 

Since $\E^\pm_\sigma$ consists of $V$-valued polynomials, it is
obviously contained in $G\otimes V$. 
\endproof

The decomposition of $\F\otimes V$ given above leads to the
following form of the Plancherel formula for $L^2V$, with
$P^\pm_\sigma$ denoting the orthogonal projection of
$\F\otimes V$ onto $\E^\pm_\sigma$: 
\begin{equation}\label{2.5}
\aligned
\|f\|_2^2
&=c_n\int_{-\infty}^{+\infty}\sum_{j\in\N\,,\,\sigma\in\Sigma^{\sgn\la}}\|P^{\sgn\la}_\sigma\pi_\la(f)P_j\|^2_{HS}\,
|\la|^n\,d\la\\
&=c_n\int_{-\infty}^{+\infty}\sum_{\sigma\in\Sigma^{\sgn\la}}\|P^{\sgn\la}_\sigma\pi_\la(f)\|^2_{HS}\,
|\la|^n\,d\la\ .
\endaligned
\end{equation}

\vskip.2cm

Let $\rho'$ be another unitary representation of $U(n)$ on a finite
dimensional Hilbert space $V'$. The convolution 
$$
f*K(x)=\int_{H_n}K(y\inv x)f(y)\,dy
$$
of integrable functions $f$ with values in $V$ and $K$ with values in
$\L(V,V')$ produces a function taking values in $V'$. In the
representations $\pi_\la$, $\la\ne0$,  
$$
\pi_\la(K)\in \L(\F,\F)\otimes \L(V,V')\cong
\L(\F\otimes V,\F\otimes V')\ , 
$$ 
and 
$$
\pi_\la(f*K)=\pi_\la(K)\pi_\la(f)\in \L(\F,\F\otimes V')\ .
$$

Let $\tilde\rho$ (resp. $\tilde\rho'$) be the representation
$\al\otimes\rho$ on $L^2V$ (resp. $\al\otimes\rho'$ on $L^2V'$) of
$U(n)$ and suppose that convolution by $K$ is an equivariant operator, 
i.e. 
\begin{equation}\label{2.6}
\tilde\rho'(g)(f*K)=\big(\tilde\rho(g)f\big)*K
\end{equation}
for $g\in U(n)$
and $f\in\S V$.
Since for $f\in\S V$ and $\xi\in\F$, with $\la>0$,
\begin{eqnarray*}
\pi_\la\Big(\tilde{\rho}'(g)(f*K)\Big)\xi
&=&\iint\rho'(g) K(y^{-1}x) f(y) U_g\pi_\la(x^{-1}) U_{g^{-1}}\xi\, dydx,\\
\pi_\la\Big((\tilde{\rho}(g)f)*K)\Big)\xi
&=&\iint K(y^{-1}x)\rho(g) f(y^{g^{-1}}) \pi_\la(x^{-1}) \xi\, dydx,
\end{eqnarray*}
 by letting $f$ tend weakly to $\delta_0\otimes v,$  with $v\in V,$ we
 see that  \eqref{2.6} implies  
 \begin{eqnarray*}
\int \rho'(g) K(x)v\, U_g\pi_\la(x^{-1}) U_{g^{-1}}\xi\,dx
=\int K(x)\rho(g)v\,  \pi_\la(x^{-1}) \xi\, dx.
\end{eqnarray*}
Replacing $\xi$ by $U_g\xi,$ we obtain
\begin{equation*}
U_g\otimes \rho'(g)\Big (\pi_\la(K)(\xi\otimes v)\Big)=\pi_\la(K) (U_g\,\xi\otimes \rho(g) v)
\end{equation*}
for every $\xi\in \F, v\in V.$ 
A similar formula holds for $\la<0$ with $\bar U$ in place of $U$.
Thus   \eqref{2.6} implies the
following identities, for $K$ defining an equivariant convolution operator:
\begin{equation}\label{2.7}
\aligned
(U\otimes\rho')(g)\pi_\la(K)
&=\pi_\la(K)(U\otimes\rho)(g)\ ,\qquad \la>0\,\\
(\bar U\otimes\rho')(g)\pi_\la(K)
&=\pi_\la(K)(\bar U\otimes\rho)(g)\ ,\qquad \la<0\,
\endaligned
\end{equation}
for $g\in U(n)$, i.e. $\pi_\la(K)$ intertwines $U\otimes\rho$ and
$U\otimes\rho'$, or $\bar U\otimes\rho$ and $\bar U\otimes\rho'$
depending on the sign of $\la$. The following is an immediate
consequence. 

\begin{lemma}\label{s2.2}
 Assume that convolution by $K\in L^1\otimes\L(V,V')$ is an
 equivariant operator. Then, setting $\Sigma^{\rho,\rho',\sgn\la}=\Sigma^{\rho,\sgn\la}\cap\Sigma^{\rho',\sgn\la}$,
$$
\pi_\la(K)=\bigoplus_{\sigma\in\Sigma^{\rho,\rho',\sgn\la}}\pi_{\la,\sigma}(K)\ ,
$$
with $\pi_{\la,\sigma}(K):\E^{\rho,\sgn\la}_\sigma\rightarrow
\E^{\rho',\sgn\la}_\sigma$. 
\end{lemma}

By a variant of Schwartz's Kernel Theorem, the convolution operators
$D$ with kernels $K\in\S'(H_n)\otimes \L(V,V')$ are characterized as
the continuous operators from $\S(H_n)\otimes V$ to $\S'(H_n)\otimes
V'$ that commute with left translations on $H_n$. 
Lemma \ref{s2.2}
applies to operators of this kind, provided that the Fourier transform
$\pi_\la(K)$ is well defined for $\la\ne0$. This is surely the case if 
$K$ has compact support,  and in particular for a left-invariant differential operator
$Df=f*(D\del_0)$. 
 We then have 
$$
\pi_\la(Df)=\pi_\la(D\del_0)\pi_\la (f)=\pi_\la(D)\pi_\la (f)\ .
$$

We apply these remarks to the differentials and Laplacians introduced
in Section \ref{difforms}.  

With $\rho_k$ denoting the representation of $U(n)$ on $\Lambda^k$
induced from its action on $H_n$ by automorphisms, and, as before let
$\tilde\rho_k=\al\otimes \rho_k$ be the tensor product acting on
$L^2\Lambda^k$. 
Then $d,d^*,\Delta_k$ are equivariant operators. 
The same applies to $\de,\bar\de,d_H$ etc. on the appropriate
$L^2$-subbundles.

Notice that $\Box,\Boxbar$ and $\Delta_H$ have the
special property of acting scalarly on $(p,q)$-forms, by
\eqref{1.10}. Since the sublaplacian $L$ has the property that
$\pi_\la(L)$ acts as a scalar multiple of the identity (namely, as
$|\la|(2m+n)\id$) on $\P_m\subset\F$, the same is true for the
image of $\Box,\Boxbar,\Delta_H$ under $\pi_\la$.

\bigskip

\setcounter{equation}{0}
\section{Cores, domains and self-adjoint extensions}\label{cores}

For $0<\del<R$ and $N\in\N$, denote by $\S_{\del,R,N}(H_n)$ the
space of functions $f$ satisfying the following
properties: 

\smallskip
\bee
\item[(i)] $f\in\S(H_n)$;
\item[(ii)] $\pi_\la(f)=0$ for $|\la|\le\del$ and $|\la|\ge R$;
\item[(iii)] for $\del<|\la|<R$,
$P_j\pi_\la(f)=0$ for $j>N$.
\ee

\smallskip

We set $\S_0=\bigcup_{\del,R,N}\S_{\del,R,N}$. 

\begin{lemma}\label{S_0} 
$\S_0$ is invariant under left translations, and dense in $L^2$. 
\end{lemma}

\begin{proof}
The first statement follows from the identity $\pi_\la(L_{(z,t)}f)=\pi_\la(f)\pi_\la(z,t)\inv$, where $L_{(z,t)}f$ is the  left translate of $f\in \S_0$ by $(z,t)^{-1}.$ 
Take now $f\in\S$. For $\del,R>0$, fix a $C^\infty$-function $u_{\del,R}(\la)$ on $\R$, with values in $[0,1]$, supported where $\del\le|\la|\le R$ and equal to 1 where $2\del\le|\la|\le R/2$. 

Given  $\eps>0$, by Plancherel's formula it is possible to find $\del,R>0$ and $N\in\N$ such that the $L^2$- function $g$ such that $\pi_\la(g)=u_{\del,R}(\la)\sum_{j\le N}P_j\pi_\la(f)$ approximates $f$ in $L^2$ by less than $\eps$.

We claim that $g$ is in $\S$, hence in $\S_0$, and this will conclude the proof, by the density of $\S$ in~$L^2$.

By definition, $g=f*h$, where $h$ is the function with $\pi_\la(h)=u_{\del,R}(\la)\sum_{j\le N}P_j$. By explicit computation of the matrix entries of the representations of $H_n$ \cite{Thanga}, the Fourier transform $h(z,\hat \la)$ of $h$ in the $t$-variable equals
$$
h(z,\hat \la)=u_{\del,R}(\la)\sum_{j\le N}\psi_j(|\la|^\half z)\ ,
$$
where the $\psi_j$ are Schwartz functions on $\C^n$. Hence $h\in\S(H_n)$ and so is $g$.
\end{proof}

We regard $\S_0$ as the inductive limit of the spaces $\S_{\del,R,N}$, each with the topology induced from~$\S$.

Obviously, $\S_0V=\S_0\otimes V$ is contained in $\S V$ and dense in
$L^2V$. Assume, as in Section \ref{bargmann}, that $V$ is a finite
dimensional Hilbert space on which $U(n)$ acts unitarily by the
representation~$\rho$.  
Taking into account the action of $U(n)$, one can then  introduce a
different chain of subspaces filling up $\S_0V$. Given $0<\del<R$,
$N\in\N$ and finite subsets
$X^\pm$ of $\Sigma^\pm$, define
$\S_{\del,R,X^\pm}V$ as the space of functions $f$ such that

\smallskip
\bee
\item[(i')] 
$f\in\S(H_n)\otimes V$;
\item[(ii')] $\pi_\la(f)=0$ for $|\la|\le\del$ and $|\la|\ge R$;
\item[(iii')] for $\del<|\la|<R$, $P^{\sgn\la}_\sigma\pi_\la(f)=0$ for $\sigma\not\in X^{\sgn\la}$;
\ee
\smallskip

It follows from Lemma \ref{s2.1} that finite unions of the
$\S_{\del,R,X^\pm}V$
 exhaust finite unions of the
$\S_{\del,R,N}V$ and viceversa. 
\medskip

 In order to develop the $L^2$-analysis of differentials
and Laplacians, we establish some general facts about densely defined operators from $L^2V$ to $L^2V'$, wih $(V,\rho)$, $(V',\rho')$  finite-dimensional representation spaces  of~$U(n)$. Precisely, we consider operators whose initial domain is $\S_0V$ and which map $\S_0V$ into $\S_0V'$, continuously with respect to the Schwartz topologies.   Most of the operators to be considered in this paper will belong to this class.

\begin{lemma}\label{s3.1}
\quad
\begin{enumerate}
\item[\rm(i)] Let $(V,\rho)$, $(V',\rho')$ be finite-dimensional representation spaces  of~$U(n)$, and let
$$
B:\S_0V\longrightarrow \S_0V'\ ,
$$ 
be a left-invariant linear operator, $U(n)$-equivariant and continuous with respect to the $\S_0$-topologies. Then there exists a family of  linear
operators $B_{\la,\sigma}:\E^{\rho,\sgn\la}_\sigma\longrightarrow \E^{\rho',\sgn\la}_\sigma$, depending smoothly on $\la\ne0$, such that
\begin{equation}\label{pi(B)}
\pi_\la(Bf)=\bigoplus_{\sigma\in\Sigma^{\rho,\rho',\sgn\la}} B_{\la,\sigma}
P^{\sgn\la}_\sigma\pi_{\la}(f)\ ,
\end{equation}
\item[\rm(ii)]Conversely, given any family of linear
  operators $B_{\la,\sigma}:\E^{\rho,\sgn\la}_\sigma\longrightarrow \E^{\rho',\sgn\la}_\sigma$, depending smoothly on $\la\ne0$, there is a unique left-invariant operator $B:\S_0V\longrightarrow \S_0V'$, $U(n)$-equivariant and continuous with respect to the $\S_0$-topologies, such that \eqref{pi(B)} holds for every $f\in\S_0V$. 
We  set 
$$
\pi_{\la,\sigma}(B)=B_{\la,\sigma}\ ,\qquad  B_\la=\sum_{\sigma\in\Sigma^{\rho,\rho',\sgn\la}}B_{\la,\sigma}P^{\sgn\la}_\sigma\ ,\qquad\pi_{\la}(B)=B_\la\ .
$$
\item[\rm(iii)] The closure of $B$ as an operator from $L^2V$ to $L^2V'$ has
domain  $\dom(B)$ consisting of those $f\in L^2V$ such that
\begin{equation}\label{domB}
\int_{-\infty}^{+\infty}\sum_\sigma\|B_{\la,\sigma}P^{\sgn\la}_\sigma\pi_\la(f)\|^2_{HS}\,|\la|^n\,d\la<\infty\ .
\end{equation}
\item[\rm(iv)] If $(V,\rho)=(V',\rho')$ and $B$ is symmetric (equivalently,  $B_{\la,\sigma}$ is symmetric for every $\la$,~$\sigma$), then $B$ is essentially
self-adjoint. 
\item[\rm(v)] If $B$ is symmetric and $m$ is a Borel function on the real line such that $m(B_{\la,\sigma})$ is well defined for every $\la$, $\sigma$, the domain of $m(B)$ is the space
of those $f\in L^2 V$ such that  
$$
\int_{-\infty}^{+\infty}\sum_\sigma
\big\|m(B_{\la,\sigma})
P^{\sgn\la}_{\sigma}\pi_\la(f)\big\|^2_{HS}\,|\la|^n\,d\la<\infty\ 
. 
$$

Moreover, the space $\S_0V\cap \dom\big(m(B)\big)$ is a core for $m(B)$ and the identity
$$
\lan
m(B)f,g\ran=c_n\int_{-\infty}^{+\infty}\sum_\sigma \tr\big(m(B_{\la,\sigma})
P^{\sgn\la}_{\sigma}\pi_\la(f)\pi_\la(g)^*\big)\,|\la|^n\,d\la 
$$
holds for $f\in\dom\big(m(B)\big)$ and $g\in L^2V$.
\end{enumerate}
\end{lemma}

\proof To prove (i), let $\{\Phi_\ell\}_{\ell\in\N}$ be an enumeration of the orthonormal basis of monomials in $\F$.
Define linear  operators $E_{\ell,\ell'}:\F\to\F$ by setting
$E_{\ell,\ell'}F=\lan F,\Phi_{\ell'}\ran_\F\Phi_\ell$.
Let also $\{e_i\}$ and $\{e'_j\}$  be 
 (finite) bases of $V$ and $V'$ respectively.

Given two compact intervals 
$[a,b]$ and $[a',b']$ such that $[a,b]\subset [a',b']^0,$
  with $a'>0$ (for intervals contained in $\R^-$ the proof is similar) and $\ell\in\N$, there exists $g_\ell\in\S_0$ such that $\pi_\la(g_\ell)=E_{\ell,\ell}$ for $\la\in [a,b]$ and $\pi_\la(g_\ell)=0$ for $\la\not\in[a',b']$ (cf. the proof of Lemma \ref{S_0} and \cite{Thanga}).

Then $B(g_\ell\otimes e_i)\in\S_0V'$ and
\begin{equation}\label{c's}
\pi_\la\big(B(g_\ell\otimes e_i)\big)=\sum_j\Big(\sum_{h,k}c^{\ell,i}_{h,k,j}(\la)E_{h,k}\Big)\otimes e'_j\ ,
\end{equation}
where the sum 
in $h$  ranges over a fixed finite set of indices independent of $\la$. The coefficients $c^{\ell,i}_{\ell,\ell',j}$ are smooth in $\la$ and, since $B$ is left-invariant,  supported in $[a',b']$.

Take now  $f=\sum_i f_i\otimes e_i\in\S_0V$, with $\pi_\la(f)=0$ for $\la\not\in[a,b]$. Then
$$
f=\sum_{i,\ell}(f_i*g_\ell*g_\ell)\otimes e_i\ ,
$$
where the sum is finite. Hence, by the continuity assumption on $B$,
$$
Bf=\sum_{i,\ell}(f_i*g_\ell)*B(g_\ell\otimes e_i)\ .
$$

Introducing the notation 
$$
\hat f_i(\la,\ell,\ell')=\lan \pi_\la(f_i)\Phi_{\ell'},\Phi_\ell\ran_\F\ ,
$$
we have
$$
\pi_\la(f_i*g_\ell)=E_{\ell,\ell}\,\pi_\la(f_i)=\sum_{\ell'}\hat f_i(\la,\ell,\ell')E_{\ell,\ell'}\ .
$$

 The composition $\pi_\la\big(B(g_\ell\otimes e_i)\big)\pi_\la(f_i*g_\ell)$ is well defined, because the second factor has the one-dimensional range  $\C\Phi_\ell$. Therefore the index $k$ in \eqref{c's} can only assume the value $\ell$, and
$$
\begin{aligned}
\pi_\la(Bf)&=\sum_j\Big(\sum_{i,\ell}\sum_hc^{\ell,i}_{h,\ell,j}(\la)E_{h,\ell}\sum_{\ell'}\hat f_i(\la,\ell,\ell')E_{\ell,\ell'}\Big)\otimes e'_j\\ 
&=\sum_j\bigg(\sum_i\sum_{h,\ell'}\Big(\sum_\ell c^{\ell,i}_{h,\ell,j}(\la)\hat f_i(\la,\ell,\ell')\Big)E_{h,\ell'}\bigg)\otimes e'_j\ .
\end{aligned}
$$

The infinite matrix $C_{i,j}(\la)=\big(c^{\ell,i}_{h,\ell,j}(\la)\big)_{h,\ell}$ has only a finite number of nonzero 
entries,
hence it defines  a linear operator $B_{i,j}^\la$ from the linear span of the $\Phi_\ell$ (i.e. the space of polynomials inside $\F$) into itself, by setting
$$
B_{i,j}^\la\Phi_\ell=\sum_h c^{\ell,i}_{h,\ell,j}(\la)\Phi_h.
$$
Notice that $B_{i,j}^\la E_{\ell,\ell'}=\sum_h c^{\ell,i}_{h,\ell,j}(\la)E_{h,\ell'}.$
 Then, setting $E'_{j,i}=\lan \cdot,e_i\ran_Ve'_j \in\L(V,V')$, one easily verifies that 
\begin{equation}\label{B_la}
B_\la =\sum_{i,j}B^\la_{i,j}\otimes E'_{j,i}\ ,
\end{equation}
maps $V$-valued polynomials into $V'$-valued polynomials and
\begin{equation}\label{Bf}
\pi_\la(Bf)=B_\la\pi_\la(f)\ ,
\end{equation}
for every $f\in\S_0V$ with  $\pi_\la(f)=0$ for $\la\not\in[a,b]$.

It is now easy to prove that for $\la\in [a,b],$ $B_\la$ is uniquely defined by the identity \eqref{Bf},  which shows that it does not depend on the choice of the functions $g_\ell$, and that if we repeat the same argument starting with a larger interval $[a^\#,b^\#]\supset[a,b]$ contained in $\R^+$, the new operators $B^\#_\la$ coincide with $B_\la$ for $\la\in[a,b]$. Covering the positive half-line by compact intervals of this type
 and repeating the same argument on the negative half-line, we find a unique map $\la\longmapsto B_\la$ defined for $\la\ne0$ and for which \eqref{Bf} holds for every $f\in\S_0V$. 

Since $B$ is $U(n)$-equivariant, a repetition of the proof of Lemma \ref{s2.2} shows that $B_\la$ maps $\E^{\rho,\sgn\la}_\sigma$ into $\E^{\rho',\sgn\la}_\sigma$ for every $\sigma\in\Sigma^{\rho,\rho',\sgn\la}$. It is obvious from the smoothness of the coefficients $c^{\ell,i}_{h,\ell,j}$ that the restricted operators $B_{\la,\sigma}$ depend smoothly on $\la$.

The proof of (ii) is quite obvious.

To prove (iii), 
denote by $\tilde B$ be the operator on $\dom (B)$, defined in \eqref{domB}, such that
$\pi_\la(\tilde Bf)=B_\la\pi_\la(f)$. It is easy to verify that
$\tilde B$ is closed and that $\S_0 V$ is dense in $\dom (B)$ in the
graph norm of $\tilde B$. Since $\tilde B$ coincides with $B$ on $\S_0
V\subset\S V$, $\tilde B$ is the closure of  $B$.

To prove (iv), assume that $B$ is symmetric. Then each operator $B_{\la,\sigma}$
is self-adjoint, hence so is
$\pi_\la(B)$. 
If $B'$ is the adjoint of $B$, taking $g$  in the domain of $B'$ and
$f\in\S_0 V$, we have 
$$
\aligned
\lan B'g,f\ran 
&=\lan g,Bf\ran\\
&=c_n\int_{-\infty}^{+\infty}\sum_\sigma\tr
\big(\pi_\la(f)^*\pi_\la(B)P^{\sgn\la}_{\sigma}\pi_\la(g)\big)\,|\la|^n\,d\la\\ 
&=c_n\int_{-\infty}^{+\infty}\sum_\sigma\tr
\big(\pi_\la(f)^*B_{\la,\sigma} P^{\sgn\la}_{\sigma}\pi_\la(g)\big)\,|\la|^n\,d\la\
. 
\endaligned
$$

By the arbitrariness of $\pi_\la(f)$ subject to conditions (i')-(iii'),
we conclude that  
$$
\int_{-\infty}^{+\infty}\sum_\sigma\|B_{\la,\sigma} P^{\sgn\la}_{\sigma}\pi_\la(g)\|^2_{HS}\,
|\la|^n\,d\la<\infty\
, 
$$
i.e. $g\in\dom (B)$, and that $\pi_\la(B'g)=\pi_\la(B)\pi_\la(g)$, 
i.e. $B'g=\tilde Bg$.

Finally, (v) is proved in a similar way.
\endproof

Consistently with the identity $\pi_\la\big(m(B)\big)=m\big(\pi_\la(B)\big)$, we write $\pi_{\la,\sigma}\big(m(B)\big)$ for $m\big(\pi_{\la,\sigma}(B)\big)$.
\vskip.3cm

\begin{remark}\label{--}
{\rm
As a typical instance of operations that will be done in the sequel,
consider an expression like $d\Delta_k\inv$. As soon as we find out
that the finite-dimensional operators $\pi_{\la,\sigma}(\Delta_k)$ are
invertible and depend smoothly on $\la$ (see the next Section), an
operator $\Psi$ satisfying the identity $\Psi\Delta_k=d$
 is automatically defined on $\S_0\Lambda^k$ with
values in $\S_0\Lambda^{k+1}$ by imposing that  
$$
\pi_\la(\Psi \om)=\sum_\sigma\pi_{\la,\sigma}(d)
\pi_{\la,\sigma}(\Delta_k)\inv P^\la_\sigma\pi_\la(\om)\ .
$$

Its closure $\bar\Psi$ is defined on the space
consists of the $\om\in L^2\Lambda^k$ such that
$$
\int_{-\infty}^{+\infty}\sum_\sigma
\|\pi_{\la,\sigma}(d)\pi_{\la,\sigma}(\Delta_k)^{-1}
P^\la_\sigma\pi_\la(\omega)\|^2_{HS}
\,|\la|^n\,d\la<\infty\
. 
$$

Notice that formal identities, like
\begin{itemize}
\item[(i)] $d\Delta_k\inv=(d\Delta_k^{-\half})\Delta_k^{-\half}$;
\smallskip
\item[(ii)]$\pi_\la(d\Delta_k\inv)=\pi_\la(d)\pi_\la(\Delta_k\inv)$;
\smallskip
\item[(iii)]$\pi_{\la,\sigma}(d\Delta_k\inv)=\pi_{\la,\sigma}(d)\pi_{\la,\sigma}
(\Delta_k\inv)$;
\end{itemize}
are fully justified on $\S_0\Lambda^k$. 
}
\end{remark}

In many instances we will make use of homogeneity properties of operators $B$ as those considered in Lemma \ref{s3.1}. As before, we assume that
$(V,\rho), (V',\rho')$ are  finite 
 dimensional  representations of $U(n)$.

 We assume that the multiplicative group $\R_+$ acts on $V$ by means
 of the linear representation $\gamma:\R_+\to \L(V)$ and on  $V'$ by
 means of the linear representation $\gamma':\R_+\to \L(V')$ 
 such that the operators 
 $\gamma (r)$ and $\gamma' (r)$ are self-adjoint,
and in such a
 way that each of these actions commutes with the corresponding action
 of $U(n)$ (on $V$ given by $\rho$ and on $V'$ given by $\rho'$). 
 
We
 also denote by $\delta_r$ the dilating automorphism of $H_n$ defined by  
 $$
 \del_r(z,t):=(r^{1/2}z,rt),\qquad r>0\ ,
 $$
 and let $\R_+$ act  on functions on $H_n$ by the representation 
 $$
 \beta(r)f:=f\circ \del_{r^{-1}}\ .
 $$
 
 Then $B$ is said to be {\it homogeneous of degree} $a$ if 
 \begin{equation}\label{homog}
B\circ(\beta\otimes\ga)(r)=r^{-a}\, (\beta\otimes\ga')(r)\circ B    
 \qquad  \mbox{on } \S_0 V, \ \mbox{for every  } r>0\ .               
\end{equation}

We shall repeatedly use the following lemma, which   applies in particular to operators such as $\de,\bar\de,d_H$ etc.

\begin{lemma}\label{density-lemma}
 Let $(V,\rho), (V',\rho')$ be finite 
 dimensional unitary representations of $U(n)$ and let
 $\beta,\gamma,\gamma'$ be as above.  
 
  If $B$ is a $U(n)$- equivariant, left-invariant  
operator as in Lemma \ref{s3.1} (i),  homogeneous in the sense of
\eqref{homog} for some $a\in\R$,  then  
$$
\overline {\range B}\cap \S_0V'=B(\S_0 V).
$$
\end{lemma}

  \proof 
We just have to verify that
 $$
\overline {\range B}\cap \S_0V'\subseteq B(\S_0 V)\ ,
$$
the other implication being contained in the assumptions.

The homogeneity of $B$ implies that
\begin{equation}\label{3.4}
\pi_\la(B)=|\la|^a \,(I\otimes \ga'(|\la|^{-1}))\, \pi_{\sgn\la}(B)\,(I\otimes \ga(|\la|))\ .
\end{equation}

Assume in fact that $f\in \S_0 V$. For $\la\ne0$, $\pi_\la(f)\in \L(\F,\F\otimes V)$ and, by \eqref{la>0} and \eqref{la<0}, 
$$
(I\otimes \ga(|\la|))\, \pi_\la(f)=(I\otimes
\ga(|\la|))\,\pi_{\sgn\la}\big(\beta(\la)f\big)=\pi_{\sgn\la}\big((\beta\otimes
\ga)(|\la|)f\big). 
$$

Similarly,
$$
(I\otimes \ga'(|\la|))\, \pi_\la(Bf)=\pi_{\sgn\la}((\beta\otimes \ga')(|\la|)(Bf)),
$$
and thus, by the homogeneity \eqref{homog} of $B,$
$$
(I\otimes \ga'(|\la|))\, \pi_\la(Bf)=|\la|^a\, \pi_{\sgn\la}(B\,((\beta\otimes \ga)(|\la|)f)\,)
=|\la|^a \,\pi_{\sgn\la}(B)(I\otimes \ga(|\la|))\, \pi_\la(f).
$$

This yields \eqref{3.4}.

Since $\ga$ and $\ga'$ commute with  $\rho$ and $\rho'$, $I\otimes \ga$ respects the decomposition 
$$
\F\otimes V=\bigoplus_{\sigma\in\Sigma^{\rho,\pm}}
\E^\pm_\sigma\ , 
$$
in \eqref{2.4}, we  have
\begin{equation}\label{3.4'}
B_{\la,\sigma}=|\la|^a \,\big(I\otimes \ga'(|\la|^{-1})\big)\, B_{\sgn\la,\sigma}\,\big(I\otimes \ga(|\la|)\big)\ ,
\end{equation}
where 
$B_{\la,\sigma}=\pi_{\la,\sigma}(B):\E^{\rho,\sgn\la}_\sigma\longrightarrow \E^{\rho',\sgn\la}_\sigma$ 
is the operator defined in \eqref{pi(B)}.

 We restrict now our attention to  $\la=\pm1$ and write, for simplicity, $B^\pm_\sigma$ instead of $B_{\pm1,\sigma}$.
  
  Since domain and codomain are finite dimensional, we have an inverse $Q^\pm_\sigma:\range {B^\pm_\sigma}\longrightarrow (\ker B^\pm_\sigma)^\perp$ of ${B^\pm_\sigma}_{|_{(\ker B^\pm_\sigma)^\perp}}$. Denote by $\tilde Q^\pm_\sigma$ the extension of $Q^\pm_\sigma$ to  $\E^{\rho,\pm}_\sigma$  equal to 0 on $(\range B^\pm_\sigma)^\perp.$  If $f\in\S_0V',$  define  the function $Jf$ by requiring that 
$$
P^{\sgn \la}_\sigma\pi_\la(Jf)=|\la|^{-a}(I\otimes \ga(|\la|^{-1}))\tilde Q^\pm_\sigma (I\otimes \ga'(|\la|))P^{\sgn \la}_\sigma\pi_\la(f).
$$

If  $f\in \overline {\range B}\cap \S_0V'$, then the range of $P_\sigma^{\sgn\la}\pi_\la(f)$ is contained in the range of $\pi_{\la,\sigma}(B)$.

Choose $0<\del<R$ such that (i) in the definition of $\S_0V$ holds for $f.$ Since  for $\del\le |\la|\le R$ the functions  $\ga'(|\la|)$ and
$\ga(|\la|^{-1})$ are smooth in $\la,$ it is easy to see that $Jf\in\S_0V.$
Moreover, applying \eqref{3.4'} to $g:=Jf,$ we see that  
$
\pi_{ \la}(B(Jf))=\pi_{\la}(f),
$
hence $f=B(Jf)\in  B(\S_0V).$
\qed

\begin{prop}\label{subspaces}
 Let $(V,\rho), (V',\rho')$, $\beta,\gamma,\gamma'$ be as above.  Then the following hold.
\begin{enumerate}
\item[\rm(i)] If $B$ is a $U(n)$- equivariant, left-invariant 
linear
operator,  homogeneous in the sense of
\eqref{homog}, and bounded from $L^2V$ to $L^2V'$, then $B$ satisfies the assumptions of Lemma \ref{s3.1} (i).
\item[\rm(ii)] Assume that $H\subset L^2V$ is a closed subspace, which is
invariant under left-translation by elements of $H_n$, under the
action of $U(n)$   and invariant under the dilations
$(\beta\otimes\ga)(r),\ r>0.$ Then 
$$
\S_0V=(\S_0V\cap H)\oplus (\S_0V\cap H^\perp)\ ,
$$ 
where  
$\S_0V\cap H$ is dense in $H$ and  $\S_0 V\cap H^\perp$ is dense in 
$H^\perp.$  
\end{enumerate}
\end{prop}

\proof
As in the proof of Lemma \ref{s3.1}, let $\{\Phi_\ell\}_{\ell\in\N}$ be an enumeration of the orthonormal basis of monomials in $\F$ and set $E_{\ell,\ell'}F=\lan F,\Phi_{\ell'}\ran_\F\Phi_\ell$.
Let also $\{e_i\}$ and $\{e'_j\}$ be  (finite) bases of $V$ and $V'$ respectively.

We fix an interval $I=[a,b]$ with $0<a<b$ and, for every $\ell\in\N$, a function $g_\ell\in\S_0$ such that $\pi_\la(g_\ell)=E_{\ell,\ell}$ for $\la\in I$.
Then $B(g_\ell\otimes e_i)\in L^2V'$ and
\begin{equation}\label{new-c's}
\pi_\la\big(B(g_\ell\otimes e_i)\big)=\sum_j\Big(\sum_{h,k}c^{\ell,i}_{h,k,j}(\la)E_{h,k}\Big)\otimes e'_j\ ,
\end{equation}
with $c^{\ell,i}_{h,k,j}\in L^2(I)$ for every choice of the indices. Then almost every point  $\la\in I$ is a Lebesgue point for all $c^{\ell,i}_{h,k,j}$ and for $\sum_{h,k}|c^{\ell,i}_{h,k,j}(\la)|^2$. 

For every $f=\sum_if_i\otimes e_i\in\S_0V$ and for a.e. $\la\in I$,
$$
\pi_\la(f)=\sum_{i,\ell}\pi_\la\big((f_i*g_\ell)*(g_\ell\otimes e_i)\big)\ ,
$$
where the sum is finite (say over $\ell\le N$). The invariance of $B$ under translations by elements $(0,t)$ of the center of $H_n$ implies that $B$ preserves the $\la$-support of the group Fourier transform. Therefore, we also have
$$
\pi_\la(Bf)=\sum_{i,\ell}\pi_\la\big(B((f_i*g_\ell)*(g_\ell\otimes e_i))\big)\ ,
$$
for a.e. $\la\in I$. On the other hand, $B\big((f_i*g_\ell)*(g_\ell\otimes e_i)\big)=(f_i*g_\ell)*B(g_\ell\otimes e_i)$. Hence, for a.e. $\la\in I$ (say, $\la\in\Lambda$),
$$
\pi_\la(Bf)=\sum_{i,\ell}\pi_\la\big(B(g_\ell\otimes e_i)\big)\pi_\la(f_i*g_\ell)\ .
$$

The same computations in the proof of Lemma \ref{s3.1} produce an infinite matrix $C_{i,j}(\la)=\big(c^{\ell,i}_{h,\ell,j}(\la)\big)_{h,\ell}$ with at most $N$ nonzero entries on each row, defined for $\la\in\Lambda$. Defining $B_\la$ by \eqref{B_la}, we have that
\begin{equation}\label{identity}
\pi_\la(Bf)=B_\la\pi_\la(f)\ ,
\end{equation}
 for $\la\in\Lambda$.
 
 Now, the homogeneity of $B$ easily implies that, for $\la,\la'\in\Lambda$, 
 $$
 B_\la=(\la/\la')^a \,\big(I\otimes \ga'(\la'/\la)\big)\, B_{\la'}\,\big(I\otimes \ga(\la/\la')\big)\ .
 $$

This identity allows to extend $B_\la$ as a smooth function  of $\la$ to every $\la>0$. 

Obviously, the same construction can be made for $\la<0$.
Then, for every $f\in\S_0V$, the identity \eqref{identity}
holds for every $\la\ne0$, which shows that $B(\S_0V)\subset\S_0V'$. Then Lemma \ref{s2.2} and the following remarks imply that we are in the hypotheses of Lemma \ref{s3.1} (ii).

\medskip
In order to prove (ii), let us denote by $P$ the orthogonal projection
from $L^2V$ onto  
$H.$ Since $H$ is invariant under left-translations, $U(n)$-invariant
and dilation invariant, $P$ is a left-invariant operator which is
$U(n)$-equivariant and homogeneous of degree 0. Moreover, by the Schwartz kernel
theorem, it  is given by the convolution $Pf=f*K$ with a tempered
distribution kernel $K$ taking values in $\L(V,V).$  
We may therefore apply (i) to $B:=P$ and conclude by means of Lemma \ref{density-lemma} that
$P(\S_0V)\subset\S_0V$, and similarly, $(I-P)(\S_0V)\subset\S_0V$. 
\qed

\bigskip

\setcounter{equation}{0}
\section{First properties of $\Delta_k$; exact and closed
  forms}\label{firstproperties}

The domain $\dom(\Delta_0)$, defined according to Lemma \ref{s3.1}, is
the ``left-invariant Sobolev space'' $H^2$ consisting of those $f\in
L^2$ such that $Xf,XYf\in L^2$ for every $X,Y\in\h_n$. This follows
from the $L^2$ boundedness of the operators $X(1+\Delta_0)^{-\half}$,
$XY(1+\Delta_0)\inv$ \cite{MPR}. We also recall that the operators
$XY\Delta_0\inv$, $X\Delta_0^{-\half}$ are bounded on $L^2$ for every
$X,Y\in\h_n$. \medskip

For $k\ge1$, we have the analogous description of $\dom(\Delta_k)$. 

\begin{lemma}\label{s4.1}
For every $k$, $\dom(\Delta_k)=H^2\Lambda^k$.
\end{lemma}

\proof It is evident from \eqref{1.12} that $H^2\Lambda^k\subset
\dom(\Delta_k)$. 

Since $\Delta_0=L-T^2$, identifying $\Delta_k$ with the matrix \eqref{1.12} 
 we have
 
$$
\aligned 
\Delta_k
&=\bpm \Delta_0&0\\ 0&\Delta_0\epm 
+\bpm \Delta_H-L+e(d\theta)i(d\theta)&i\bar\de-i\de\\ \\ 
i\de^*-i\bar\de^*&\Delta_H-L+i(d\theta)e(d\theta)\epm \\
&=\Delta_0+P\ .
\endaligned
$$
where $P$ is symmetric on $H^2\Lambda^k$.
By \eqref{1.10}, each entry in $P$ involves at most first-order
derivatives in the left-invariant vector fields. Therefore, for
$\om\in H^2\Lambda^k$, 
$$
\aligned
\|P\om\|_2&\le C\big(\|\om\|_2+\|\Delta_0^\half\om\|_2\big)\\
&\le C\big(\|\om\|_2+\|\Delta_0\om\|_2^\half\|\om\|_2^\half\big)\\
&\le C(1+\eps\inv)\|\om\|_2+C\eps\|\Delta_0\om\|_2\ ,
\endaligned
$$
for every $\eps>0$. By the Kato-Rellich theorem \cite{Kato},
$\Delta_0+P$ is self-adjoint on $\dom(\Delta_0)=H^2\Lambda^k$. 
\endproof

The following statement is an immediate consequence.

\begin{prop}\label{s4.2}
$\Delta_k$ is injective on its domain.
\end{prop}

\proof  Let $\om=\om_1+\theta\wedge\om_2\in\dom(\Delta_k)$, with
$\om_1,\om_2$ horizontal. Then 
$$
\aligned
\lan\Delta_k\om,\om\ran&=\lan\Delta_H\om_1,\om_1\ran
+\| T\om_1\|_2^2+\| i(d\theta)\om_1\|_2^2+\big\lan
\big[d_H^*,e(d\theta)\big]\om_2,\om_1\big\ran \\
&\qquad +\big\lan\big[i(d\theta),d_H\big]\om_1,\om_2\big\ran
+\lan\Delta_H\om_2,\om_2\ran+\| T\om_2\|_2^2+\| e(d\theta)\om_2\|_2^2\
. 
\endaligned
$$

Notice that
$$
\lan\Delta_H\om_1,\om_1\ran=\|d_H\om_1\|_2^2+\|d^*_H\om_1\|_2^2\ ,
$$
and the same holds for $\om_2$. Moreover,
$$
\aligned
\big\lan \big[d_H^*,e(d\theta)\big]
&\om_2,\om_1\big\ran+\big\lan\big[i(d\theta),d_H\big]\om_1,\om_2\big\ran\\
&=2\RE\big\lan\big[i(d\theta),d_H\big]\om_1,\om_2\big\ran\\
&=2\RE\big\lan d_H\om_1,e(d\theta)\om_2\big\ran-2\RE\big\lan
i(d\theta)\om_1,d_H^*\om_2\big\ran\\ 
&\ge -\|d_H\om_1\|_2^2-\|e(d\theta)
\om_2\|_2^2-\|i(d\theta)\om_1\|_2^2-\|d_H^*\om_2\|_2^2\ .
\endaligned
$$

It follows that
\begin{equation}\label{4.1}
\lan\Delta_k\om,\om\ran\ge
\|d^*_H\om_1\|_2^2+\|d_H\om_2\|_2^2+\|T\om_1\|_2^2+\|T\om_2\|_2^2\ . 
\end{equation}

Therefore, if $\Delta_k\om=0$, then $T\om=0$. Since
$\pi_\la(T)=i\la\,\id$, this implies that $\pi_\la(\om)=0$ for almost
every $\la$, and finally that $\om=0$. 
\endproof

\begin{cor}\label{s4.3}
 For every $\la>0$ and $\sigma\in\Sigma^\pm$,
 $d\pi_{\pm\la,\sigma}(\Delta_k)$ is invertible and for every pair of
elements 
 $u,v\in \E_\sigma^\pm$, 
 $\lan\pi_{\pm\la,\sigma}(\Delta_k)u,v\ran$ is a polynomial in
 $\la$. For every $\al>0$, $\Delta_k^{-\al}$ maps $\S_0\Lambda^k$ into
 itself. 
\end{cor}

\proof By \eqref{4.1}, $\|\Delta_k^\half\om\|_2\ge \|T\om\|_2$ for
every $\om\in\S_0\Lambda^k$. This implies that  
$$
\|\pi_{\la,\sigma}(\Delta_k)^\half \xi\|\ge|\la|\|\xi\|, \quad \xi\in \E^{\sgn\la}_\sigma,
$$ 
for every $\la,\sigma$ with $\la\ne0$. The rest is obvious.
\endproof

We call {\it Riesz transforms} the operators
$$
R_k=d\Delta_k^{-\half}:\S_0\Lambda^k\longrightarrow \S_0\Lambda^{k+1}\ ,
$$
and their adjoints
$$
R_k^*=\Delta_k^{-\half}d^*:\S_0\Lambda^{k+1}\longrightarrow \S_0\Lambda^k\ .
$$

\begin{lemma}\label{s4.4}
The following
identities hold (with the convention that $R_{-1}=R_{2n+1}=0$):
\begin{eqnarray}
& R_k=\Delta_{k+1}^{-\half}d\ ,
\quad R_k^*=d^*\Delta_{k+1}^{-\half}\ , \label{4.2}\\
& R_{k+1}R_k=R_k^*R_{k+1}^*=0\ ,\label{4.3}\\
& R_k^*R_k+R_{k-1}R_{k-1}^*=\id\ . \label{4.4}
\end{eqnarray}
In particular, $R_kR_{k-1}=0.$

Moreover, if $1\le k\le 2n$, $R_{k-1}R_{k-1}^*$, $R_k^*R_k$ are
orthogonal projections on complementary orthogonal subspaces of
$L^2\Lambda^k$ and $R_k$ and $R_k^*$ are partial isometries.  
\end{lemma}

\proof From the identity $d\Delta_k=\Delta_{k+1}d$ on test functions
we derive that  
$$
\pi_{\la,\sigma}(d)\pi_{\la,\sigma}(\Delta_k)
=\pi_{\la,\sigma}(\Delta_{k+1})\pi_{\la,\sigma}(d)
$$ 
for all $\la,\sigma$. Hence 
$$
\pi_{\la,\sigma}(d)\pi_{\la,\sigma}(\Delta_k)^{-\half}
=\pi_{\la,\sigma}(\Delta_{k+1})^{-\half}\pi_{\la,\sigma}(d)
$$
by finite-dimensional linear algebra. In turn, this gives the first
identity of \eqref{4.2} on $\S_0\Lambda^k$. The second identity is
proved in the same way. 

Then \eqref{4.3} follows from \eqref{4.2} and the identity $d^2=0$.

On $\S_0\Lambda^k$, by applying again $\pi_{\la,\sigma}$ to each term,
$$
\aligned
R_k^*R_k+R_{k-1}R_{k-1}^*
&=\Delta_k^{-\half}d^*d\Delta_k^{-\half}
+\Delta_k^{-\half}dd^*\Delta_k^{-\half}\\
&=\Delta_k^{-\half}\Delta_k\Delta_k^{-\half}\\
&=\id\ ,
\endaligned
$$
which gives  \eqref{4.4}. 

Since the two summands on the left-hand side of \eqref{4.4} are positive
operators, they are $L^2$-contractions. Since their sum is the
identity and their product is zero by \eqref{4.3}, they are
idempotent. This proves that they are orthogonal projections. It
follows that $R_k$ and $R_{k-1}^*$ are partial isometries.  
\endproof

The following statement says  in particular  that the cohomology
groups of the De Rham complex are trivial.  

\begin{prop}\label{s4.5}
Let $1\le k\le 2n$. The following subspaces of $L^2\Lambda^k$ are the same:
\bee
\item[\rm(i)] the range of $R_{k-1}R_{k-1}^*$;
\item[\rm(ii)] the range of $R_{k-1}$;
\item[\rm(iii)] $\ker R_k$;
\item[\rm(iv)] $\ker d$;
\item[\rm(v)]$\overline{d(\S_0\Lambda^{k-1})}$;
\item[\rm(vi)]$\overline{d(\D\Lambda^{k-1})}$;
\item[\rm(vii)]$\{\om\in L^2\Lambda^k: \om=du\mbox{  in the sense of
    distributions  for some } u\in \D'\Lambda^{k-1}\}$. 
\ee
We call this space $(L^2\Lambda^k)_{d\ex}$ or $(L^2\Lambda^k)_{d\cl}$.
Similarly, the following spaces 
\bee
\item[\rm(i')] the range of $R_k^*R_k$;
\item[\rm(ii')] the range of $R_k^*$;
\item[\rm(iii')] $\ker R_{k-1}^*$;
\item[\rm(iv')] $\ker d^*$;
\item[\rm(v')] $\overline{d^*(\S_0\Lambda^{k+1})}$;
\item[\rm(vi')] $\overline{d^*(\D\Lambda^{k+1})}$;
\item[\rm(vii')]$\{\om\in L^2\Lambda^k: \om=d^*v\mbox{  in the sense of
    distributions  for some } v\in \D'\Lambda^{k+1}\}$ 
\ee
are the same; we call them $(L^2\Lambda^k)_{d^*\ex}$ or $(L^2\Lambda^k)_{d^*\cl}$.
\end{prop}

\proof Since $R_{k-1}$ is a partial isometry, its range is closed, and
$$
\range R_{k-1}=(\ker R_{k-1}^*)^\perp=(\ker
R_{k-1}R_{k-1}^*)^\perp=\range R_{k-1}R_{k-1}^*\ . 
$$

This proves the identity of the spaces in (i) and (ii). In the same
way one proves the same for (i') and (ii'). From Lemma \ref{s4.4} we
then obtain the orthogonal decomposition 
$$
L^2\Lambda^k=\range R_{k-1}\oplus \range R_k^*\ .
$$

But $\range R_k^*=(\ker R_k)^\perp$, so that $\range R_{k-1}=\ker
R_k$, i.e. (ii)=(iii).  

By Plancherel's formula and Lemma \ref{s3.1}, $\om\in\ker d$ if and
only if $\pi_{\la,\sigma}(d)\pi_{\la,\sigma}(\om)=0$ for a.e. $\la$
and every $\sigma$. By Corollary \ref{s4.3} and \eqref{4.2}, this is
equivalent to saying that
$\pi_{\la,\sigma}(R_k)\pi_{\la,\sigma}(\om)=0$ for a.e. $\la$ and
every $\sigma$, i.e. that $R_k\om=0$. So (iii)=(iv).  
By Corollary \ref{s4.3}, $d(\S_0\Lambda^{k-1})=
R_{k-1}(\S_0\Lambda^{k-1})$ and this implies that (v)=(ii). 
\smallskip

We thus have shown that the spaces (i) - (v) are the same, and the equality of the spaces  (i')- (v') are proved in the same way.
\smallskip

In order to prove that the spaces (i) - (v) agree also with the space (vi), we first observe that
$\overline{d(\D\Lambda^{k-1})}\subset \ker d,$
since $d^2=0$ on $\D\Lambda^{k-1}.$ We thus have
$\overline{d(\D\Lambda^{k-1})}\subset R_{k-1}(L^2\Lambda^{k-1}).$ To
prove that these spaces are indeed the same, it will suffice to prove
that $\sigma\perp R_{k-1}(L^2\Lambda^{k-1})$ whenever $\sigma\in
L^2\Lambda^k$ satisfies $\sigma\perp d(\D\Lambda^{k-1}).$ But, the
latter condition means that $d^*\sigma=0$ in the sense of
distributions. So, by Lemma \ref{s3.1}, $\sigma\in \dom d^*,$ and
since (iv')=(ii'), we see that $\sigma=R_k^*\xi$ for some $\xi\in
L^2\Lambda^{k+1}.$ This implies that for every $R_{k-1}\mu\in
R_{k-1}(L^2\Lambda^{k-1})$ 
$$
\lan\sigma,R_{k-1}\mu\ran= \lan R_k^*\xi,R_{k-1}\mu\ran=\lan \xi, R_kR_{k-1}\mu\ran=0.
$$
We have thus seen that the spaces (i) - (vi) all agree, and in a
similar way one proves that the spaces (i') - (vi') are all the same. 

\smallskip 
The proof  that these spaces also do agree with the space (vii)
respectively (vii') will require deeper $L^p$-methods, and will
therefore be postponed to   Section \ref{applications} (see
Corollary \ref{exact2}). 

\endproof

\vskip.2cm

Working out the same program for $\de,\bar\de$, their adjoints and the
box-operators one encounters some differences. One simplification
comes from the fact that $\Box$ and $\Boxbar$ act as scalar operators
on horizontal forms of a given bi-degree.

On the other hand, a complication comes from the fact they have a
non-trivial null space in $L^2$ for certain values of $p$ or $q$.  
It is well known since \cite{FS} that $L+i\al T$ is injective on $L^2$
if and only if $\al\ne\pm(n+2j)$, $j\in\N$, and that it is
hypoelliptic under the same restriction. It follows from \eqref{1.10}
that $\Box$ (resp. $\Boxbar$) is injective, and hypoelliptic, on
$(p,q)$-forms provided that $p\ne0,n$ (resp. $q\ne0,n$). 

For $p=0$, $\Box_0=\de^*\de$ and $\ker \Box=\ker\de$, while, for
$p=n$, $\Box=\de\de^*$ and $\ker\Box=\ker\de^*$. Similarly,  $\ker
\Boxbar=\ker\bar\de$ for $q=0$,  and $\ker\Boxbar=\ker\bar\de^*$ for
$q=n$. 

For these values of $p$ (resp. $q$), we shall denote by $\Box'$
(resp. $\Boxbar'$) the unprimed operator with domain and range
restricted to the orthogonal complement of the corresponding null
space. Notice that the core $\S_0\Lambda^{p,q}$ splits according to
the decompositions $\ker\de\oplus(\ker \de)^\perp$,
$\ker\bar\de\oplus(\ker \bar\de)^\perp$. 
The negative powers ${\Box'}^{-\al}$(resp. ${\Boxbar '}^{-\al}$)  are
then well defined on $\S_0\Lambda^{p,q}\cap (\ker\de)^\perp$
(resp. $\S_0\Lambda^{p,q}\cap (\ker\bar\de)^\perp$). 

By \eqref{1.10}
$$
\Box_0=\Boxbar_n=\half (L+inT)\ ,\qquad \Box_n=\Boxbar_0=\half (L-inT)\ .
$$

We denote by $\cC$ (resp. $\bar \cC$) the orthogonal projection from
scalar $L^2$ onto $\ker (L+inT)$ (resp. $\ker (L-inT)$). The same symbols will be used to denote the extension to forms by componentwise application.

\smallskip
Thus, $\cC$ is the orthogonal projection onto $\ker \de$ when acting on $(0,q)$-forms as well as onto $\ker \de^*$ when acting on $(n,q)$-forms, and $\bar \cC$ is the orthogonal projection onto $\ker \bar\de$ when acting on $(p,0)$-forms as well as onto $\ker \bar \de^*$ when acting on $(p,n)$-forms.

Regard $\de$ as a closed operator from $L^2\Lambda^{p,q}$ to
$L^2\Lambda^{p+1,q}$. The {\it holomorphic Riesz transforms} are
defined on $\S_0\Lambda^{p,q}$ 
(with values in $\S_0\Lambda^{p+1,q}$) by 
\begin{equation}\label{4.5}
\Ri_p=\left\{  \begin{array}{cc}
 \de\Box_p^{-\half}=\Box_{p+1}^{-\half}\de\hfill & \text{for }1\le
 p\le n-2\ ,\hfill\\ 
\de{\Box'_0}^{-\half}(I-\cC)=\Box_1^{-\half}\de \hfill& \text{for } p=0\
,\hfill\\ 
\de\Box_{n-1}^{-\half}={\Box'}_n^{-\half}\de \hfill& \text{for }p=n-1\
.\hfill 
\end{array}\right.
\end{equation}

We observe that, is all cases,
\begin{equation}\label{factorde1}
\Ri_p\Box_p^\half=\Box_{p+1}^\half\Ri_p=\de\ .
\end{equation}

The adjoint operators $\Ri_p^*$ from $\S_0\Lambda^{p+1,q}$ to  $\S_0\Lambda^{p,q}$ are
 \begin{equation}\label{4.6}
\Ri_p^*=\left\{  \begin{array}{cc}
\Box_p^{-\half}\de^*=\de^*\Box_{p+1}^{-\half}\hfill & \text{for }1\le
p\le n-2\ ,\hfill\\ 
{\Box'_0}^{-\half}\de^*=\de^*\Box_1^{-\half} \hfill& \text{for } p=0\
,\hfill\\ 
\Box_{n-1}^{-\half}\de^*=\de^*{\Box'}_n^{-\half}(I-\bar \cC) \hfill&
\text{for }p=n-1\ .\hfill 
\end{array}\right.
\end{equation}

The analogues of \eqref{4.3} and  \eqref{4.4} are
\begin{equation}\label{4.7-9}
\aligned
&\Ri_{p+1}\Ri_p=\Ri_p^*\Ri_{p+1}^*=0\ ,\\
&\Ri_p^*\Ri_p+\Ri_{p-1}\Ri_{p-1}^*=\id\ ,\qquad (1\le p\le n-1)\\
&\Ri_0^*\Ri_0=I-\cC\ , \quad \Ri_{n-1}\Ri_{n-1}^*=I-\bar \cC\ .
\endaligned
\end{equation}

Proposition \ref{s4.5} has the following analogue.

\begin{prop}\label{s4.6}
For $0\le p\le n-1$, the following subspaces of $L^2\Lambda^{p,q}$ are
the same:  
\bee
\item[(i)] $\ker \Ri_p$;
\item[(ii)] $\ker \de$;
\ee
We call this space $(L^2\Lambda^{p,q})_{\de\cl}$.

For $1\le p\le n$, the following subspaces of $L^2\Lambda^{p,q}$ are
the same:
\bee
\item[(iii)] the range of $\Ri_{p-1}\Ri_{p-1}^*$;
\item[(iv)] the range of $\Ri_{p-1}$;
\item[(v)] $\overline{\de(\S_0\Lambda^{p-1,q})}$.
\ee
We call this subspace $(L^2\Lambda^{p,q})_{\de\ex}$.

For $1\le p\le n-1$,
$(L^2\Lambda^{p,q})_{\de\cl}=(L^2\Lambda^{p,q})_{\de\ex}$.\smallskip

Similarly, for $1\le p\le n$, the following subspaces of
$L^2\Lambda^{p,q}$ are the same:    
\bee
\item[(i')] $\ker \Ri_{p-1}^*$;
\item[(ii')] $\ker \de^*$;
\ee
and we call this subspace $(L^2\Lambda^{p,q})_{\de^*\cl}$. 

For $0\le p\le n-1$, the following subspaces of
$L^2\Lambda^{p,q}$ are the same:
\bee   
\item[(iii')] the range of $\Ri_p^*\Ri_p$;
\item[(iv')] the range of $\Ri_p^*$;
\item[(v')] $\overline{\de^*(\S_0\Lambda^{p+1,q})}$;
\ee
and we 
call this subspace $(L^2\Lambda^{p,q})_{\de^*\ex}$.

Finally, for $1\le p\le n-1$,
$(L^2\Lambda^{p,q})_{\de^*\cl}=(L^2\Lambda^{p,q})_{\de^*\ex}$.
\end{prop}

We also set
$(L^2\Lambda^k_H)_{\de\ex}=\sum_{p+q=k}(L^2\Lambda^{p,q})_{\de\ex}$
etc. 

The {\it antiholomorphic Riesz transforms} $\overline\Ri_q$ 
and their adjoints $\overline\Ri_q^*$ 
are defined by 
conjugating all terms in 
\eqref{4.5} and \eqref{4.6} respectively,
and replacing $p$
by $q$. 
The analogue of formula \eqref{factordebar1} also
holds true for all $q$
\begin{equation}\label{factorde1}
\overline\Ri_p\Boxbar_q^\half=\Boxbar_{q+1}^\half\overline\Ri_q=\bar\de\ .
\end{equation}

The rest goes in perfect analogy with the holomorphic case. 

\begin{defn}\label{C-Cbar}
{\rm
On $(p,q)$-forms, we also define the operators
\begin{equation}\label{C-on-pq-forms}
\aligned
& C_p = I- \Ri_p^* \Ri_p\ , \qquad
\overline C_q = I- \overline\Ri_q^* \overline\Ri_q\ ,\qquad \text{for}\quad 0\le p,q\le
n-1\ ,\\
& C_n=I=\overline C_n.
\endaligned
\end{equation}
Notice that, by \eqref{4.7-9},  $C_p=\Ri_{p-1}\Ri_{p-1}^*$ for $1\le p\le n-1,$  and similarly 
$\overline C_q=\overline\Ri_{q-1}\overline \Ri_{q-1}^*$ for $1\le q\le n-1.$ 
}
\end{defn}

The following statements are obvious in view of Proposition \ref{s4.6}.
\begin{lemma}\label{s4.8}
$C_p$ is  the orthogonal projection of
$L^2 \Lambda^{p,q}$ onto the kernel of $\de,$    and $\overline C_p$ is  the orthogonal projection of $L^2 \Lambda^{p,q}$ onto the kernel of $\overline\de,$

Moreover,  if $\om\in\S_0 \Lambda^{p,q}$, with $1\le p\le n-1$, then 
$$
C_p \om=0\qquad\text{if and only if}\qquad \om\in\overline{\de^*(\S_0\Lambda^{p+1,q})}  \qquad\text{if and only if}\qquad \de^* \om=0
,
$$
whereas for $p=0$, 
$$
C_0 \,\om=0\qquad\text{if and only if}\qquad \om\in\overline{\de^*(\S_0\Lambda^{1,q})},
$$
and for $p=n$,
$$
C_n \,\om=0\qquad\text{if and only if}\qquad \om=0.
$$
 Analogous statements hold for the operators $\overline C_q,$ if we
 replace $p$ by $q$ and conjugate all terms.  
 In particular,  $C_0=\cC, \overline C_0=\overline \cC,$ and
 $\de^*\om=0$ whenever $C_p\om=0,$ and $\overline\de^*\om=0$ whenever
 $\overline C_q\om=0.$ 
\end{lemma}

Given a horizontal $k$-form $\omega=\sum_{p+q=k}\omega_{pq}$ we
finally set
\begin{equation}\label{C-on-forms}
C\omega= \sum_{p+q=k}C_p \omega_{pq}\, ,
\qquad\text{and}\quad
\overline C\omega= \sum_{p+q=k}\overline C_q \omega_{pq}\, .
\end{equation}

\bigskip

\setcounter{equation}{0}
\section{A decomposition of $L^2\Lambda_H^k$ 
related to the $\de$ and $\bar\de$ complexes}\label{decomposition}

In 
this section,
 we shall  work under the {\it assumption that $0\le k\le n,$}
as this turns out to be more convenient in view of the Lefschetz
decomposition described in Prop. 2.1 of  \cite{MPR}. 
 The case where $k>n$ can be reduced to the case $k\le n$ by means of
 Hodge duality, as will be shown later  in  Section
 \ref{Hodge}. 
\medskip

Our starting point in the spectral analysis of $\Delta_k$ 
is the 
decomposition obtained in Proposition  \ref{s4.5}

\begin{equation}\label{maindecom}
L^2 \Lambda^k = 
(L^2\Lambda^k)_{d\ex}
\oplus 
(L^2\Lambda^k)_{d^*\cl}\ .
\end{equation}

Since $d\Delta_{k-1}=\Delta_k d$ for all $k\ge1$, using the results
from \cite{MPR} for $\Delta_1$, we can lift the decomposition
of $L^2 \Lambda^1$ into $\Delta_1$-invariant subspaces and the related
spectral properties to $(L^2\Lambda^2)_{d\ex}$.
Therefore, inductively we analyse the 
$(L^2\Lambda^k)_{d\ex}$-component in the decomposition of 
$L^2 \Lambda^k$
by means of the preceeding step.

Thus, we are led to study the 
$(L^2\Lambda^k)_{d^*\cl}$-component in the decomposition of 
$L^2 \Lambda^k$.

By \eqref{1.11} we can characterize the $d^*$-closed forms.
Notice that, if
$\om\in\S_0\Lambda^k$, $\om=\om_1+\theta\wedge\om_2$ with
$\om_1,\om_2$ horizontal, then
$$
\om\in \bigl(\S_0\Lambda^k\bigr)_{d^*\cl}
\qquad\text{if and only if}\qquad \om_2 =T\inv d_H^* \om_1 \
. 
$$
In fact, if $\om_2 =T\inv d_H^* \om_1$,
then the  second equation $i(d\theta)\om_1-d_H^*\om_2=0$ arising
from in
\eqref{1.11} follows from the first one.

Hence, if we set
\begin{equation}
\label{Phi}
\Phi(\om) = \om+\theta\wedge T\inv d_H^* \om
\end{equation}
we obtain an isomorphism
$$
\Phi:\, \S_0\Lambda_H^k \longrightarrow \bigl(\S_0\Lambda^k\bigr)_{d^*\cl}\ .
$$
Notice that,
because of the invariance of $\theta$ and the equivariance of $d_H^*$,
$\Phi$ commutes with the action of $U(n)$.  
\medskip

Clearly, $\Delta_k$ maps a subspace $V$ of $(\S_0\Lambda^k)_{d^*\cl}$
into itself if and only if the (non-differential, see \eqref{Dk}
below) operator 
\begin{equation}\label{DefDk}
D_k:=\Phi\inv \Delta_k \Phi\,,
\end{equation}
maps $W=\Phi\inv V\subset \S_0\Lambda^k_H$ into itself.

For this reason, we begin by decomposing $\S_0\Lambda^k_H$ into
orthogonal subspaces which are invariant under $D_k$ and on which
$D_k$ takes a simple form. 

\subsection{The subspaces}

\ \medskip

The decomposition is based on the following lemma.

\begin{lemma}
Every $\om\in\S_0\Lambda^k_H$ decomposes as
\begin{equation}\label{threeterms}
\om=\om'+\de\xi+\bar\de\eta\ ,
\end{equation}
where $\xi,\eta\in\S_0\Lambda^{k-1}_H$, and
$\om'\in\S_0\Lambda^k$ satisfies the condition
\begin{equation}\label{om'}
\de^*\om'=\bar\de^*\om'=0\ .
\end{equation}

The term $\om'$ is uniquely determined, and we can assume, in addition, that
\begin{equation}\label{xi-eta}
C_{p-1}\xi=\bar C_{q-1}\eta=0\ .
\end{equation}
\end{lemma}

Notice that, even with the extra assumption \eqref{xi-eta}, $\xi$ and
$\eta$ are not uniquely determined. 

\proof Assume that $\om$ is a $(p,q)$-form.
If $p=0,$ we obviously have the decomposition
$\om=\om'+\bar\de\eta$, with $\bar\de^*\om'=0,$ and $\de^*\om'=0 $ holds tivially, since 
$\om'$ is a $(0,q)$- form. A  similar argument applies if $q=0$.  

We therefore assume that $p,q\ge1$.
Consider the homogeneous $U(n)$-equivariant differential operator
$$
(\de\ \bar\de): \ \begin{pmatrix}\xi\\ \eta\end{pmatrix}\mapsto \de\xi+\bar\de\eta\ ,
$$ 
acting from
$L^2(\Lambda^{p-1,q})\oplus L^2(\Lambda^{p,q-1})$ to $L^2\Lambda^{p,q}$ and
its adjoint 
$
\begin{pmatrix}\de^*\\ \bar\de^*\end{pmatrix}$.
 
In $L^2\Lambda^{p,q}$, we
have 
$$
\range (\de\ \bar\de)=\range\de+\range\bar\de\ ,\qquad \ker
\begin{pmatrix}\de^*\\ \bar\de^*\end{pmatrix}=\ker\de^*\cap\ker\bar\de^*\ , 
$$
so that
$$
L^2\Lambda^{p,q}=(\ker\de^*\cap\ker\bar\de^*)\oplus
\overline{(\range\de+\range\bar\de)}\ . 
$$

Moreover,
$$
(\de\ \bar\de)(\S_0\Lambda^{p-1,q}\oplus \S_0\Lambda^{p,q-1})
=\de\S_0\Lambda^{p-1,q}+\bar\de\S_0\Lambda^{p,q-1}\ ,
$$
so that, by Lemma \ref{density-lemma},
$$
\S_0\Lambda^{p,q}=(\ker\de^*\cap\ker\bar\de^*\cap\S_0\Lambda^{p,q})
\oplus(\de\S_0\Lambda^{p-1,q}+\bar\de\S_0\Lambda^{p,q-1})\ .
$$

This gives the decomposition \eqref{threeterms}. By orthogonality, the
two terms $\om'$ and $\de\xi+\bar\de\eta$ are uniquely
determined. Since $C_{p-1}$ and $\bar C_{q-1}$ preserve $\S_0$-forms,
we can replace $\xi$ by 
 $(I-C_{p-1})\xi$ and $\eta$ by $(I-\bar
C_{q-1})\eta$, without changing the equality. 
\endproof

Observe that the decomposition \eqref{threeterms}, without the extra
assumptions on $\xi$ and $\eta$, can be iterated, so to obtain in a next step that 
$$
\begin{aligned}
\om
&=\om'+\de(\xi'+\de\al_1+\bar\de\beta_1)
+\bar\de(\eta'+\de\al_2+\bar\de\beta_2)\\
&=\om'+\de\xi'+\bar\de\eta'+\de\bar\de\beta_1+\bar\de\de\beta_2\ ,
\end{aligned}
$$
where now each of the primed symbols represents a form satisfying
\eqref{om'}. If $\om$ is a horizontal $k$-form, the iteration stops
after $k$ steps, leaving no ``remainder terms''. 

We are so led to introduce, for each $m\le k$ the spaces of forms
\begin{equation}\label{iterated}
\om=\underbrace{\cdots \de\bar\de\de}_{\text{\rm
    $m$-terms}}\xi+\underbrace{\cdots \bar\de\de\bar\de}_{\text{\rm
    $m$-terms}}\eta\ , 
\end{equation}
with $\xi,\eta\in\S_0\Lambda^{k-m}_H$ and
$\de^*\xi=\bar\de^*\xi=\de^*\eta=\bar\de^*\eta=0$. 

It is convenient to observe that in a sequence of at least three
alternating $\de$'s and $\bar\de$'s, we can replace 
a product  $\de\bar\de$  or   $\bar \de\de$ by
$d_H^2=-T\inv e(d\theta)$. Since $T\inv$ preserves $\S_0$-forms, the
form $\om$ in \eqref{iterated} can thus be written as 
$$
\om=\begin{cases} 
e(d\theta)^\ell(\de\xi+\bar\de\eta)&\text{ if }m=2\ell+1\ ,\\
e(d\theta)^\ell(\bar\de\de\xi+\de\bar\de\eta)&\text{ if }m=2\ell+2\ .
\end{cases}
$$

\begin{defn}
{\rm 
We set
$$
\aligned
& W_0^{p,q} 
= \bigl\{ \om\in\S_0\Lambda^{p,q}:\de^*\om=\bar\de^*\om=0\bigr\}\ , \\
& W_1^{p,q} = \bigl\{ \om=\de\xi+\bar\de\eta\,: \xi,\eta\in W_0^{p,q} \bigr\}\ ,\\
&
W_2^{p,q} = 
\bigl\{ \om=\bar\de\de\xi+\de\bar\de\eta\, :\, 
\xi,\eta\in W_0^{p,q} \bigr\}\,.
\endaligned
$$

For $\ell\in\N$ and $j=1,2$, we set
$$
W^{p,q}_{j,\ell}=e(d\theta)^\ell W^{p,q}_j\,.
$$

We also set
$$
W_0^k=\sum_{p+q=k} W_0^{p,q}= \bigl\{ \om\in\S_0\Lambda^k_H:\de^*\om=\bar\de^*\om=0\bigr\}\ ,
$$
and, for $j=1,2$ and $\ell\in\N,$
$$
W_j^k=\sum_{p+q+j=k} W_j^{p,q}
$$
 and 
$$
W^k_{j,\ell}=\sum_{p+q+j+2\ell=k} W_{j,\ell}^{p,q}=e(d\theta)^\ell W^{k-2\ell}_j\,,
$$
whenever $k\ge j+2\ell. $ 

The symbols $\mathcal W_j^k,\mathcal W_j^{p,q}$, etc. denote the $L^2$-closures 
of the corresponding spaces $W_j^k,\ W_j^{p,q}$, etc..
}
\end{defn}

We wish to characterize which spaces among the $W_0^{p,q}$ and $W_{j,\ell}^{p,q}$ are
non-trivial. 

\begin{prop}\label{non-triviality-lemma}
Let $0\le k\le n$ and $p+q=k$.  Then $W^{p,q}_0$ is trivial if and
only if $k=n$ and $1\le p,q\le n-1$.
\end{prop}

\proof
We show first that $W^{p,q}_0$ is non-trivial for $p+q\le n-1$. In order to do so, it is sufficient to prove that, under this assumption, there is a non-zero $\beta\in \P_1\otimes \Lambda^{p,q}$, with $\P_1=\span\{w_1,\dots,w_n\}$ as in \eqref{2.2}, such that
$$
\pi_1(\de^*)\beta=\pi_1(\bar\de^*)\beta=0\ .
$$

From this it will easily follow from \eqref{de^*-om} and \eqref{pi(Z)} that $\pi_\la(\de^*)\beta=\pi_\la(\bar\de^*)\beta=0$ for every $\la>0$. 

Let $\om\in\Lambda^{p,q}$ be such that $\pi_\la(\om)=\chi(\la)P_\beta$, where $\chi$ is a smooth cut-off function with compact support in $(0,+\infty)$ and $P_\beta$ is the orthogonal projection of $\F\otimes\Lambda^{p,q}$ onto $\C\beta$. Then $\om\in\S_0\Lambda^{p,q}$ and $\de^*\om=\bar\de^*\om=0$.

Take 
$$
\beta = \biggl( \sum_{j=1}^{p+1} (-1)^j w_j\, \zeta\wedge\cdots\wedge
\widehat{\zeta_j}\wedge \cdots\wedge\zeta_{p+1} \bigg)
\wedge\bar\zeta^{I'}\ ,
$$
where  $I'=(p+2,\dots,p+q+1)$. Then, writing $I_{\widehat j} =
(1,\dots,\widehat j,\dots,p+1)$ we have
$$
\aligned
\pi_1 (\de^*)(\beta)
& = \frac1{\sqrt2} \sum_{\ell,J}\sum_{j=1}^{p+1} (-1)^j \varepsilon^{I_{\widehat j} }_{\ell J}
w_\ell w_j \, \zeta^J \wedge\bar\zeta^{I'}\\
& =  \frac1{\sqrt2}  \sum_{j=1}^{p+1} (-1)^j \Big[ \sum_{\ell=1}^{j-1}
(-1)^{\ell-1} w_\ell w_j \, \zeta_1\wedge\cdots\wedge
\widehat{\zeta_\ell}\wedge \cdots\wedge
\widehat{\zeta_j}\wedge \cdots\wedge \zeta_{p+1} 
\\
&\qquad\qquad\qquad\qquad + \sum_{\ell=j+1}^{p+1}
(-1)^\ell w_\ell w_j \, \zeta_1\wedge\cdots\wedge
\widehat{\zeta_j}\wedge \cdots\wedge
\widehat{\zeta_\ell}\wedge \cdots\wedge \zeta_{p+1}
\Big] \wedge\bar\zeta^{I'} \\
& = 0\ .
\endaligned
$$ 

Next, since $\pi_1(Z_\ell)=\sqrt2\de_{w_\ell}$
we have
$$
\pi_1 (\bar\de^*)(\beta)
= (-1)^{p+1}\sqrt2 \sum_{\ell,\, J'} \varepsilon^{I'}_{\ell J'} \big(
\de_{w_\ell}w_j \big) \zeta^{I_{\widehat j}}\wedge \overline
\zeta^{J'} = 0\ ,
$$
since $j\not\in I'$, so that $\de_{w_\ell}w_j=0$.

This shows that $W^{p.q}_0\neq \{0\}$ when $p+q\le n-1$.

Next, consider $W^{n,0}_0$.  Take $\beta = \zeta_1\wedge\cdots\wedge
\zeta_n\in\P_0\otimes\Lambda^{n,0}$.  Clearly $\pi_\lambda (\bar\de^*) \beta=0$ for every $\la\ne0$, while $\pi_\lambda (\bar\de^*) \beta=0$ for $\la<0$ by  \eqref{de^*-om} and \eqref{pi(Z)}.
 As before, this  implies
that $W^{n,0}_0\neq\{0\}$.
\medskip

Finally, consider $W^{n-s,s}_0$, with $1\le s\le n-1$
and let $\omega\in W^{n-s,s}_0$.  Since $\Box$ is
injective on this space and $\de^*\omega=0$ we have
$$
\omega = \de^* \big( \de \Box\inv \big)\omega=: \de^* \nu\ .
$$

Similarly,  since  $\bar\de^*\omega=0$,
$$
\aligned
\omega 
& = \bar\de^* \bar\de \Boxbar\inv \omega 
= \bar\de^* \bar\de\Boxbar\inv \de^* \nu \\
& = \bar\de^* \bar\de \de^* (\Boxbar-iT) \inv \nu 
=  \bar\de^* \de^* \big( -\bar\de (\Boxbar-iT) \inv \nu\big) \ ,
\endaligned
$$
i.e. 
\begin{equation}\label{def-mu}
\omega = \bar\de^* \de^* \mu\ ,
\end{equation}
for some $\mu$.  But $\bar\de^*\omega=0$ if and only if
$$
0=  \de^* \bar\de^* (\de^* \mu) = 
 ( \bar\de^* \de^*+ \de^* \bar\de^*) (\de^* \mu) = 
Ti(d\theta)(\de^* \mu) \ .
$$

It follows that, if $\mu$ is as in \eqref{def-mu}, then $\omega
= \bar\de^* \de^* \mu \in W^{s,n-s}_0$ if and only if 
\begin{equation}\label{mu-ker-idtheta}
i(d\theta)(\de^* \mu) =0\ ,
\end{equation}
i.e. $\de^* \mu \in \ker i(d\theta)$. 

Therefore, 
$\de^*\mu\in \S_0 \Big(\ker_{\Lambda^{s,n-s+1}} i(d\theta) \Big)$.
Since $\max\{ 0, s+(n-s+1) -n\} =\max\{0,1\}=1>0$,
according
to Prop. 2.1 in in \cite{MPR}, we have 
$\ker_{\Lambda^{s,n-s+1}} i(d\theta)=
\{0\}$, that is, $\de^*\mu=0$; hence $\omega=0$.
\qed

\begin{prop}\label{non-tr}
Assume that $j=1,2$. Then the  space 
$W_{j,\ell}^{p,q}$ 
is non-trivial  if and only if
$\ell+j+p+q\le n$.
In this case, 
$e(d\theta)^\ell$ is bijective from  $W_{j}^{p,q}$ onto $W_{j,\ell}^{p,q}.$   
\end{prop}

\proof
 We first prove the ``only if'' part.
Observe that $W_0^{p,q}$ and $W_1^{p,q}$ are in the kernel of
$i(d\theta),$ which is immediate from \eqref{1.6} and \eqref{1.7}. 

In order to prove the statement for $j=1,$ we set
$(\tilde p,\tilde q)=(p+1,q)$ or $(p,q+1).$ By Prop.  2.1 in
\cite{MPR} we know that 
$e(d\theta)^\ell \bigl(\ker i(d\theta)_{|_{L^2 \Lambda^{\tilde p,\tilde q}}}\bigr)$
is non-trivial if and only if
$\max(0,\tilde p+\tilde q+2\ell-n)\le\ell$, that is, 
$\ell\le n-\tilde p-\tilde q$.  Since $W_1^{p,q}\subseteq
L^2 \Lambda^{p+1,q}+ L^2 \Lambda^{p,q+1}$ it follows that
$W_{1,\ell}^{p,q}$  can be non-trivial only when 
$\ell\le n-p-q-1.$ 

To prove that  $e(d\theta)^\ell$ is injective on $W_{1}^{p,q}$ under this condition, we show by induction on $\ell$ that $e(d\theta)^\ell$ is injective on 
$\ker i(d\theta)_{|_{\Lambda^{\tilde p,\tilde q}}}$
when 
$\ell \le n-\tilde p-\tilde q=n-p-q-1$. The case $\ell=0$ is trivial. And, by \eqref{1.7tris} we see that  for $\ell\ge 1$ when acting on $(\tilde p,\tilde q)$-forms 
\begin{equation}\label{iem}
\aligned 
\ [i(d\theta),\, e(d\theta)^\ell]
& = \sum_{\nu=0}^{\ell-1}
e(d\theta)^\nu
[i(d\theta),\, e(d\theta)]e(d\theta)^{\ell-1-\nu} \\
& =
\sum_{\nu=0}^{\ell-1} (n-\tilde p-\tilde q-2\ell+2+2\nu) e(d\theta)^{\ell-1}
\\
& = \ell(n-\tilde p-\tilde q-\ell+1)e(d\theta)^{\ell-1}\\
& = \ell(n- p- q-\ell)e(d\theta)^{\ell-1}
\ ,
\endaligned
\end{equation}
which allows to prove injectivity of  $e(d\theta)^\ell$ on $\ker i(d\theta)_{|_{\Lambda^{\tilde p,\tilde q}}}$ from injectivity of  $e(d\theta)^{\ell-1}$ under the assumption on $\ell$.

\smallskip
We now turn to the case $j=2,$ which requires a more refined discussion. Let us set
$$
K_2^{p,q}=W_2^{p,q}\cap\ker i(d\theta)_{|\, L^2\Lambda^{p+q+2}_H} .
$$
We claim that $W_2^{p,q}$   decomposes  as an orthogonal sum
\begin{equation}\label{K2}
 W_2^{p,q} =  K_2^{p,q} \oplus e(d\theta) W^{p,q}_0.
\end{equation}
It is obvious by \eqref{1.6} that $e(d\theta) W^{p,q}_0\subset W_2^{p,q},$  and clearly  the two subspaces on the right-hand side are orthogonal. 

Assume that  $\om=\de\bar\de\xi+\bar\de\de\eta\in W_2^{p,q},$ with $\xi,\eta\in W_0^{p,q}.$ Then 
$$i(d\theta)\om=i(\Box\eta-\Boxbar\xi)\in W^{p,q}_0.
$$ 
Indeed, by \eqref{1.6bis} and
\eqref{commutations} we have 
$$
\aligned
i(d\theta)\bigl(\de\bar\de\xi+\bar\de\de\eta\bigr)
& = T\inv {d_H^*}^2 \bigl(\de\bar\de\xi+\bar\de\de\eta\bigr) \\
& = T\inv \bigl( -\de^*\de\Boxbar\xi +\bar\de^*\Box\bar\de\xi
+\de^*\Boxbar\de\eta -\bar\de^*\bar\de\Box\eta\bigr) \\
& =  T\inv \bigl( -\Box\Boxbar\xi +(\Box-iT)\Boxbar\xi
+(\Boxbar+iT)\Box\eta -\Boxbar\Box\eta\bigr) \\
& =  i(\Box\eta-\Boxbar\xi)\ .
\endaligned
$$  
We have seen that  $i(d\theta) W_2^{p,q}\subset W^{p,q}_0,$ 
and therefore $\om\in  W_2^{p,q}\cap (e(d\theta) W^{p,q}_0)^\perp$ if
and only if $\om\in K_2^{p,q}.$ This proves \eqref{K2}. Let us set
$K_{2,\ell}^{p,q}=e(d\theta)^\ell K_2^{p,q}.$ Then  
\begin{equation}\label{K22}
W_{2,\ell}^{p,q} =  K_{2,\ell}^{p,q} \oplus e(d\theta)^{\ell+1} W^{p,q}_0,
\end{equation}
and this decomposition is again orthogonal. 

Indeed, if $\omega\in\ker(i(d\theta))$ and $\sigma$ is a form
orthogonal to $\omega$, then for every $\ell\ge 1$
$$
\lan e(d\theta)^\ell\omega,\, e(d\theta)^\ell\sigma\ran  =0.
$$
For, by \eqref{iem}Ê we have 
$$
i(d\theta)e(d\theta)^\ell\om=c_\ell e(d\theta)^{\ell-1}\om\, , 
$$
which implies, by induction on $\ell,$ that
$$
\lan e(d\theta)^\ell\omega,\, e(d\theta)^\ell\sigma\ran  =
\lan i(d\theta)e(d\theta)^\ell\omega,e(d\theta)^{\ell-1} \sigma\ran =0\ .
$$

In order to prove the statement in the lemma for $j=2,$ using the
orthogonal decomposition in \eqref{K22} we may now argue as before by
means of Prop.  2.1 in \cite{MPR} in order  to see that
$W_{2,\ell}^{p,q}$ can be non-trivial only if $\ell\le n-p-q-2.$
Moreover, to verify that $e(d\theta)^\ell$ is  injective on
$W_{2}^{p,q}$ under this condition, it suffices to check injectivity
on each of the subspaces on the right-hand side of \eqref{K2}. But
this can be done by the same reasoning that we used for the case
$j=1.$ \medskip

For the ``if'' part, assume again that $p+q+j+\ell\le n$ and $ j=1,2$.  Then,
$p+q\le n-1$, so that
$W^{p,q}_0\neq\{0\}$ by Proposition \ref{non-triviality-lemma}.  Then 
$W^{p,q}_j\neq\{0\}$, and by the first part $W^{p,q}_{\ell,j}\neq\{0\}$.
\qed

\begin{lemma}\label{de*}
 For $\xi\in W^{p,q}_0$,
\begin{equation}\label{[de*,ede]}
\begin{aligned}
\de^*e(d\theta)^\ell \de\xi
&=e(d\theta)^\ell\Box \xi+i\ell e(d\theta)^{\ell-1}\bar\de\de\xi\\ 
\de^*e(d\theta)^\ell\bar\de\xi
&=0\\
\bar\de^*e(d\theta)^\ell \bar\de\xi
&=e(d\theta)^\ell\barBox \xi-i\ell e(d\theta)^{\ell-1}\de\bar\de\xi\\ 
\bar\de^*e(d\theta)^\ell\de\xi
&=0\ ,
\end{aligned}
\end{equation}
and
\begin{equation}\label{[de*,edebarde]}
\begin{aligned}
\de^*e(d\theta)^\ell \bar\de\de\xi
&=-e(d\theta)^\ell\bar\de\Box \xi\\ 
\de^*e(d\theta)^\ell \de\bar\de\xi
&=e(d\theta)^\ell\bar\de\big(\Box \xi-i(\ell+1)T\big)\xi\\ 
\bar\de^*e(d\theta)^\ell \de\bar\de\xi
&=-e(d\theta)^\ell\de\barBox \xi\\ 
\bar\de^*e(d\theta)^\ell \bar\de\de\xi
&=e(d\theta)^\ell\de\big(\barBox \xi+i(\ell+1)T\big)\xi\ .
\end{aligned}
\end{equation}
\end{lemma}

\proof 
Since $[\de^*,e(d\theta)]=i\bar\de$ commutes with $e(d\theta)$
(compare \eqref{1.7}, \eqref{1.7bis}) and 
similarly for $\bar\de^*$, we obtain by induction that 
\begin{equation}\label{commute-e}
[\de^*,e(d\theta)^\ell]=i\ell\bar\de e(d\theta)^{\ell-1}\
,\qquad[\bar\de^*,e(d\theta)^\ell]=-i\ell\de e(d\theta)^{\ell-1}\
. 
\end{equation}

We verify the first identity in \eqref{[de*,ede]}, the others being similar and following by invoking also \eqref{1.6} and \eqref{commutations}:
$$
\begin{aligned}
\de^*e(d\theta)^\ell \de\xi &= e(d\theta)^\ell\de^*\de\xi+i\ell\bar\de e(d\theta)^{\ell-1}\de\xi\\
&=e(d\theta)^\ell\Box\xi+i\ell e(d\theta)^{\ell-1}\bar\de\de\xi\ . 
\end{aligned} 
$$
\endproof

This immediately gives the following inclusions.

\begin{cor}\label{de*W}
For $\ell\ge1$,
$$
\de^*W^{p,q}_{1,\ell}\subset W^{p,q}_{2,\ell-1}\ ,\qquad \de^* W^{p,q}_{2,\ell}\subset W^{p,q}_{1,\ell}\ ,
$$
and similarly for $\bar\de^*$.
\end{cor}

\begin{prop} \label{horizdecomp}
$L^2\Lambda^k_H$ decomposes as the orthogonal sum
\begin{eqnarray*}
L^2\Lambda^k_H&=&\osum_{p+q=k}\W^{p,q}_0\oplus
\osum_{\substack{j,\ell,p,q\\  j=1,2\\ p+q+j+2\ell=k}}\W^{p,q}_{j,\ell}\\
&=& \W^k_0\oplus \osum_{1+2\ell\le k}\W^k_{1,\ell}\oplus \osum_{2+2\ell\le k}\W^k_{2,\ell}.
\end{eqnarray*}

We recall that $\W^{p,q}_0$ is non-trivial for $p+q\le n-1$, and if
$p+q=n$ for $pq=0$.
\end{prop}

\proof 
We have already shown that $\S_0\Lambda^k_H$ is contained in
the sum of the subspaces on the right-hand side. It is then sufficient
to show that any two $\S_0$-forms belonging to two different subspaces
are orthogonal.  

It is quite obvious that $W^{p,q}_0$ is orthogonal to $W^{p',q'}_0$ if $(p,q)\ne(p',q')$. 

The fact that $W^k_0$ is orthogonal to $W^k_{j,\ell}$ for  $j=1,2$ is a
consequence of the fact that $W^k_0\subset \ker \de^*\cap\ker\bar\de^*$,
whereas $W^k_{j,\ell}\subset\range\de+\range\bar\de$.

To prove the remaining orthogonality relations, we shall  proceed
inductively. For this purpose, it will be convenient to represent the 
elements of $W^k_{j,\ell}$ in the form \eqref{iterated} with
$m=j+2\ell,$ and rename, for the purpose of this proof, 
$W^{p,q}_{j,\ell}$ as $W^{p,q}_m$ if $m=j+2\ell.$    Given $m\ge m'\ge1$,
there are three kinds of scalar products to consider, 
$$
\lan \underbrace{\de\bar\de\cdots}_{\text{\rm
    $m$-terms}}\sigma,\underbrace{\de\bar\de\cdots }_{\text{\rm
    $m'$-terms}}\sigma'\ran\ ,\qquad 
\lan \underbrace{\bar\de\de\cdots}_{\text{\rm
    $m$-terms}}\sigma,\underbrace{\de\bar\de\cdots }_{\text{\rm
    $m'$-terms}}\sigma'\ran\ ,\qquad 
\lan \underbrace{\bar\de\de\cdots}_{\text{\rm
    $m$-terms}}\sigma,\underbrace{\bar\de\de\cdots }_{\text{\rm
    $m'$-terms}}\sigma'\ran\ , 
$$
with $\sigma\in W^{p,q}_0$ and $\sigma'\in W^{p',q'}_0$. In the first case we have
$$
\lan
\underbrace{\de\bar\de\cdots}_{m}\sigma,\underbrace{\de\bar\de\cdots
}_{m'}\sigma'\ran=\lan\underbrace{\bar\de\cdots}_{m-1}\sigma,
\de^*\underbrace{\de\bar\de\cdots}_{m'}\sigma'\ran\ . 
$$

By Corollary \ref{de*W}, this is the scalar product of an element of
$W^{p,q}_{m-1}$ with an element of $W^{p',q'}_{m'-1}$. 

By induction on $m'$, this shows that $W^{p,q}_m\perp W^{p',q'}_{m'}$
unless $p=p'$, $q=q'$, $m=m'$. 

\endproof

We discuss now to what extent the pairs $(\xi,\eta)\in W_0^{p,q}\times
W_0^{p,q}$ provide a parametrization of the spaces $W_{j,\ell}^{p,q}$
for $j=1,2$. 

\begin{lemma}\label{modi}
Given $ \xi\in W_0^{p,q}$, there exists a unique $ \xi'\in W_0^{p,q}$
such that $\de\xi=\de\xi'$ and  $C_p\xi'=0$.  
An analogous statement holds for $\bar\de$ in place of $\de$.
\end{lemma}

\proof
The case $p=n$ is trivial - here $\xi'=0.$ If $1\le p\le n-1,$ then by
Lemma  Ê\ref{s4.8} we have $\xi'=\xi.$  

There only remains the case  $p=0,$  where $C_0=\cC$ is the orthogonal
projection onto the kernel of $\Box$ (which in this case  agrees with
$\ker\de$). This is a self-adjoint operator, so that, by Lemma
\ref{density-lemma}, $\S_0\Lambda^{0,q}=(\ker \Box \cap
\S_0\Lambda^{0,q})\oplus (\range \Box \cap \S_0\Lambda^{0,q}).$  The
commutation relation $\bar\de^*\Box=(\Box-iT)\bar \de^*$ from
\eqref{commutations} then implies that the two subspaces in this
decomposition are mapped under $\bar\de^*$ to $\ker (\Box-iT) \cap
\S_0\Lambda^{0,q}$ and $\range (\Box-iT) \cap \S_0\Lambda^{0,q},$
respectively. This shows that  
$$
\bar\de^*\cC\xi=P\bar\de^*\xi=0\,,
$$ 
where $P$ denotes the orthogonal projection onto the kernel of
$\Box-iT.$ Then  $\xi'=(I-\cC)\xi$ has the desired properties.  
\endproof

Set
\begin{equation}\label{Xpq-Ypq}
X^{p,q}=\{\xi\in W_0^{p,q}:C_p\xi=0\}\ ,\quad Y^{p,q}
=\{\eta\in W_0^{p,q}:\bar C_q\eta=0\}\ ,\quad Z^{p,q}=X^{p,q}\times Y^{p,q}\ .
\end{equation}

In combination with Proposition \ref{non-tr} the previous lemma  implies
that the spaces $Z^{p,q}$ provide parametrisations for the spaces
$W^{p,q}_{j,\ell}:$ 

\begin{cor}\label{cor5.7}
Assume that $j=1,2$ and $p+q+j+\ell\le n.$ Then the maps
$$
\begin{aligned}
e(d\theta)^\ell (\de\ \ \bar\de)&:Z^{p,q}\longrightarrow
W^{p,q}_{1,\ell}, \quad (\xi,\eta)\mapsto \de\xi+ \bar\de \eta,\\ 
e(d\theta)^\ell (\bar\de\de\ \ \de\bar\de)&:Z^{p,q}\longrightarrow
W^{p,q}_{2,\ell}, \quad (\xi,\eta)\mapsto \bar\de\de\xi+
\de\bar\de\eta 
\end{aligned}
$$
are bijections. Notice that this applies in particular to the spaces
$W^{p,q}_{j,\ell}$ appearing in the orthogonal decomposition of
$L^2\Lambda^k_H$ in Proposition \ref{horizdecomp} under the assumption
$k\le n.$ 
\end{cor}

\begin{remark}\label{projpar}
{\rm Recall that, by Lemma \ref{s4.8}, 
$$
X^{p,q}=\begin{cases}W_0^{p,q} &\text{ if }1\le
  p\le n-1\,,\\  
 \{0\}&\text{ if }p=n\,,\\
 \bigl\{\xi\in
 \S_0\Lambda^{0,q}:\mathcal C\xi= 0, \bar \de^*\xi=0\bigr\}&\text{ if }p=0\,.
 \end{cases}
 $$
 By the proof of Lemma \ref{modi}, the latter space is indeed nothing
 but $(I-\mathcal C)W_0^{0,q}.$
 
Analogous statements hold true for $Y^{p,q}$. 
Finally, notice that  the spaces $Z^{p,q}$ are non-trivial if $p+q\le
n-1$.} \medskip
 
\end{remark}

\subsection{The action of $\Delta_k$}

\ \medskip

Let $\Phi$ be the bijection \eqref{Phi} from $\S_0\Lambda^k_H$ onto
$(\S_0\Lambda^k)_{d^*\text{-cl}}$, and let $D_k=\Phi\inv \Delta_k
\Phi$ be the operator in \eqref{DefDk}. 

For
$\om\in\S_0\Lambda^k_H$, by  
 \eqref{1.12} 
 we have \smallskip 
\begin{equation}\label{Dk}
\aligned
D_k \om 
& = \bigl( \Delta_H -T^2 +e(d\theta)i(d\theta) \bigr)\om
+\bigl(T\inv [d_H^*, e(d\theta)] d_H^*\bigr)\om \smallskip \\
& = \bigl( \Delta_H -T^2  +T\inv d_H^* e(d\theta) d_H^*\bigr)\om \
.  \smallskip 
\endaligned
\end{equation}

The following identities are easily derived from \eqref{1.6},
\eqref{1.7} and \eqref{1.7bis}: 
\begin{equation}\label{box-e}
\begin{aligned}
&\Box e(d\theta)= e(d\theta)(\Box -iT)\, ,\\
&\Boxbar
e(d\theta) = e(d\theta)(\Boxbar +iT)\, ,\\
&[\Delta_H,\, e(d\theta)]=0\, .
\end{aligned}
\end{equation}

It follows from \eqref{1.10} that, when acting on $k$-forms,
\begin{equation}\label{box-bar}
\Box-\Boxbar=i(n-k)T\ .
\end{equation}

\begin{lemma}\label{D-e}
The following identities hold
\begin{itemize}
\item[(i)]$
D_k e(d\theta)= e(d\theta)(D_{k-2} +n-k+1)$;\smallskip
\item[(ii)]$
D_k e(d\theta)^\ell= e(d\theta)^\ell\big(D_{k-2\ell}
+\ell(n-k+\ell)\big)$, \ for $\ell\ge1$.
\end{itemize}
\end{lemma}

\proof
By \eqref{1.7}, \eqref{1.7bis} and  \eqref{Dk}, \eqref{box-e}, when applied to a
horizontal $(k-2)$-form, 
\begin{equation}\label{[D-k,edttheta]}
\aligned
D_k e(d\theta)
& = e(d\theta)(\Delta_H-T^2) +T\inv \bigl(d_H^* e(d\theta)\bigr)^2 \\
& = e(d\theta)(\Delta_H-T^2) +T\inv \bigl( e(d\theta)d_H^*
+i(\bar\de-\de)\bigr)^2 \\ 
& = e(d\theta)D_{k-2}  +iT\inv e(d\theta)  \bigl( d_H^*(\bar\de-\de)
+(\bar\de-\de)d_H^*\bigr) -T\inv (\bar\de-\de)^2 \\ 
& = e(d\theta)D_{k-2}  +iT\inv e(d\theta)(\Boxbar-\Box) -e(d\theta) \\ 
& = e(d\theta)(D_{k-2} +n-k+1)\, .
\endaligned
\end{equation}

Identity (ii) now follows by induction.
\endproof

\begin{prop}\label{invariance}
The subspaces $W_0^{p,q},\, W_{1,\ell}^{p,q},\, W_{2,\ell}^{p,q}$
are
invariant under the action of $D_k$.
\end{prop} 

\proof

If $\om\in W_0^{p,q}$, then $d_H^*\om=0$ and therefore
\begin{equation}\label{d0}
D_k\om=(\Delta_H -T^2)\om=(\Delta_0+i(q-p)T)\om\,,
\end{equation}
by \eqref{1.10}, where $\Delta_0$ denotes the scalar operator $L-T^2.$
 The last expression shows that $D_k\om$ is a
$(p,q)$-form, and the previous one that $d_H^*D_k\om=0$, by
\eqref{commutations}. 

By Lemma \ref{D-e}, when $j=1,2$, it suffices to take $\ell=0$.

\medskip

Take now $\om\in W_1^k$, $\omega=\de\xi+\bar\de\eta$, with
$\de^*\xi=\bar \de^*\xi=\de^*\eta= \bar \de^*\eta=0.$
We have
\begin{equation}\label{Dk-W1}
\aligned
D_k \om 
& = \bigl( \Delta_H -T^2  +T\inv d_H^* e(d\theta) d_H^*\bigr)
(\de\xi+\bar\de\eta) \\
& = \de (\Delta_H-T^2+iT)\xi +\bar\de (\Delta_H-T^2-iT)\eta 
+T\inv d_H^* e(d\theta) d_H^* (\de\xi+\bar\de\eta) \\
& = \de (\Delta_H-T^2+iT)\xi +\bar\de (\Delta_H-T^2-iT)\eta \\
& \qquad\qquad\qquad
+T\inv d_H^* e(d\theta) \Box\xi +T\inv d_H^* e(d\theta) \Boxbar\eta \\
& = \de (\Delta_H-T^2+iT)\xi +\bar\de (\Delta_H-T^2-iT)\eta 
+T\inv (i\bar\de-i\de) (\Box\xi +\Boxbar\eta) \\
& = \de \bigl( (\Delta_H-T^2+iT -iT\inv\Box)\xi
-iT\inv\Boxbar\eta\bigr)\\
& \qquad\qquad\qquad 
+\bar\de \bigl( (\Delta_H-T^2-iT+iT\inv\Boxbar)\eta +iT\inv\Box\xi \bigr) \
.  \smallskip 
\endaligned
\end{equation}

Therefore, $D_k(\de\xi+\bar\de\eta)=\de\xi'+\bar\de\eta'$, where
$$
\aligned
\xi' &= (\Delta_H-T^2+iT-iT\inv\Box)\xi -iT\inv\Boxbar\eta \\
\eta' & = (\Delta_H-T^2-iT+iT\inv\Boxbar)\eta +iT\inv\Box\xi \ ,
\endaligned
$$
that is,
\begin{equation}\label{formula*}
\aligned
D_k \begin{pmatrix} \de & \bar\de \end{pmatrix}
& = \begin{pmatrix} \de & \bar\de \end{pmatrix}
\begin{pmatrix} \Delta_H -T^2 +iT -iT\inv\Box & -iT\inv\Boxbar \\
iT\inv\Box & \Delta_H-T^2-iT+iT\inv\Boxbar\end{pmatrix} \smallskip
\\
& = \begin{pmatrix} \de & \bar\de \end{pmatrix}\bigg[
 (\Delta_H -T^2)I -iT\inv
\begin{pmatrix} \Box-T^2  &\Boxbar \\
-\Box & -\Boxbar+T^2\end{pmatrix} \bigg]
\ .
\endaligned
\end{equation}

Using the commutation relations \eqref{commutations}
  we see that
$$
\de^*\xi'=\bar \de^*\xi'=\de^*\eta'= \bar \de^*\eta'=0.
$$
Therefore, also $W_1^k$ is $D_k$-invariant. Moreover, if $\xi$ and
$\eta$ are $(p,q)$-forms, so are $\xi'$ and $\eta'$, hence each
$W^{p,q}_1$ is $D_k$-invariant. 
\medskip

Finally, take $\om\in W_2^k$, $\omega=\bar\de\de\xi+\de\bar\de\eta$, with
$\de^*\xi=\bar \de^*\xi=\de^*\eta= \bar \de^*\eta=0$. We first compute
$$
\aligned
d_H^*e(d\theta)d_H^*\bar\de\de\xi
&=d_H^*e(d\theta)(-\bar\de\Box\xi+\barBox\de\xi)\\
&=e(d\theta)d_H^*(-\bar\de\Box\xi+\barBox\de\xi)
+i(\bar\de-\de)(-\bar\de\Box\xi+\barBox\de\xi)\\
&=e(d\theta)(-\barBox\Box\xi+\de^*\barBox\de\xi)
+i(\de\bar\de\Box\xi+\bar\de\barBox\de\xi)\\
&=e(d\theta)\big(-\barBox\Box\xi+(\barBox+iT)\Box\xi\big)
+i\big(\de\bar\de\Box\xi+\bar\de\de(\barBox+iT)\xi\big)\\
&=iT e(d\theta)\Box\xi+i\de\bar\de\Box\xi
+i\bar\de\de(\barBox+iT)\xi\\
&=i\bar\de\de(-\Box+\barBox+iT)\xi \\
&=(n-k+1)T\bar\de\de\xi\ ,
\endaligned
$$
by \eqref{box-bar}, since $\xi$ is a $(k-2)$-form. Similarly,
$$
d_H^*e(d\theta)d_H^*\de\bar\de\eta=(n-k+1)T\de\bar\de\eta\ .
$$

Therefore, 
\begin{equation}\label{DW2}
D_k\om=(\Delta_H-T^2+n-k+1)\om\ .
\end{equation}

As before, \eqref{commutations} implies that $D_k\om\in W_2^k$, and each
subspace $W_2^{p,q}$ is mapped into itself. 

\endproof

\subsection{Lifting by $\Phi$}

\ \medskip

Denote by $V^{p,q}_0$, $V^{p,q}_{1,\ell}$, etc., the subspaces $\Phi
(W^{p,q}_0)$, $\Phi (W^{p,q}_{1,\ell})$, etc., of
$(L^2\Lambda^k)_{d^*\cl}$. We want to show that their closures $\V^{p,q}_0$, $\V^{p,q}_{1,\ell}$, etc. give an
orthogonal decomposition of $(L^2\Lambda^k)_{d^*\cl}$. 

In a way, this is not {\it a priori} obvious, because $\Phi$ is not an
orthogonal map. The fact that it preserves the orthogonality of the
subspaces we are working with is quite peculiar. On the other hand,
the reader may have noticed already an instance of this peculiarity in
the fact that a non-symmetric operator such as $D_k$ admits a rather
fine decomposition into invariant subspaces which are orthogonal. 

\begin{prop}\label{decom3}
For $0\le k\le n$ we have the orthogonal decompositions
\begin{equation}\label{decom-Vk}
\bigl( L^2\Lambda^k\bigr)_{d^*\cl} 
=  
\osum_{\substack{p+q=k<n \\ p+q=n,\, pq=0 }}\V^{p,q}_0
\oplus\osum_{\substack{j,\ell,p,q\\  j=1,2\\ p+q+j+2\ell=k}}\V^{p,q}_{j,\ell} \ ,
\end{equation}
where each of the subspaces 
$\V^{p,q}_0, \mathcal V_{j,\ell}^{p,q}$  
 is non-trivial and
$\Delta_k$-invariant.
\end{prop}

\proof Since $\Phi$ is a bijection from $\S_0\Lambda^k_H$ onto
$(\S_0\Lambda^k)_{d^*\cl}$, it follows from Proposition
\ref{horizdecomp} that  
$$
\bigl( \S_0\Lambda^k\bigr)_{d^*\cl} 
=  
\osum_{\substack{p+q=k<n \\ p+q=n,\, pq=0 }}V^{p,q}_0
\oplus\osum_{\substack{j,\ell,p,q\\j=1,2\\  p+q+j+2\ell=k}}V^{p,q}_{j,\ell} \ .
$$

Hence it remains to show that this decomposition is orthogonal. By
\eqref{Phi}, this amounts to proving that 
$$
d_H^*(W^{p,q}_{j,\ell})\perp d_H^*(W^{p',q'}_{j',\ell'})\mbox{
  whenever }  W^{p,q}_{j,\ell}\ne  W^{p',q'}_{j',\ell'}. 
$$
This, in turn,  is an immediate consequence of Corollary \ref{de*W}
and Proposition \ref{horizdecomp}. Notice that $d_H^*W^{p,q}_0=\{0\}.$

\endproof

\setcounter{equation}{0}
\section{Intertwining operators and different scalar forms for
  $\Delta_k$}\label{intws}   

Following the decomposition of $L^2\Lambda^k$ described
in the previous section, we continue assuming $0\le k\le n$.

In this section we describe 
the  form
that $\Delta_k$ attains on each of the subspaces
of the decomposition \eqref{decom-Vk}
of $(L^2 \Lambda^k)_{d^*\cl}$.  In particular, we will show that,
up to conjugation with invertible operators, 
$\Delta_k$ acts on 
$V_0^{p,q}$ and on each $V_{2,\ell}^{p,q}$ as a scalar operator.  For
$V_{1,\ell}^{p,q}$ instead, a further splitting will be necessary in
order to reduce $\Delta_k$ to a scalar form in a similar way.

In the process, we will also describe the
intertwining operators that reduce $\Delta_k$ to such scalar forms. 
\medskip

\subsection{The case of $V_0^{p,q}$}\quad
\medskip

The simplest case is the one of $V_0^{p,q}$, because $\Phi$ acts on this space as the
identity map and we already know by \eqref{d0} 
that $D_k\big|_{ W^{p,q}_0} =\Delta_0 +i(q-p)T$.  Hence, in this
case $\Delta_k$ is itself a scalar operator and we
simply 
have the following
\begin{prop}
Let $p+q\le n-1$ or, $pq=0$ if $p+q=n$.
Then, on 
$ V_{0}^{p,q}$,
\begin{equation}\label{j=0}
\Delta_k
= 
\Delta_0 +i(q-p)T    \ .
\end{equation}\medskip
\end{prop}

 \medskip

When we pass to $j=1,2$ we want to express $\Delta_k$ in terms of the
parameters $(\xi,\eta)$ in the definition of $W_{j,\ell}^{p,q}$ which
we can choose from the parameter spaces  
$Z^{p,q} = X^{p,q}\times Y^{p,q}.$

\subsection{The case of $V_{2,\ell}^{p,q}$}\quad
\medskip

 According to Corollary \ref{cor5.7}, we can write
\begin{equation}\label{w2par-ell}
W_{2,\ell}^{p,q} = 
\bigl\{ \om=e(d\theta)^\ell(\bar\de\de\xi+\de\bar\de\eta)\, :\, 
(\xi,\eta)\in Z^{p,q} \bigr\} .
\end{equation}
Recall from the discussion in Section \ref{firstproperties} and the
definitions of $X^{p,q}$ and $Y^{p,q}$ (see \eqref{Xpq-Ypq}) that $\Box$ is injective when
restricted to $X^{p,q}$ and  
$\Boxbar$ is injective when restricted to $Y^{p,q}$.

\begin{prop}\label{p6.1}
Let $A_{2,\ell} =\Phi e(d\theta)^\ell \bpm \de\bar\de &
\bar\de\de \epm: Z^{p,q}\rightarrow  V^{p,q}_{2,\ell}$. 
Then, $A_{2,\ell}$ is injective on $Z^{p,q}$.

The operator $\Delta_k$ restricted to  the subspace
$ V_{2,\ell}^{p,q}$ is given by the following expression:
\begin{equation}\label{j=2}
{\Delta_k }_{\big|_{V_{2,\ell}^{p,q}}}
= A_{2,\ell} \Bigl( 
\Delta_0 +i(q-p)T +(\ell+1)(n-k+\ell+1)  
\Bigr) A_{2,\ell}\inv
\, .
\end{equation}
\end{prop}
\proof
By Corollary \ref{cor5.7} it follows at once that $A_{2,\ell}$ is injective on $Z^{p,q}$.

When $k=p+q+2+2\ell$,  from Lemma \ref{D-e} we have 
$$
 D_k e(d\theta)^{\ell}
=e(d\theta)^{\ell} \Bigl(D_{k-2\ell} +\ell(n-k+\ell)  \Bigr).
$$
Moreover, by \eqref{DW2} we know that $D_{k-2\ell},$ when acting on
$W^{p,q}_2,$ is given by $\Delta_0+i(q-p)T+n-(k-2\ell)+1,$ so that on
$W^{p,q}_2$ 
\begin{equation}\label{action}
D_k e(d\theta)^\ell =e(d\theta)^{\ell}\Bigl( \Delta_0 +i(q-p)T +(\ell+1)(n-k+\ell+1)  
\Bigr).
\end{equation}

 By the definitions of $\Phi$ and $ V_{2,\ell}^{p,q}$, and the commutation relations \eqref{commutations}, this
 proves \eqref{j=2}.
\qed

\subsection{The case of $V_{1,\ell}^{p,q}$}\quad
\medskip

We now turn to the case $j=1$.  In this case the situation is quite
more involved, as we already observed in the case of $1$-forms, see
\cite{MPR}.  Let us begin by recalling that according to  Corollary \ref{cor5.7}, we can write
\begin{equation}\label{w2par-ell}
W_{1,\ell}^{p,q} = 
\bigl\{ \om=e(d\theta)^\ell(\de\xi+\bar\de\eta)\, :\, 
(\xi,\eta)\in Z^{p,q} \bigr\} .
\end{equation}

Consider the subspace $V^{p,q}_{1,\ell}=\Phi (W_{1,\ell}^{p,q})$.

Our next goal will be to formally diagonalize the matrix $
\begin{pmatrix} \Box-T^2  &\Boxbar \\
-\Box & -\Boxbar+T^2\end{pmatrix}$ appearing in
formula \eqref{formula*}. This matrix operator is acting on column
vectors $\begin{pmatrix} 
     \xi\cr
     \eta \cr
\end{pmatrix}$ corresponding to pairs 
 $(\xi,\eta)\in X^{p,q}\times Y^{p,q}=Z^{p,q},$  where $p+q+1+2\ell=k.$ We put

\begin{equation}\label{def-of-s}
s:=p+q=k-2\ell-1\ .
\end{equation}
Notice that $0\le s\le n-1.$

We define 
the operator matrix $Q$ acting on $\begin{pmatrix}
     \xi\cr
     \eta \cr
\end{pmatrix}$ by
\begin{equation}\label{Qmatrix}
Q=\bpm -Q^+_-&-Q^-_+\\ Q^+_+&Q^-_-\epm\ ,
\end{equation}
where, for $\eps,\delta=\pm$, 
the expression of $Q^\eps_\del$ is
$$
Q^\eps_\del=\Gamma+\eps m-\del iT\ ,
$$
where
\begin{equation}\label{m-Gamma}
\aligned
m & =\frac{n-s}2\ ,\\
\Gamma & =\sqrt{\Delta_H-T^2+m^2}\ .
\endaligned
\end{equation}

Observe here that the operator $\Delta_H-T^2+m^2$ satisfies the
estimate $\Delta_H-T^2+m^2\ge m^2\ge 1/4,$  so that it has a unique
positive square root.

The following identities are easily verified:
\begin{equation}\label{Q-iden}
\aligned 
Q^+_+ Q^-_-
& = 2\Box\\
Q^+_- Q^-_+
& = 2\Boxbar \\
Q^+_+ Q^-_+
& = 2\big[ \Box -T^2 -iT(m+\Gamma)\big] 
= 2\big[ \Boxbar -T^2 +iT(m-\Gamma)\big] \\
Q^-_- Q^+_-
& = 2\big[ \Box -T^2 -iT(m-\Gamma)\big] 
= 2\big[ \Boxbar -T^2 +iT(m+\Gamma)\big] \ ,
\endaligned
\end{equation}
since 
\begin{equation}\label{halfdelta}
\Box -imT = \Boxbar +imT=\half\Delta_H.
\end{equation}

\begin{lemma}\label{newpar}

If  $p+q\le n-1,$  then the following properties hold true: 
 \begin{itemize}
\item[(i)] The operator matrix 
$ Q: \S_0 \Lambda^{p,q}\times\S_0 \Lambda^{p,q}\rightarrow \S_0 \Lambda^{p,q}\times\S_0 \Lambda^{p,q}$ is invertible, with inverse 
$$
 Q\inv 
= \frac{1}{4iT\Gamma} 
\bpm - Q^-_- & - Q^-_+ \\  Q^+_+ &  Q^+_- \epm
\ .
$$
Moreover, $ Q$ maps the subspace $W_0^{p,q}\times W_0^{p,q}$ bijectively onto itself.

\item[(ii)] If $p=0,$ then $ Q^-_-\cC=\cC  Q^-_-=0,$ and if  $q=0,$ then $ Q^-_+\bar\cC=\bar\cC  Q^-_+=0.$
 \end{itemize}
\end{lemma}

\proof
To prove (i), we compute  formally the determinant of $ Q$ and  find
by \eqref{Q-iden} that 
$\text{det\,}  Q 
= -4iT\Gamma\ . 
$ The formula for $ Q\inv$ is now obvious. Notice also that the
operators $ Q^\eps_\delta$ leave the space $W^{p,q}_0$ invariant.  The
remaining statement in (i) is now clear. 
\medskip

As for (ii), notice that 
if $p=0$ then $\cC$ projects onto the kernel of $\de,$ which coincides
with the kernel of $\Box.$ And, on $\ker\Box,$ by \eqref{halfdelta} we
have $\Delta_H=-2i\,mT\ge 0,$ so that
$\Gamma=\sqrt{-2imT-T^2+m^2}=m-iT$, and hence $ Q^-_-=0$ on $\ker\de.$
This implies $\cC  Q^-_-= Q^-_-\cC=0.$ The remaining identities in
(ii) are proved analogously. 
\medskip
\endproof

We set
\begin{equation}\label{def-Xipq}
\Xi^{p,q}=X^{p,q}\cap Y^{p,q}=\{\xi\in W_0^{p,q}:C_p\xi=\bar C_q\xi=0\}\ ,\qquad 
\widetilde Z^{p,q}=W_0^{p,q}\times \Xi^{p,q}\ .
\end{equation}
\medskip

\begin{lemma}\label{Xi}
$$
\begin{pmatrix}
I-C_p & 0\\
0& I-\bar C_q
\end{pmatrix} Q (\widetilde Z^{p,q})=Z^{p,q}.
$$
\end{lemma}

\begin{proof}
It suffices to show that
$$
\begin{pmatrix}\label{newpar1}
I-C_p & 0\\
0& I-\bar C_q
\end{pmatrix}  Q (\widetilde Z^{p,q})
= \begin{pmatrix}
I-C_p & 0\\
0& I-\bar C_q
\end{pmatrix}  Q (W_0^{p,q}\times W_0^{p,q}),
$$
since $ Q (W_0^{p,q}\times W_0^{p,q})=W_0^{p,q}\times W_0^{p,q}.$

We have $\Xi^{p,q}=(I-C_p-\bar C_q)W_0^{p,q}$,
 which means that it suffices to show that
 $$
\begin{pmatrix}
I-C_p & 0\\
0& I-\bar C_q
\end{pmatrix}  Q 
\begin{pmatrix}0\cr \eta
  \end{pmatrix}
 =\begin{pmatrix}
    -(I-C_p) Q^-_+\eta \cr
    (I-\bar C_q) Q^-_-\eta 
    \end{pmatrix}
$$
is zero for every $\eta=(C_p+\bar C_q)\eta'$. This follows from the identities
$$
(I-C_p)(C_p+\bar C_q)=\bar C_q\ ,\qquad (I-\bar C_q)(C_p+\bar C_q)=C_p\ ,
$$
and from  (ii) of Lemma \ref{newpar}.
\end{proof}

\begin{lemma}\label{G-lemma}

Let 
\begin{equation}\label{G-def}
G= \bpm\Box-T^2&\barBox\\-\Box&-\barBox+T^2 \epm
\end{equation}
be the matrix appearing in formula \eqref{formula*}.
Then,
$-iT\inv G$ admits the diagonalization
$$
-iT\inv G= Q\bpm  m+\Gamma &0\\0& m- \Gamma \epm  Q\inv\ .
$$
\end{lemma}

\proof
In order to formally  compute the eigenvalues $\lambda_\pm$ of $-iTG$,
observe that the characteristic equation for $G$ is 
$$
\tau^2 -\tau\big[\Box-\Boxbar\big]+T^2\big[\Box+\Boxbar-T^2\big]=0\ ,
$$
which has roots 
$$
\tau_\pm
= imT\pm\sqrt{-T^2\big[\Box+\Boxbar-T^2+m^2\big]}=iT(m\pm\Gamma)\ .
$$
Therefore,
$$
G-\tau_\pm I
= \bpm \Box -T^2 -iT(m\pm\Gamma) & \Boxbar \\
-\Box & -\Boxbar+T^2 -iT(m\pm\Gamma) \epm \ ,
$$
and, by \eqref{Q-iden},
$$
\aligned
G-\tau_+ I
& = \half  \bpm Q^+_+ Q^-_+ & Q^+_-Q^-_+ \\
-Q^+_+Q^-_- & - Q^+_- Q^-_- \epm \\
& = \half\bpm  Q^-_+ \\ -Q^-_- \epm
\bpm Q^+_+ & Q^+_- \epm\ ,
\endaligned
$$
and analogously,
$$
G-\tau_-I
 = \half \bpm  Q^+_- \\ -Q^+_+ \epm
\bpm Q^-_- & Q^-_+ \epm\ .
$$ 

These equations show that eigenvectors of $G$ of
eigenvalues $\tau_\pm$ are given, respectively, by 
\begin{equation}\label{Q_pm}
Q^+ =\bpm  -Q^+_- \\ Q^+_+ \epm\ ,\qquad Q^- =\bpm -Q^-_+ \\ Q^-_- \epm\ , 
\end{equation}
so that
$$
Q= \big( Q^+ | Q^-\big) 
$$
is indeed a matrix which formally diagonalizes $-iT\inv G$ as claimed.
\qed

Recall now from \eqref{w2par-ell} and the definition of
$V^{p,q}_{1,\ell}$ that if  we define the operator 
$A_{1,\ell}:(W^{p,q}_0)^2\rightarrow  L^2\Lambda^k$ as
\begin{equation}\label{i1}
\begin{aligned}
A_{1,\ell}\bpm\xi\\ \eta\epm
&:=\Phi e(d\theta)^\ell
\bpm \de&\bar\de\epm
\bpm\xi\\ \eta\epm\\ 
&=\bpm I\\ T\inv
  d_H^*\epm e(d\theta)^\ell\bpm
  \de&\bar\de\epm\bpm\xi\\ \eta\epm\ ,
\end{aligned}
\end{equation}
then $A_{1,\ell}(Z^{p,q})=V^{p,q}_{1,\ell}.$ 
Observe also that Lemma \ref{Xi} shows that we may
realize $Z^{p,q}$ in this identity as the space
$Q(\widetilde Z^{p,q})$ and use $\widetilde Z^{p,q}$ as a parameter space for $V^{p,q}_{1,\ell}$. This has the advantage of reducing the operator $D_k$ in \eqref{formula*} to diagonal form.

We therefore
define the modified intertwining operator $\mathcal A_{1,\ell}$ by  
\begin{equation}\label{calA1ell}
\mathcal A_{1,\ell}
: = A_{1,\ell} Q_{\big| \widetilde Z^{p,q}}:\widetilde Z^{p,q}\to V_{1,\ell}^{p,q}\ .
\end{equation}

By
\eqref{formula*},  Lemma \ref{D-e}  and  Lemma \ref{G-lemma}, we have  
\begin{equation}\label{calA1}
\Delta_k\mathcal A_{1,\ell}=\mathcal A_{1,\ell}\Big
(\Delta_h-T^2+\ell(n-k+\ell)+\begin{pmatrix} 
m+\Gamma & 0\cr
0 & m-\Gamma
\end{pmatrix}\Big).
\end{equation}
This suggests to further introduce the operators $\cA^{\pm}_{1,\ell}$ acting by 
\begin{equation}\label{A1pm}
\cA^+_{1,\ell}\xi=\mathcal A_{1,\ell}\begin{pmatrix} \xi \\ 0
   \end{pmatrix}=A_{1,\ell}Q^+\xi\ , \quad \cA^-_{1,\ell}\eta=\mathcal A_{1,\ell}\begin{pmatrix} 0 \\ \eta
   \end{pmatrix}=A_{1,\ell}Q^-\eta\ ,
\end{equation}
with $Q^\pm$ as in \eqref{Q_pm}.
The following proposition is then immediate.
\begin{prop}\label{----}

The space $V_{1,\ell}^{p,q}$
 decomposes as the direct sum 
$$
V_{1,\ell}^{p,q}=\A_{1,\ell}^+(W_0^{p,q}) + \A_{1,\ell}^-(\Xi^{p,q})\ .
$$ 

Moreover, the linear mappings 
$$
\A_{1,\ell}^+: W_0^{p,q}\to \A_{1,\ell}^+(W_0^{p,q})  , \quad \A_{1,\ell}^-: \Xi^{p,q}\to \A_{1,\ell}^-(\Xi^{p,q})
$$
are bijective, and the following identities hold, on $W_0^{p,q}$ and $\Xi^{p,q}$ respectively: 
\begin{equation}\label{diagonal}
\begin{aligned}
(\A^+_{1,\ell})\inv\Delta_k    \A^+_{1,\ell}&=
 L-T^2 +i(q-p)T +\ell(n-k+\ell)+ m \\
&\qquad\qquad\qquad\qquad+\sqrt{L-T^2 +i(q-p)T +m^2}\ , \\
(\A^-_{1,\ell})\inv \Delta_k\A^-_{1,\ell} &=  
 L-T^2 +i(q-p)T +\ell(n-k+\ell)+ m \\
&\qquad\qquad\qquad\qquad-\sqrt{L-T^2 +i(q-p)T +m^2} \ .
\end{aligned}
\end{equation}
\end{prop}

Define
\begin{equation}\label{V1pm-ell}
V_{1,\ell}^{p,q,+} = \A_{1,\ell}^+(W_0^{p,q})\ , \quad V_{1,\ell}^{p,q,-} = \A_{1,\ell}^-(\Xi^{p,q}).
\end{equation}

It should be stressed that up to this point we have not yet shown that
the subspaces 
$V_{1,\ell}^{p,q,+}$ and $V_{1,\ell}^{p,q,-}$ 
 are mutually orthogonal.
 This fact will be a consequence of the
 analysis of the intertwining 
operators of the next section, see Lemma \ref{new-matrix-R}.

\bigskip

\setcounter{equation}{0}
\section{Unitary intertwining operators and projections}\label{intws}

The intertwining operators for $\Delta_k$ that we have defined in the
previous section where non-unitary and unbounded. In order to verify
that the forms to which $\Delta_k,$  
when restricted to the subspaces $V_0^{p,q}, V_{1,\ell}^{p,q,\pm}$ and
$V_{2,\ell}^{p,q},$ had been reduced on the corresponding parameter
spaces by means of the formulas \eqref{j=0}, \eqref{diagonal}  and
\eqref{j=2}  are indeed describing the  spectral theory of $\Delta$ on
these subspaces, we need to replace the previous intertwining
operators by unitary ones. Our next tasks will therefore be the
following ones:  

\medskip
\bee
\item replace these intertwining operators with unitary ones; 
\item determine the orthogonal projections from $L^2 \Lambda^k$ onto  
$\mathcal V^{p,q}_0, \mathcal V_{1,\ell}^{p,q,\pm}$ and  $\mathcal
V_{2,\ell}^{p,q},$ the $L^2$-closures of the invariant subspaces 
$V_0^{p,q}, V_{1,\ell}^{p,q,\pm}$ and $V_{2,\ell}^{p,q}.$
\ee

These two tasks can be accomplished simultaneously by making use of
the polar 
decomposition of the intertwining operators.

We shall repeatedly use the following basic fact from spectral theory
(compare \cite{RS} for the case $H=K$).  
\begin{prop}\label{rs}
Let $H,K$ be Hilbert spaces and $A: \dom A \subset
H\rightarrow K$ be a densely defined, closed 
operator.
Then there exist  a positive self-adjoint operator $|A|:\, \dom
A\subset H \rightarrow H,$  with $\dom |A|=\dom A,$ and a partial
isometry $U:\, H\rightarrow K$ with  
 $\ker U=\ker A$ and $\range U=\overline{\range A},$ so that $A=U|A|.$
 $|A|$ and $U$ are uniquely determined by these properties together
 with the additional condition $\ker |A|=\ker A.$  

Moreover, $|A|=\sqrt{A^*A}$, 
 $U^*U$ is the orthogonal projection from $H$ onto
$(\ker A)^\perp= \overline{\range A^*}$, and $UU^*$ is the
orthogonal projection from $K$ onto $\overline{\range A}=(\ker
A^*)^\perp$.
\end{prop}

In order to pass from a possibly unbounded intertwining operator to a
unitary one, we also need the following general principle.

\begin{prop}\label{detlef}
Let $H_1,H_2$ be Hilbert spaces and let $\cD_1\subset H_1$, 
$\cD_2\subset H_2$ be dense subspaces.  Assume that for $j=1,2$,
 $S_j:\dom S_j\subset H_j\rightarrow H_j$ is a  self-adjoint operator
 on $H_j$ for which $\cD_j$ is a core such that $S_j(\cD_j)\subset
 \cD_j.$  Moreover, let $A: \dom A\subset  H_1\to H_2$ be a 
closed operator  such that the following properties hold true:

 \begin{itemize}
\item[(i)]  $\cD_1\subset\dom A$  and $A(\cD_1)\subseteq\cD_2;$ 
\item[(ii)] $A$ intertwines $S_1$ and $S_2$ on the core $\cD_1,$ i.e., 
\begin{equation}\label{str-intertwines1}
AS_1\xi=
S_2 A\xi\qquad\text{for all\ }\xi\in\cD_1\, .
\end{equation}
\end{itemize}
Consider the polar decomposition $A=U|A|$
from Proposition\ref{rs}, where  $|A|=\sqrt{A^*A},$  and where $U:H_1\rightarrow H_2$ is a partial isometry, and  assume furthermore that $\cD_1\subset \dom |A|,$ and that 
 \begin{itemize}
\item[(iii)] $|A|(\cD_1)=\cD_1;$  
\item[(iv)] the  commutation  relation
\begin{equation}\label{str-intertwines3}
S_1|A|\xi=|A| S_1\xi \qquad\text{for all\ }\xi\in\cD_1
\end{equation}
holds true on the core $\cD_1.$
 \end{itemize}
Then, also $U$ intertwines $S_1$ and $S_2$ on the core $\cD_1,$ i.e., 
 $U(\cD_1)=A(\cD_1)\subset \cD_2, $ and
\begin{equation}\label{str-intertwines4}
US_1 \xi =S_2U\xi \qquad\text{for all\ }\xi\in\cD_1\, .
\end{equation}

Moreover, we have $\overline{\range A}=\overline{A(\cD_1)}=U(H_1),$ $\ker A=\ker |A|=\ker U,$ 
and $P:=UU^*$ is the orthogonal projection  from $H_2$ onto $\overline{A(\cD_1)}.$
\medskip 

Let us finally denote by $S_2^r={S_2}_{\big|_{\overline{A(\cD_1)}}}$ the restriction of  $S_2$ to $\overline{A(\cD_1)}$, with domain $\dom S_2^r:=\dom S_2\cap \overline{A(\cD_1)}.$ 
If we assume in addition that
 \begin{itemize}
\item[(v)] $\ker |A|=\{0\}\,;$
\item[(vi)]  $(I-iS_1)^{-1}(\cD_1)\subseteq \cD_1\, ;$
\item[(vii)] $P(\cD_2)=A(\cD_1)\, ,$
 \end{itemize}
 then $U$ is injective, and we even have that $U(\dom S_1)=\dom S_2^r,$ and
  $$
 S_2^r=US_1U^{-1} \quad \mbox{on}\quadÊ\dom S_2^r.
 $$
\end{prop}

\proof
Let us  re-write \eqref{str-intertwines1} as
$$
U|A|S_1\xi=
S_2 U|A|\xi\qquad\text{for all\ }\xi\in\cD_1\, . 
$$
Applying \eqref{str-intertwines3}, we find that 
$$
US_1(|A|\xi)=S_2U(|A|\xi)\qquad\text{for all\ }\xi \in\cD_1\, ,
$$
which  implies \eqref{str-intertwines4} because of (iii). Note that $U(\cD_1)=A(\cD_1)\subset \cD_2$ in view of (iii).

Since $A$ is closed and $\cD_1$ is a core for $A,$ we have $\overline{\range A}=\overline{A(\cD_1)},$ and the remaining statements about $\overline{\range A},$ $\ker A$ and $UU^*$ are obvious by Proposition \ref{rs}.

If we assume in addition that (v) and (vi) hold true, then clearly $U$ is injective. Moreover, \eqref{str-intertwines4} implies that
$$
U(I-iS_1)\xi=(I-iS_2)U\xi \qquad\text{for all\ }\xi\in\cD_1.
$$
Since $U(\cD_1)\subseteq \cD_1,$ by (vi) we then obtain that 
$
U(I-iS_1)^{-1}\xi=(I-iS_2)^{-1}U\xi 
$
for every $\xi\in\cD_1,$ hence 
\begin{equation}\label{resolv}
U(I-iS_1)^{-1}=(I-iS_2)^{-1}U
\end{equation}
on $H_1.$  Noticing that $\dom S_j=\range(I-iS_j)^{-1},$ \eqref{resolv} implies that $U(\dom S_1)\subseteq \dom S_2,$ so that $U(\dom S_1)\subseteq \dom S_2^r,$ and that \eqref{str-intertwines4} holds true even for every $\xi\in \dom S_1:$
\medskip

If $x=(I-iS_1)^{-1}y\in \dom S_1$ (with $y\in H_1$), then $Ux=(I-iS_2)^{-1}Uy\in\dom S_2,$ and 
$$
(I-iS_2)Ux=Uy=U(I-iS_1)x.
$$
It therefore only remains to show that $\dom S_2^r\subseteq U(\dom S_1).$

To this end, we first observe that, because of (vii) and \eqref{str-intertwines1},
$$
S_2P(\cD_2)\subseteq S_2A(\cD_1)\subseteq  A S_1(\cD_1)\subseteq A(\cD_1)=P(\cD_2).
$$
Since $S_2$ is self-adjoint, this implies 
$$
S_2P(\cD_2) \subseteq P(\cD_2) \quad \mbox{and} \quad (S_2(I-P)(\cD_2) \subseteq (I-P)(H_2).
$$

Assume  now that $x\in \dom S_2\cap \range U.$ Then $x=Uy$ for some unique $y\in H_1.$ Choose a sequence $\{x_n\}_n$ in $\cD_2$ such that 
$$
x_n\to x \quad \mbox{and} \quad S_2x_n\to S_2 x.
$$
Since  $S_2x_n=S_2(Px_n)+S_2 ((I-P)x_n)),$ where the components in this decomposition lie in mutually orthogonal spaces, we see that there is some $z=Uw\in P(\cD_2)\subset U(H_2)$ such that
$$
Px_n\to x=Uy \quad \mbox{and} \quad S_2(Px_n)\to z=Uw.
$$
We can write $P x_n$ in a unique way as $P x_n=Uy_n,$ with $y_n\in\cD_1,$ since $U(\cD_1)=A(\cD_1)=P(\cD_2).$ Since $U$ is isometric on $H_1,$ we then must have 
that $y_n\to y.$ Moreover, by \eqref{str-intertwines4},
 $US_1y_n=S_2Uy_n=S_2(Px_n)\to z,$  so that $S_1y_n\to w.$ This shows that $y\in\dom S_1,$ hence $x=Uy\in U(\dom S_1).$

\qed

\begin{remark}\label{nelson}
{\rm 

If we do not require that the crucial commutation relation  in (iv) is satisfied, but that in addition to the conditions (i) to (iii) the natural assumptions 
$\cD_2\subset\dom A^*$ and  $A^*(\cD_2)\subset \cD_1$ hold true, then one can conclude that 
\begin{equation}\label{str-intertwines2}
S_1|A|^2\xi=|A|^2 S_1\xi \qquad\text{for all\ }\xi\in\cD_1\, .
\end{equation}

Indeed, then for  $\xi\in\cD_1$ and $\eta\in\cD_2,$ \eqref{str-intertwines1} and (i) imply that 
$
\lan \xi, S_1A^*\eta\ran=\lan\xi,A^* S_2\eta\ran,
$
hence 
$$
S_1A^*\eta=A^*S_2\eta \qquad\text{for all\ }\eta \in\cD_2\, .
$$
Combining this with \eqref{str-intertwines1}, we obtain $
S_1A^*A\xi=A^*AS_1
$
for every $\xi \in\cD_1\,,$  which verifies \eqref{str-intertwines2}. 
\medskip

One might hope that  \eqref{str-intertwines3} would follow from
\eqref{str-intertwines2}  by means of general spectral theory. 
However, this hope is destroyed by a classical example due to Nelson
(cf. \cite{RS}), which  shows that condition \eqref{str-intertwines2}
will in general not suffice to conclude that the operators $S_1$ and
$|A|^2$ commute, in the sense that their respective spectral
resolutions commute. This, however, would be needed in order to derive
\eqref{str-intertwines3}. 

However, in our applications, $S_1$  will turn out to be a scalar
operator on the Heisenberg group, and $A$ a positive square matrix
whose entries are scalar operators too, so that
\eqref{str-intertwines3}  will easily follow from formula
\eqref{cayley-hamilton} for the square root of such a matrix.

}
\end{remark}

In the sequel, by $P_{H_1}:H\to H_1$ we shall denote  the orthogonal
projection from the Hilbert space $H$ onto its closed subspace $H_1.$ 

In our later applications of Proposition \ref{rs}, the next
observation will often facilitate the computation of the corresponding
operators $A^*A.$ 
\begin{lemma}\label{rmrk}
Let
 $H,K$ be Hilbert spaces and  $H_1\subseteq H$ and $K_1\subseteq K$ be
closed subspaces.  Let 
$A: \dom A\subset H \rightarrow K$
be a densely defined, closed operator, and assume that $\cD\subset
\dom A$ is a core for $A.$  Assume furthermore that $\cD_1:=\cD\cap
H_1$ is dense in $H_1$ and  that $\dom A_1:=\dom A\cap H_1$ is mapped
under $A$ into $K_1,$  
  so that  the operator  $A_1:\dom A_1\subset H_1\to K_1,$ given by
  restricting $A$ to $\dom A_1:=\dom A\cap H_1,$  is densely defined
  and closed. 

 Under these conditions,  also  $A^*$ is densely defined, and $\dom
 A^*\cap K_1\subset \dom A_1^*.$ We shall further assume that
 $\cE\subset K$ is a subspace of  $\dom A^*$ such that $A(\cD)\subset
 \cE$ and $A^*(\cE)\subset \cD$ (so that, in particular,  
 $\cE_1:=\cE\cap K_1$ is contained in $\dom A_1^*$).
Then we have 
$$
A_1^* A_1\xi = P_{H_1} A^* A\xi \qquad\text{for all\ }\xi \in\cD_1\, .
$$
In
particular, if we know that $A^*A$ maps $\cD_1$
 into $H_1$, then $A_1^* A_1 \xi= A^* A\xi$  for every $\xi \in\cD_1.$
\end{lemma}
\proof
Since $A(\cD_1)\subset \cE_1,$  it suffices to prove that 
$
A_1^*  = P_{H_1} A^*
$
 on $\cE_1.$
But, if $x\in \cD_1\subset\dom A_1, \xi\in\cE_1\subset\dom A^*_1,$ then 
\begin{eqnarray*}
\lan x,A_1^*\xi\ran=\lan A_1x,\xi\ran=\lan A x,\xi\ran
=\lan x,A^*\xi\ran=\lan x, P_{H_1}A^*\xi\ran.
\end{eqnarray*}
This implies that $A_1^*\xi=P_{H_1}A^*\xi,$ since $\cD_1$ is dense in $H_1.$
\qed

\bigskip

\subsection{A unitary intertwining operator for $V_0^{p,q}$}\quad
\medskip

We recall from the preceding  discussion  that the intertwining
operator on  $\mathcal V_{0}^{p,q}$ is $\Phi$, which reduces to the 
identity on this space. 
Hence, this case is trivial.

\medskip

\subsection{Unitary intertwining operators for $V_{1,\ell}^{p,q,\pm}$}\quad
\medskip

Our next goal is to replace the intertwining operators
$\A_{1,\ell}^{\pm}$ from Proposition \ref{----} by unitary
ones. 
Recall from 
Proposition \ref{----} 
and \eqref{A1pm} that 
$$
\A_{1,\ell}^{\pm}=A_{1,\ell}Q^\pm\ ,\qquad \dom \A_{1,\ell}^+=W_0^{p,q}\ ,\quad \dom \A_{1,\ell}^-=\Xi^{p,q}\ ,
$$
 where, according to \eqref{i1}, 

\begin{multline}\label{i4}
A_{1,\ell}
=\bpm e(d\theta)^\ell\de
  &e(d\theta)^\ell\bar\de\\ 
i\ell T\inv e(d\theta)^{\ell-1}\bar\de\de+T\inv e(d\theta)^\ell\Box
&  -i\ell T\inv e(d\theta)^{\ell-1}\de\bar\de + T\inv
e(d\theta)^\ell\barBox\epm\ .
\end{multline}

According to Proposition \ref{detlef}, we seek to define unitary intertwining operators 
$U_{1,\ell}^{\pm}$ by defining 
\begin{equation}\label{u1pm}
U_{1,\ell}^\pm:=U_{1,\ell}^{p,q,\pm}
:=\A_{1,\ell}^{\pm}\Big((\cA_{1,\ell}^\pm)^*\cA_{1,\ell}^\pm\Big)^{-\half},
\end{equation}
which are expected  to be isometries from the closed subspaces 
$\overline{W_0^{p,q}}=\W^{p,q}$, resp. $\overline{\Xi^{p,q}}$, onto
their ranges $\cV_{1,\ell}^{p,q,\pm}.$ Recall, however, that we have
not  shown yet that the latter spaces are mutually orthogonal; this
will in fact follow easily from the subsequent discussions. 

Now, since 
$$
(\cA_{1,\ell}^\pm)^*\cA_{1,\ell}^\pm={Q^\pm}^*(A_{1,\ell}^*A_{1,\ell})Q^\pm,
$$
we shall begin by computing $A_{1,\ell}^*A_{1,\ell}$. 

Subsequently,
we will compute the product $Q^*A_{1,\ell}^*A_{1,\ell}Q$, showing, in
particular, that it is a diagonal matrix. The diagonal terms will give
the explicit forms of $(\cA_{1,\ell}^\pm)^*\cA_{1,\ell}^\pm$, whereas
the vanishing of the off-diagonal terms will prove the orthogonality
of the spaces $\cV_{1,\ell}^{p,q,\pm}$.  Since these computations are
tedious and unenlightening, we shall only state here the relevant
identities, postponing their proofs to the Appendix. 

Let us  set, for $s+j\le n$,
\begin{equation}\label{csj}
c_{s,j}=\frac{j!(n-s)!}{(n-s-j)!}\, .
\end{equation}
\medskip

\begin{lemma}\label{ip4}
We have that $A_{1,\ell}^*A_{1,\ell} = -c_{s+1,\ell}T^{-2}N$, where
\begin{equation}\label{Nmatrix}
N=
\bpm
\Box(\Box-i\ell T -T^2) & \Box\Boxbar \\
\Box\Boxbar & \Boxbar(\Boxbar+i\ell T- T^2) 
\epm\ .
\end{equation}
\end{lemma}

\begin{lemma}\label{new-matrix-R}
Let $R=-T^{-2}Q^*NQ$ on $(W_0^{p,q})^2$.  Then
$$
R=\begin{pmatrix}R_{11}&0\\0&R_{22}\end{pmatrix}\ ,
$$
where
$$
\begin{aligned}
R_{11}
& =(\Gamma+m)^2\big(\Delta_H+
2m(2m-\ell)\big)+2(\Gamma+m)\big((2m-\ell)\Delta_H-2mT^2\big)\\  
&\qquad\qquad\qquad\qquad\qquad\qquad\qquad\qquad\qquad\qquad
+\Delta_H(\Delta_H-T^2)+2m\ell T^2\ ,
\end{aligned}
$$
maps $W_0^{p,q}$ bijectively onto itself, and
$$
\begin{aligned}
R_{22}&=(\Gamma-m)^2\big(\Delta_H+2m(2m-\ell)\big)
-2(\Gamma-m)\big((2m-\ell)\Delta_H-2mT^2\big) \\
&\qquad\qquad\qquad\qquad\qquad\qquad\qquad\qquad\qquad\qquad
+\Delta_H(\Delta_H-T^2)+2m\ell T^2\ ,
\end{aligned}
$$
maps $ \Xi^{p,q}$ bijectively onto itself, and is zero on
$C_pW_0^{p,q}\oplus \bar C_qW_0^{p,q}$, the orthogonal complement of
$\Xi^{p,q}$ in $W_0^{p,q}$. 
\end{lemma}

\begin{proof}
The proof of the formulas for the components of $R$ is postponed to
the Appendix. Given these formulas, we prove here the mapping
properties of $R_{11}$ and $R_{22}$. 

On $W_0^{p,q}$, $R_{11}$ acts as a symmetric scalar operator. Since
$\Delta_H=L+i(q-p)T$, $\Gamma$ and $-T^2$ are positive operators, 
we have
$$
\begin{aligned}
R_{11}&\ge (\Gamma+m)^2\big(\Delta_H+2m(2m-\ell)\big)-4m^2T^2+2m\ell T^2\\
&=(\Gamma+m)^2\big(\Delta_H+2m(2m-\ell)\big)-2m(2m-\ell)T^2\\
&\ge2m^3(2m-\ell)>0\ .
\end{aligned}
$$

It follows that the operators $(R_{11})_{\la,\sigma}$ in \eqref{pi(B)}
also satisfy the same inequality from below, and hence are
invertible. Applying Lemma \ref{s3.1} (ii), we obtain that $R_{11}$
admits an inverse $R_{11}\inv:\S_0\longrightarrow \S_0$. 

We tensor with $\Lambda^{p,q}$ and restrict $R_{11}\inv$ to
$W_0^{p,q}$.  By \eqref{commutations}, the composition
  $\de^*R_{11}$ can be expressed as $R'_{11}\de^*$, with $R'_{11}$
  differing from $R_{11}$ in that $\Delta_H$ is replaced by
  $\Delta_H-iT$ (also in the expression of $\Gamma$), and similarly
  for $\bar\de^*R_{11}$, $\de^*R_{11}\inv$ and $\bar\de^*R_{11}\inv$. 
Therefore, $R_{11}$ maps  $W_0^{p,q}$ bijectively onto itself.

As to $R_{22}$, we first observe that
\begin{equation}\label{detR}
\begin{aligned}
R_{11}R_{22}&=\det R\\
&=T^{-4}(\det Q)^2\det N\\
&=T^{-4}(-4iT\Gamma)^2(-T^2)\big(\Delta_H-T^2+\ell(2m-\ell)\big)\Box\Boxbar\\
&=16 (\Delta_H-T^2+m^2) 
\big(\Delta_H-T^2+\ell(2m-\ell)\big)\Box\Boxbar\ ,
\end{aligned}
\end{equation}
so that 
\begin{equation}\label{R22-R11inv}
R_{22}=16 (\Delta_H-T^2+m^2)
\big(\Delta_H-T^2+\ell(2m-\ell)\big)\Box\Boxbar R_{11}\inv\ . 
\end{equation}

Moreover, by the injectivity of $R_{11}$,
$$
\ker R_{22}=\ker R_{11}R_{22}=\ker\Box\oplus \ker\Boxbar\ .
$$

In order to repeat the same argument used above for $R_{11}$, we start
from the operator $\widetilde
R_{22}=R_{22}+\del_{p,0}\cC+\del_{q,0}\bar\cC$ (with $\del$ denoting
the Kronecker symbol) acting on scalar-valued functions. By
\eqref{R22-R11inv}, $\widetilde R_{22}$ in invertible on $\S_0$ and,
after tensoring and restricting, it is also invertible on
$W_0^{p,q}$. For $\xi\in \Xi^{p,q}$, 
$$
\widetilde R_{22}\inv R_{22}\xi=\widetilde R_{22}\inv\widetilde R_{22}\xi=\xi\ .
$$

The conclusion now follows at once.
\end{proof}

\begin{cor}\label{V1-orth}
We have  that
$$
\cA_{1,\ell}^*\cA_{1,\ell}=c_{s+1,\ell}  R_{|_{\widetilde Z^{p,q}}}\ .
$$

In particular,
$$
(\cA_{1,\ell}^+)^*\cA_{1,\ell}^+=c_{s+1,\ell} R_{11}\ ,\qquad 
(\cA_{1,\ell}^-)^*\cA_{1,\ell}^-=c_{s+1,\ell}{R_{22}}_{ |_{\Xi^{p,q}}}\ ,
$$
and the subspaces $\V_{1,\ell}^{p,q,+}$ and $\V_{1,\ell}^{p,q,-}$ are orthogonal.
\end{cor}

\proof
Obviously, $R$ maps the subspace $\widetilde Z^{p,q}$ of $(W_0^{p,q})^2$ into itself, so that the identities follow from Lemma \ref{new-matrix-R} and Lemma \ref{rmrk}.
  The first statements are obvious. And, 
since the matrix $Q^*NQ$ is diagonal, so is 
$\A_{1,\ell}^*\A_{1,\ell}$.  Thus, the map $\A_{1,\ell}$ preserves the
orthogonality of the coordinate subspaces $W_0^{p,q}\times\{0\}$ and 
$\{0\}\times \Xi_0^{p,q}$.
\endproof

Let us finally compute $U_{1,\ell}^\pm$ more explicitly. To this end, notice that if we combine the column-vectors of operators $U_{1,\ell}^\pm$ to form a  square matrix, 
then
\begin{equation}\label{u1lpmp}
U_{1,\ell}:=\Big (U_{1,\ell}^+\quad U_{1,\ell}^-\Big)=\cA_{1,\ell}(\cA_{1,\ell}^*\cA_{1,\ell})^{-\half}.
\end{equation}

  Recall that we have set $s=p+q$ and $m=(n-s)/2$.

\begin{prop}\label{U1ell-prop}
We have that
\begin{equation}\label{U1ell-eq}
U_{1,\ell} = c_{s+1,\ell}^{-\half} \, e(d\theta)^{\ell-1}
\bpm S_{11} & S_{12}\\ S_{21} & S_{22} \epm 
\bpm
\Sigma_{11}& 0\\ 0&\Sigma_{22}
\epm\ ,
\end{equation}
where 
\begin{equation}\label{S-components}
\begin{aligned}
S_{11}&=e(d\theta)(-\de Q^\Piu_-  +\bar\de Q^\Piu_+)\\
S_{12}&=e(d\theta)(-\de Q^\Meno_+  +\bar\de Q^\Meno_-) \\
S_{21}&=\ell(\bar\de\de-\de\bar\de)-ie(d\theta)\big[\Delta_H
+(\Gamma+m)(2m-\ell)\big] \\ 
S_{22}&=
-\ell(\bar\de\de-\de\bar\de)+ie(d\theta)\big[\Delta_H -(\Gamma-m)(2m-\ell)\big]
\end{aligned}
\end{equation}
and $\Sigma_{11},\Sigma_{22}$ are given by 
\begin{equation}\label{Sigma-components}
\begin{aligned}
\Sigma_{11}=R_{11}^{-\half}&=\Big[(\Gamma+m)^2\big(\Delta_H+
2m(2m-\ell)\big)+2(\Gamma+m)\big((2m-\ell)\Delta_H-2mT^2\big)\\  
&\qquad+\Delta_H(\Delta_H-T^2)+2m\ell T^2\Big]^{-\half}\\
\Sigma_{22}=R_{22}^{-\half}&=\Big[(\Gamma-m)^2\big(\Delta_H+2m(2m-\ell)\big)
-2(\Gamma-m)\big((2m-\ell)\Delta_H-2mT^2\big) \\
&\qquad+\Delta_H(\Delta_H-T^2)+2m\ell T^2\Big]^{-\half} \ .
\end{aligned}
\end{equation}
\end{prop}

\proof
From Corollary \ref{V1-orth}, we have  
$$
(\cA_{1,\ell}^*\cA_{1,\ell})^{-\half}=c_{s+1,\ell}^{-\half}R^{-\half}\ ,
$$
so that
$$
\begin{aligned}
U_{1,\ell}
&=A_{1,\ell}Q(\cA_{1,\ell}^*\cA_{1,\ell})^{-\half}=c_{s+1,\ell}^{-\half} A_{1,\ell}QR^{-\half}\\
&=c_{s+1,\ell}^{-\half} \bpm I\\ T\inv
  d_H^*\epm e(d\theta)^\ell\bpm
  \de&\bar\de\epm 
\bpm -Q^\Piu_-&-Q^\Meno_+\\
  Q^\Piu_+&Q^\Meno_-\epm \bpm {(R^+)^{-\half}} & 0 \\ 0&
 (R^-)^{-\half} \epm
\ . 
  \end{aligned}
  $$
  
  We verify that the factor $T\inv$ in the second row is going to disappear. From Lemma \ref{de*}, we have
  $$
  T\inv d_H^*e(d\theta)^\ell (\de\ \ \bar\de)
=T\inv e(d\theta)^\ell \bpm\Box &\barBox\epm+i\ell T\inv e(d\theta)^{\ell-1} 
\bpm\bar\de\de& -\de\bar\de\epm \ .
$$
Let us define the matrix $S$ by requiring that
$$
e(d\theta)^{\ell-1}S=\bpm I\\ T\inv
  d_H^*\epm e(d\theta)^\ell\bpm
  \de&\bar\de\epm 
\bpm -Q^\Piu_-&-Q^\Meno_+\\
  Q^\Piu_+&Q^\Meno_-\epm\, .
$$
Then
$$
S=\bpm
e(d\theta)\de &e(d\theta)\bar\de\\ 
i\ell T\inv \bar\de\de+T\inv e(d\theta)\Box
& -i\ell T\inv \de\bar\de + T\inv
e(d\theta)\barBox
\epm
\bpm -Q^\Piu_-&-Q^\Meno_+\\
  Q^\Piu_+&Q^\Meno_-\epm \ .
$$

In particular,
$$
\begin{aligned}
S_{11}&=e(d\theta)(-\de Q^\Piu_-  +\bar\de Q^\Piu_+)\\
S_{12}&=e(d\theta)(-\de Q^\Meno_+  +\bar\de Q^\Meno_-)
\ .
\end{aligned}
$$

Moreover,
$$
\begin{aligned}
S_{21}
&= 
\big[i\ell T\inv \bar\de\de+T\inv e(d\theta)\Box\big](-Q^\Piu_-)
+ \big[ -i\ell T\inv \de\bar\de + T\inv
e(d\theta)\barBox\big]  Q^\Piu_+\\
& =
\ell(\bar\de\de-\de\bar\de)-ie(d\theta)\Delta_H
-(\Gamma+m) \big[i\ell T\inv (\bar\de\de+\de\bar\de)+ T\inv e(d\theta)(\Box-\Boxbar)\big]
 \\
& =  
\ell(\bar\de\de-\de\bar\de)-ie(d\theta)\big[\Delta_H +(\Gamma+m)(2m-\ell)\big]
\ .
\end{aligned}
$$

Finally, a similar computation shows that
$$
S_{22}=
-\ell(\bar\de\de-\de\bar\de)+ie(d\theta)\big[\Delta_H -(\Gamma-m)(2m-\ell)\big]
\ ,
$$
as we claimed. 

In order to conclude the proof, it suffice to notice that
$\Sigma_{jj}= R_{jj}^{-\half}$, $j=1,2$, where $R_{jj}$ are given in
  Lemma \ref{new-matrix-R}.
\qed

\bigskip

We wish now to apply Proposition \ref{detlef} to $\A_{1,\ell}^\pm.$
We restrict ourselves to $\A_{1,\ell}^-$, the other case being simpler.  

We  set 
\begin{eqnarray*}
&&\cD_1= \Xi^{p,q}, \quad H_1=\overline{\Xi^{p,q}}, \\
&&\cD_2=\S_0\Lambda^k, \quad H_2=L^2\Lambda^k,\\
&&S_1=D^-, \quad S_2=\Delta_k,
\end{eqnarray*}
where 
\begin{equation}\label{dpm}
D^\pm:=L-T^2 +i(q-p)T +\ell(n-k+\ell)+ m 
\pm\sqrt{L-T^2 +i(q-p)T +m^2},
\end{equation}
and  denote by $A$ the closure of  $\A_{1,\ell}^-$. The
commutation relation \eqref{str-intertwines1} is then satisfied
because of \eqref{diagonal}.  Moreover, clearly $S_2(\cD_2)\subset
\cD_2.$  
\medskip

Notice also that  $A$ maps $\cD_1$ bijectively onto $V_{1,\ell}^{p,q,-}\subseteq\cD_2.$  

\medskip
Next, according to Corollary \ref{V1-orth}, $A^*A$ is a positive
scalar operator, and so is $S_1.$ But then also $|A|=\sqrt{A^*A}=\sqrt{c_{s+1,\ell}} \,R_{22}^\half$ is a
scalar operator, hence commutes with $S_1,$ so that condition (iv) in
Proposition \ref{detlef} is satisfied too.  

Conditions (iii) and (v) of Proposition \ref{detlef} follow from Lemma \ref{new-matrix-R} and condition (vi) is obvious. 

Finally, our explicit formulas for $U=U_{1,\ell}^-$  in Proposition
\ref{U1ell-prop} show that here $U$ maps the space 
 $\Xi^{p.q}$ into
$\S_0\Lambda^k, $  so that $U^*$ maps $\S_0\Lambda^k $ into $\Xi^{p.q},$
and we see that $P(\cD_2)=P(\S_0\Lambda^k)=U(\Xi^{p,q})= 
A\Big(|A|^{-1}(\Xi^{p,q})\Big)=A(\Xi^{p,q})=A(\cD_1).$
This shows that also condition (vii) is satisfied.
\qed

In the same way, we see that all the hypotheses of Proposition
\ref{detlef} are satisfied by $U_{1,\ell}^-,$ and as a consequence we
obtain

\begin{prop}\label{u1l-inj}
$U_{1,\ell}^\pm$ defined by \eqref{u1pm}  maps $W_0^{p,q}$, respectively
$\Xi^{p,q}$, onto $V_{1,\ell}^{p,q,\pm}$ and intertwines $D^\pm$ with
$\Delta_k$ on the core. 

Moreover, 
$U_{1,\ell}^+:\W_0^{p,q}\to L^2\Lambda^k$  and
$U_{1,\ell}^-:\overline{\Xi^{p,q}}\to L^2\Lambda^k$ are linear
isometries onto their ranges $\cV_{1,\ell}^{p,q,+}$ and
$\cV_{1,\ell}^{p,q,-},$  respectively, 
which intertwine $D^+$  resp. $D^-$ with the restriction of $\Delta_k$
to $\cV_{1,\ell}^{p,q,\pm},$  i.e., 
\begin{equation}\label{diagonal2p+}
{\Delta_k }_{\big|_{\cV_{1,\ell}^{p,q,\pm}}} =  U^\pm_{1,\ell}\, D^\pm\,
(U^\pm_{1,\ell})\inv \quad \mbox{on}\quad  \dom {\Delta_k }_{\big|_{\cV_{1,\ell}^{p,q,\pm}}}
  \ .
\end{equation}

Here, $(U^\pm_{1,\ell})\inv$ denotes the inverse of $U^\pm_{1,\ell}$
when viewed as an operator into its range $\cV_{1,\ell}^{p,q,\pm}.$ 

Finally, if we regard of $U^\pm_{1,\ell}$ as an operator mapping into
$L^2\Lambda^k,$ then
$P_{1,\ell}^\pm:=P_{1,\ell}^{p,q,\pm}:=U_{1,\ell}^\pm(U_{1,\ell}^\pm)^*$
is the orthogonal projection from $L^2\Lambda^k$ onto
$\cV_{1,\ell}^{p,q,\pm}.$ 
\end{prop}

\bigskip

\subsection{A unitary intertwining  operator for  $V^{p,q}_{2,\ell}$}\quad
\medskip

We next wish to replace the intertwining operator $\A_{2,\ell}$ from Proposition \ref{p6.1}  by a unitary one,
denoted by $U_{2,\ell}=U_{2,\ell}^{p,q}$, which,  according to Proposition \ref{detlef}, should be given by 
 $A_{2,\ell}(A_{2,\ell}^*A_{2,\ell})^{-\half}$. In fact, it will be convenient to modify this expression introducing the unitary central factor $\sigma(T)=i\inv T/|T|$.

Recall that the non-unitary intertwining operator $\A_{2,\ell}$ from  $Z^{p,q}$ to $\cV_{2,\ell}^{p,q}$ is

\begin{equation}\label{intert2}
\begin{aligned}
A_{2,\ell} 
& =\Phi e(d\theta)^\ell
\begin{pmatrix} \bar\de\de & \de\bar\de \end{pmatrix}
=\begin{pmatrix}I\\ T\inv
  d_H^*\end{pmatrix}e(d\theta)^\ell\begin{pmatrix}
  \bar\de\de&\de\bar\de\end{pmatrix}\\
& = \begin{pmatrix} e(d\theta)^\ell \bar\de\de & e(d\theta)^\ell \de\bar\de\\
T\inv d_H^* e(d\theta)^\ell\bar\de\de &
T\inv  d_H^*e(d\theta)^\ell\de\bar\de
\end{pmatrix}\ .
\end{aligned} \smallskip
\end{equation}

Since $A_{2,\ell}$ acts on $Z^{p,q},$ the identities in Lemma
\ref{de*}  in combination with \eqref{commutations}  imply that 
$$
d_H^* e(d\theta)^\ell\bar\de\de
 = e(d\theta)^\ell \big[ (\Boxbar+i\ell T)\de -(\Box+iT)\bar\de\big]\ .
$$

Analogously,
$$
d_H^* e(d\theta)^\ell\de\bar\de
= e(d\theta)^\ell \big[ (\Box-i\ell T)\bar\de -(\Boxbar-iT)\de\big]\ .
$$

Therefore,
\begin{equation}\label{iB2}
A_{2,\ell} 
= e(d\theta)^\ell\begin{pmatrix}  \bar\de\de &  \de\bar\de\\
T\inv\big[ (\Boxbar+i\ell T)\de -(\Box+iT)\bar\de\big] &
T\inv \big[ (\Box-i\ell T)\bar\de -(\Boxbar-iT)\de\big]
\end{pmatrix}\ .\medskip
\end{equation}

\begin{lemma}\label{ip5}
\quad
We have 
\begin{enumerate}
\item[\rm(i)] 
$$
A_{2,\ell}^*A_{2,\ell} = -c_{s+1,\ell}T^{-2}E 
=: -c_{s+1,\ell}T^{-2} 
\begin{pmatrix}
E_{11} & E_{12}\\ E_{21} & E_{22}
\end{pmatrix}\ ,
$$
where
\begin{equation}\label{E}
\begin{aligned} 
E_{11}
& = \Box\Boxbar (\Delta_H-T^2) +i(\ell+1) T\Box
\big[\Delta_H-T^2 -i(n-s-\ell-1)T \big]\ ,\\
E_{12} & =  E_{21} =  -\Box\Boxbar(\Delta_H-T^2)\ ,
\\
E_{22}
& =  \Box\Boxbar (\Delta_H-T^2)  -i(\ell+1) T\Boxbar \big[\Delta_H-T^2
+i(n-s-\ell-1)T \big] \ ;
\end{aligned}
\end{equation}

\item[\rm(ii)]
$$
\big( A_{2,\ell}^*A_{2,\ell}\big)^{\half} =
\frac{\sqrt{c_{s+1,\ell}}}{|T|\sqrt{\Delta'}}
\Big[E-T^2\sqrt{c\Box\Boxbar\Delta'\Delta''}\,
I\Big], 
$$
with $E$ as above. 
\end{enumerate}

Moreover, $\big( A_{2,\ell}^*A_{2,\ell}\big)^{\half} $ maps $Z^{p,q}$ bijectively onto itself, and, on $Z^{p,q}$, 
$$
\big( A_{2,\ell}^*A_{2,\ell}\big)^{-\half} =
\frac{1}{\sqrt{c_{s+1,\ell}\, c}|T|\sqrt{\Box\Boxbar\Delta'\Delta''}}\tilde M\ ,
$$
where $\tilde M$ and  $\Delta'$ are given by
$$
\tilde M= \Box\Boxbar(\Delta_H-T^2) \bpm 1 & 1\\ 1 & 1 \epm
+ \bpm
M_{11} & 0\\ 0 & M_{22}
\epm \ ,
$$
with
\begin{equation}\label{Mmatrix}
\begin{aligned}
M_{11}
& =  -i(\ell+1) T\Boxbar\big(
\Delta_H-T^2\big)-T^2\Big( c\Boxbar +
\sqrt{c\Box\Boxbar\Delta''} \Big)\ ,\\ 
M_{22}
& =  i(\ell+1) T\Box\big(
\Delta_H-T^2\big)  -T^2\Big( c\Box +
\sqrt{c\Box\Boxbar\Delta''} \Big)
 \ ,
\end{aligned}
\end{equation}
\begin{equation}\label{Gamma2}
\begin{aligned}
\Delta'& : = \Big(2\Box\Boxbar -(\ell+1)^2T^2\Big)(\Delta_H-T^2) -T^2
\Big( -cT^2 +
  2\sqrt{c\Box\Boxbar\Delta''} \Big) \ ,\smallskip\\
\Delta''&:=\Delta_H-T^2 +c\ ,
\end{aligned}
\end{equation}
and
$$
c :=(\ell+1)(n-s-\ell-1)\ .
$$
\end{lemma}

\begin{proof}
The proof of the formulas is postponed to the Appendix, where we also prove the identity
\begin{equation}\label{detE}
\det E= cT^4\Box\Boxbar(\Delta_H-T^2+c)\ .
\end{equation}

Hence we only prove here  that $\big( A_{2,\ell}^*A_{2,\ell}\big)^{\half} $ maps $Z^{p,q}$ bijectively onto itself, assuming the validity of \eqref{E} and \eqref{detE}.

We can factor $E$ as
$$
E=-T^2\bpm \Box&0\\0&\Boxbar\epm \,E'
$$
where 
$$
\det E'=c\Delta''\ge c^2>0\ .
$$

Applying Lemma \ref{s3.1} as in the proof of Lemma \ref{new-matrix-R}, we can conclude that the operator
$$
\tilde E=-T^2\bpm \Box+\del_{p,0}\cC&0\\0&\Boxbar+\del_{q,0}\bar\cC\epm \,E'
$$
maps bijectively $(W_0^{p,q})^2$ onto itself. Restricting to $Z^{p,q}$, we obtain the conclusion.
\end{proof}

Some cancellations occur when we proceed to computing the matrix product
$A_{2,\ell}
\tilde M$,  as the next lemma shows.

\begin{lemma}\label{7.9}
We have that
$$
A_{2,\ell} \tilde M =: e(d\theta)^\ell T P = e(d\theta)^\ell T \bpm P_{11}& P_{12}\\ P_{21}&
P_{22}\epm\ ,
$$
where

\begin{equation}
\aligned
P_{11} 
& = \overline{P_{12}} = -\bar\de\de \Big[ i(\ell+1) \Boxbar\big(
\Delta_H-T^2\big)  +T\big( c \Boxbar+
\sqrt{c\Box\Boxbar\Delta''} \big) \Big] 
- e(d\theta) \Box\Boxbar (\Delta_H-T^2)\ ,\\
P_{21} & = \overline{P_{22}}
=-\de\Big[c\Boxbar(\Delta_H-T^2)+(\Boxbar+i(\ell+1) T)( c\Boxbar +
\sqrt{c\Box\Boxbar\Delta''} \big)\Big]+\bar\de\Box( c\Boxbar +
\sqrt{c\Box\Boxbar\Delta''} \big)\ .
\endaligned 
\end{equation}
\end{lemma}

\proof
Let $A=\bpm A_{11}& A_{12} \\ A_{21}& A_{22}\epm$ denote the matrix on
the right hand side of \eqref{iB2}, and set $P=
T\inv A\tilde M$.  Then
$$
\aligned 
A\tilde M
&=\left(
\begin{array}{ccc}
  \de\bar\de+\bar\de\de  & \de\bar\de+\bar\de\de  &   \\
i(\ell+1)(\de-\bar\de)  &   -i(\ell+1)(\bar\de-\de)  
\end{array}
\right) \Box\Boxbar(\Delta_H-T^2)\\
& \qquad
+ \begin{pmatrix}  \bar\de\de M_{11}&  \de\bar\de M_{22}\\
T\inv\big[ (\Boxbar+i\ell T)\de -(\Box+iT)\bar\de\big] M_{11}&
T\inv \big[ (\Box-i\ell T)\bar\de -(\Boxbar-iT)\de\big]M_{22},
\end{pmatrix}
\endaligned
$$
where by \eqref{1.6}  $  \de\bar\de+\bar\de\de =-Te(d\theta).$ This implies that 
$$
P_{11} 
=-\bar\de\de \Big[ i(\ell+1) \Boxbar\big(
\Delta_H-T^2\big)  +T\big( c \Boxbar+
\sqrt{c\Box\Boxbar\Delta''} \big) \Big] 
- e(d\theta) \Box\Boxbar (\Delta_H-T^2)
$$
and $P_{12}=\overline{P_{11}},$ 
which proves the  statements about $P_{11}$ and $P_{12},$ and 
$$
\aligned 
P_{21}
&= iT\inv (\ell+1)(\de-\bar\de)\Box\Boxbar(\Delta_H-T^2) \\
& \qquad
+ 
T^{-2}   \Big[ (\Boxbar+i\ell T)\de -(\Box+iT)\bar\de\Big]
 \Big[ - i(\ell+1) T\Boxbar\big(
\Delta_H-T^2\big)  -T^2\big( c\Boxbar +
\sqrt{c\Box\Boxbar\Delta''} \big)
 \Big]
\endaligned 
$$
 where $P_{22}=\overline{P_{21}}.$

Next, using \eqref{commutations} and the identity \eqref{box-bar}, with
$s=p+q$ in place of $k,$ we have
$$
\aligned
P_{21}& = 
iT\inv (\ell+1)(\de-\bar\de)\Box\Boxbar(\Delta_H-T^2) \\
&\qquad +
T^{-1}   \Big[ \de(\Boxbar+i(\ell+1) T) -\bar\de\Box\Big]
 \Big[ - i(\ell+1) \Boxbar\big(
\Delta_H-T^2\big)  -T\big( c\Boxbar +
\sqrt{c\Box\Boxbar\Delta''} \big)
 \Big]\\
 &= \Big[iT\inv(\de-\bar\de)\Box-iT\inv \de\Boxbar
+(\ell+1)\de+iT\inv\bar\de\Box\Big](\ell+1)\Boxbar(\Delta_H-T^2)\\
 &\qquad - \Big[\de(\Boxbar+i(\ell+1) T)  -\bar\de\Box\Big]\big( c\Boxbar +
\sqrt{c\Box\Boxbar\Delta''} \big)\\
&=-c\de\Boxbar(\Delta_H-T^2) - \Big[\de(\Boxbar+i(\ell+1) T)  -\bar\de\Box\Big]\big( c\Boxbar +
\sqrt{c\Box\Boxbar\Delta''} \big)\\
&=-\de\Big[c\Boxbar(\Delta_H-T^2)+(\Boxbar+i(\ell+1) T)( c\Boxbar +
\sqrt{c\Box\Boxbar\Delta''} \big)\Big]
+\bar\de\Box( c\Boxbar +
\sqrt{c\Box\Boxbar\Delta''} \big)\ .
\endaligned
$$
This proves the lemma.

\qed

From the previous results we immediately get an explicit formula for
$U_{2,\ell},$  at least when $p\ne0$ and $q\ne 0.$ However, if $p=0$
or $q=0,$  our formulas, when properly interpreted, persist,  and we
obtain the following result: 

Recall that if $p=0,$ then $X^{p,q}=(I-\cC)X^{p,q},$ and if $q=0,$
then $Y^{p,q}=(I-\overline\cC)Y^{p,q}.$ Let us correspondingly put 
\begin{eqnarray*}
\Box_r=\begin{cases} 
     \Box  & \mbox{ if } p\ge 1, \\
     \Box' &  \mbox{ if } p=0,
\end{cases}
\qquad
\Boxbar_r=\begin{cases} 
     \Boxbar  & \mbox{ if } q\ge 1, \\
     \Boxbar' &  \mbox{ if } q=0,
\end{cases}
\end{eqnarray*}
so that $\Box_r$ is always invertible on $X^{p,q},$ and $\Boxbar_r$ on $Y^{p,q}.$

\begin{prop}\label{U2ell-prop}                        
The operator $U_{2,\ell},$ which acts on $Z^{p,q},$ is given by 
\begin{equation}\label{U2ell-eq}
U_{2,\ell} =\frac{e(d\theta)^\ell}{\sqrt{c_{s+1,\ell}\, c}}
\,H\,
\frac{1}{\sqrt{\Delta'\Delta''}} \ ,
\end{equation}
where the operator matrix $H=\bpm H_{11}& H_{12}\\ H_{21}&
H_{22}\epm$  is defined by 
\begin{equation}
\aligned
H_{11} 
& = \overline{H_{12}} = -\bar\Ri\Ri(\Boxbar+iT)^\half \Big[ i(\ell+1) \Boxbar^\half\big(
\Delta_H-T^2\big)  +T\big( c\, \Boxbar^\half+\Box^\half
\sqrt{c\Delta''} \big) \Big] \\
&\hskip2cm - e(d\theta) \Box^\half\Boxbar^{\half} (\Delta_H-T^2)\ ,\\
H_{21} & = \overline{H_{22}}
=-\Ri\Big[c\,\Boxbar^\half(\Delta_H-T^2)+(\Boxbar+i(\ell+1) T)( c\,\Boxbar^\half+
\Box^\half\sqrt{c\Delta''} \big)\Big]\\
&\hskip2cm +\bar\Ri( c\,\Boxbar\Box^\half+\Box\Boxbar^\half
\sqrt{c\Delta''} \big)\ ,
\endaligned 
\end{equation}
and where $\Delta', \Delta''$  and $c$ are given by Lemma \ref{ip5}.
\end{prop}

\bigskip

Finally, we have the following analogue of Proposition \ref{u1l-inj}.

\begin{prop}\label{u2l-inj}
The operator $U_{2,\ell}$ in Proposition \ref{U2ell-prop}  maps the
space $Z^{p,q}$ onto $V_{2,\ell}^{p,q}$ and intertwines $D:=\Delta_0
+i(q-p)T +(\ell+1)(n-k+\ell+1)$ with $\Delta_k$ on the core. Moreover  
\medskip
$U_{2,\ell}:\overline{Z^{p,q}}\to L^2\Lambda^k$  is a  linear isometry
onto $\cV_{2,\ell}^{p,q}$ which intertwines $D$   with the restriction
of $\Delta_k$ to $\cV_{2,\ell}^{p,q},$  i.e., 
\begin{equation}\label{diagonal2p+}
{\Delta_k }_{\big|_{\cV_{2,\ell}^{p,q}}} =  U_{2,\ell}\, D\,
(U_{2,\ell})\inv \quad \mbox{on}\quad  \dom {\Delta_k }_{\big|_{\cV_{2,\ell}^{p,q}}}
  \ .
\end{equation}

Here, $(U_{2,\ell})\inv$ denotes the inverse of $U_{2,\ell}$ when
viewed as an operator into its range $\cV_{2,\ell}^{p,q}.$ 

Finally, if we regard of $U_{2,\ell}$ as an operator mapping into
$L^2\Lambda^k,$ then
$P_{2,\ell}:=P_{2,\ell}^{p,q}:=U_{2,\ell}(U_{2,\ell})^*$ is the
orthogonal projection from $L^2\Lambda^k$ onto $\cV_{2,\ell}^{p,q}.$ 
\end{prop}

\proof
This will follow by applying  Proposition \ref{detlef} to $\A_{2,\ell}.$ To this end, we set 
\begin{eqnarray*}
&&\cD_1= Z^{p,q}, \quad H_1=\overline{Z^{p,q}}, \\
&&\cD_2=\S_0\Lambda^k, \quad H_2=L^2\Lambda^k,\\
&&S_1=D, \quad S_2=\Delta_k,
\end{eqnarray*}
and  denote by $A$ the closure of  $\A_{2,\ell}$ on $Z^{p,q}.$ The
commutation relation \eqref{str-intertwines1} is then satisfied
because of \eqref{j=2}.  Moreover, clearly $S_2(\cD_2)\subset \cD_2,$  
and $A$ maps $\cD_1$ bijectively onto $V_{2,\ell}^{p,q}\subseteq\cD_2.$  

\medskip
Next, according to Lemma \ref{ip5}, $A^*A$ is a positive matrix with
scalar operator entries, and  $S_1=S_1I$ is a scalar operator. But
then also $|A|=\sqrt{A^*A}$ is a matrix with scalar operator entries,
hence commutes with $S_1I,$ so that condition (iv) in Proposition
\ref{detlef} is satisfied too.  

In order to verify  conditions (iii), (v) and (vi), we can make use
of the   joint spectral theory of $L$ and $i^{-1}T$ described in
Section \ref{Lp}.  
Indeed, it is immediate  by means of the spectral decomposition of
$S_1$ that  (vi)  
is satisfied. 

Moreover, $|A|$ maps $Z^{p,q}$ into itself; this can be verified as follows:

The formula for $|A|=\big( A_{2,\ell}^*A_{2,\ell}\big)^{\half}$ in
Lemma \ref{ip5}  shows that it suffices to prove that the operator
matrix $E$ maps $Z^{p,q}$ into itself. This in return will be verified
if we can show that $E_{12}$ maps $Y^{p,q}$ into  
$X^{p,q},$  and $E_{21}$ maps $X^{p,q}$ into $Y^{p,q}.$  But,
according to Lemma \ref{boxmap} in the Appendix, $\Box\Boxbar$ maps
$W_0^{p,q}$ into $\Xi^{p,q}$, so that the latter claims are
immediate. 

  And,  the formula for $A_{2,\ell}^*A_{2,\ell}$ in Lemma \ref{ip5} in
  combination with Lemma \ref{Ainv} and Plancherel's theorem shows
  that $A_{2,\ell}^*A_{2,\ell}=|A|^2$  has a trivial kernel in $L^2,$
  and then the same applies to $|A|,$ which proves (v). 

\medskip

Finally, our explicit formulas for $U=U_{2,\ell}$  in Proposition
\ref{U2ell-prop} show that  $U$ maps the space $Z^{p.q}$ into
$\S_0\Lambda^k, $  so that $U^*$ maps $\S_0\Lambda^k $ into $Z^{p.q},$
and we see that $P(\cD_2)=P(\S_0\Lambda^k)=U(Z^{p,q})= 
A\Big(|A|^{-1}(Z^{p,q})\Big)=A(Z^{p,q})=A(\cD_1),$ where $P=UU^*.$
This shows that also condition (vii) is satisfied, which concludes the
proof of Proposition \ref{u2l-inj}. 

\qed

\bigskip

\setcounter{equation}{0}
\section{Decomposition of $L^2\Lambda^k$}\label{Hodge}

We are now in the position to completely describe the orthogonal
decomposition of $L^2\Lambda^k$ into $\Delta_k$-invariant subspaces
and the unitary intertwining operators that reduce $\Delta_k$ into
scalar form.

\begin{thm}\label{dec-k-le-n}
Let $0\le k\le n$.  Then $L^2\Lambda^k$ admits the orthogonal
decomposition
\begin{multline}\label{eq-dec-k}
L^2\Lambda^k =  
\osum_{\substack{p+q=k<n \\ p+q=n,\, pq=0 }}
\V_0^{p,q} \, \oplus\, 
\osum_{\eps=\pm}\osum_{p+q+2\ell=k-1} 
\V_{1,\ell}^{p,q,\eps} 
\, \oplus \osum_{p+q+2\ell=k-2}   \V_{2,\ell}^{p,q} \\
\oplus\, 
\osum_{p+q=k-1} R\,\V_0^{p,q}\,  \oplus\, 
\osum_{\eps=\pm}\osum_{p+q+2\ell=k-2}R\, 
\V_{1,\ell}^{p,q,\eps} 
\, \oplus \osum_{p+q+2\ell=k-3}  
R\, \V_{2,\ell}^{p,q}\ ,\medskip
\end{multline}
where $R=R_{k-1}$ denotes the Riesz transform.
\end{thm}

\proof This follows immediately from \eqref{maindecom}, 
Proposition \ref{decom3} and Proposition \ref{s4.5}, since, according
to Lemma \ref{s4.4}, $R_{k-1}R_{k-2}=0.$ 
\qed

\bigskip

{\footnotesize
{\hspace{-1cm}
\begin{table}[htdp]
\begin{center}
\begin{tabular}{|c|c|c|c|}
\hline 
\sc{subspace} & \sc{scalar form} & \sc{intertwining} &
\sc{orthogonal}\\
& &\sc{(with domain)} & \sc{projection}
\\
\hline
$\begin{matrix}
\V_0^{p,q}\\
\text{{\scriptsize $(p+q=k<n$,}}\\
 \text{{\scriptsize or  $pq=0$ if $p+q=n)$}}
\end{matrix}$
& 
$\Delta_0 +i(q-p)T$ & $\begin{matrix} I\medskip\\(\V_0^{p,q})\end{matrix} $ & $\begin{matrix}
\phantom{\bigg|}\!\!\!I-RR^*&\!\!\!\!\! -\sum  P^{p,q,\pm}_{1,\ell}\\
& -\sum  P^{p,q}_{2,\ell}\medskip\end{matrix}$\\
\hline
$\begin{matrix}
R\V_0^{p,q}\\
\text{{\scriptsize $(p+q=k-1)$}}
\end{matrix}$ & 
$\Delta_0 +i(q-p)T$ & $\begin{matrix} R\medskip\\(\V_0^{p,q})\end{matrix} $ &  $\begin{matrix}
\phantom{\bigg|}\!\!RR^*&\!\!\!\!\! -\sum R P^{p,q,\pm}_{1,\ell}R^*\\
& -\sum  RP^{p,q}_{2,\ell}R^*\medskip\end{matrix}$\\
\hline
$\begin{matrix}
\V_{1,\ell}^{p,q,\pm}\\
\text{{\scriptsize $(p+q=k-2\ell-1)$}}
\end{matrix}$ &  
$\begin{matrix}
\phantom{\bigg|}\Delta_0 +i(q-p)T +\frac{n-p-q}{2} +\ell(n-p-q+\ell)\smallskip\ \\
\qquad
\pm\sqrt{\Delta_0 +i(q-p)T +\big(\frac{n-p-q}{2}\big)^2}\medskip
\end{matrix}$
& $\begin{matrix} U^{p,q,\pm}_{1,\ell}\medskip \\
(\W_0^{p,q},\,\text{resp. }\overline{\Xi^{p,q}})\end{matrix} $
& $\begin{matrix} P^{p,q,\pm}_{1,\ell} \\
\text{(Prop. \ref{u1l-inj})}\end{matrix} $ \\
\hline
$\begin{matrix}
R\, \V_{1,\ell}^{p,q,\pm}\\
\text{{\scriptsize $(p+q=k-2\ell-2)$}}
\end{matrix}$ & 
$\begin{matrix} 
\phantom{\bigg|}\Delta_0 +i(q-p)T +\frac{n-p-q}{2} +\ell(n-p-q+\ell)\smallskip\ \\
\qquad
\pm\sqrt{\Delta_0 +i(q-p)T +\big(\frac{n-p-q}{2}\big)^2}\medskip
\end{matrix}$
& $\begin{matrix} RU^{p,q,\pm}_{1,\ell}\medskip \\
(\W_0^{p,q},\,\text{resp. }\overline{\Xi^{p,q}})\end{matrix} $
& $R\, P^{p,q,\pm}_{1,\ell}R^* $ \\
\hline
$\begin{matrix}
\V_{2,\ell}^{p,q}\\
\text{{\scriptsize $(p+q=k-2\ell-2)$}}
\end{matrix}$ & 
$\Delta_0 +i(q-p)T +(\ell+1)(n-k+\ell+1) $
& $\begin{matrix}\phantom{\bigg|} U^{p,q}_{2,\ell}\medskip \\
(\overline{Z^{p,q}})\medskip\end{matrix} $
& $\begin{matrix} P^{p,q}_{2,\ell} \\
\text{(Prop. \ref{U2ell-prop})}\end{matrix} $ \\
\hline
$\begin{matrix}
R\, \V_{2,\ell}^{p,q}\\
\text{{\scriptsize $(p+q=k-2\ell-3)$}}
\end{matrix}$ & 
$\Delta_0 +i(q-p)T +(\ell+1)(n-k+\ell+2) $
& $\begin{matrix}\phantom{\bigg|} RU^{p,q}_{2,\ell}\medskip \\
(\overline{Z^{p,q}})\medskip\end{matrix} $
& $R\, P^{p,q}_{2,\ell}R^* $ \\
\hline
\end{tabular}
\end{center}
\bigskip
\caption{Components of $L^2\Lambda^k$, $0\le k\le n$}
\label{table}
\end{table}

}
}


The Hodge Laplacian  $\Delta_k$ leaves all the  subspaces in this
decomposition invariant, and we have seen that, after applying the
unitary intertwining operators derived in the previous sections, it
will assume a scalar form on each of the corresponding parameter
spaces.  

In Table \ref{table}, we list these  subspaces, the corresponding  scalar
forms of $\Delta_k ,$  the associated  unitary intertwining operators
as well as the orthogonal projections onto these subspaces. 

By $J,$  we denote the inclusion the operator of a given subspace into
$L^2\Lambda^k.$

\bigskip

\subsection{The $*$-Hodge operator and the case $n<k\le 2n+1$.}\quad
\medskip

We now remove the condition $0\le k\le n$ and prove a decomposition
theorem for $L^2\Lambda^k$ also in the case $n<k\le 2n+1$.

 We are going to use the $*$-Hodge
operator defined on an arbitrary Riemannian $d$-manifold  $M,$   acting for each point $m\in M$ as a linear mapping 
$$
*:\Lambda_m^k\rightarrow\Lambda_m^{d-k},
$$
where $\Lambda^k_m$ denotes the $k$-th exterior product of the dual of the tangent space at $m.$ It will be viewed also as a linear mapping acting on forms on $M.$  For its  definition and basic properties we refer to \cite{Range}. 
We summarize the main properties in the following statement.

\begin{prop}\label{*pro}
The  $*$-Hodge operator is almost involutive, i.e., ,   $*(*\om)=(-1)^{k(d-k)}\om,$ and the following properties hold true:
\begin{itemize}
\item[(1)] for $\omega_1,
\omega_2 \in L^2\Lambda^k(M)$,
$$
\int_M \omega_1\wedge *\overline{\omega_2} 
= \lan\omega_1,\, \omega_2\ran_{L^2\Lambda^k}\ ;
$$
\item[(2)] as a mapping  
$*:L^2\Lambda^k\rightarrow L^2\Lambda^{d-k},$
the operator $*$ is unitary;\smallskip
\item[(3)] $d^*=-*d*$;\smallskip
\item[(4)] $*\Delta_k=\Delta_{d-k}*$.
\end{itemize}
\end{prop}

In our situation, $M= H_n$ and  $d=2n+1.$

It follows from property (4) above that a subspace $\V\subseteq L^2
\Lambda^k$  is $\Delta_k$-invariant if and only if $*\V\subseteq L^2
\Lambda^{d-k}$ is $\Delta_{d-k}$-invariant.
Thus,
we wish to describe the $\Delta_k$-invariant subspaces of 
 $L^2 \Lambda^k$, when $n< k\le 2n+1$.\medskip

We denote by $\Lambda^k_V$ the space of {\it vertical $k$-forms,} that is, the
forms 
$\om=\theta\wedge\om_2$, with $\om_2\in\Lambda^{k-1}_H,$  
and by $\mu=\theta\wedge d\theta\wedge\cdots\wedge d\theta$ the volume
element on $H_n.$ Similarly, $\mu_H= d\theta\wedge\cdots\wedge
d\theta$ will denote the corresponding volume element on the
horizontal structure. In the same way as the $*$-Hodge operator on
$H_n$ is determined by the relations
$\sigma\wedge*\overline{\omega}=\lan \sigma,\omega\ran \mu$ for all
$\sigma, \om\in \Lambda^k,$ we can introduce  the $*$-Hodge operator
$*_H$ acting  on the  
horizontal structure, by requiring that
$\sigma\wedge*_H\overline{\omega}=\lan \sigma,\omega\ran \mu_H$ for
all $\sigma, \om\in \Lambda_H^k.$ 

The following results are easy consequences of these defining relations.

\begin{lemma}\label{rem-abv}{\rm

Let $\om\in\S_0\Lambda_H^k.$ Then the following hold true:
\begin{itemize}
\item[(i)] if we put $\om':=(-1)^k*_H\om,$  then $*\om=\theta\wedge\om';$
\item[(ii)]  $ *_H\omega=*(\theta\wedge\omega).$
\end{itemize}
}
\end{lemma}

\bigskip

We set
\begin{equation}\label{Wstar0}
\stackrel{*}{W}_0^{r,s} =\bigl\{ \om'\in\S_0\Lambda^{r,s}:\, \de\om'=\bar\de\om'=0\big\}
\end{equation}
and define
\begin{equation}\label{Z0-2}
\aligned
Z_0^{r,s} & = \bigl\{ \om=\theta\wedge\om'\in\S_0\Lambda_V^k:\,
\om'\in \stackrel{*}{W}_0^{r,s}\big\}\ ,\smallskip \\
Z_1^{r,s} & = 
\bigl\{ \om=\theta\wedge\om'\in\S_0\Lambda_V^k:\, \om'=\de^*\sigma+\bar\de^*\tau,
\ \sigma,\tau\in \stackrel{*}{W}_0^{r,s} \big\}\ , \smallskip\\
Z_2^{r,s} & = 
\bigl\{ \om=\theta\wedge\om'\in\S_0\Lambda_V^k:\, \om'=\bar\de^*\de^*\sigma+\de^*\bar\de^*\tau,
\ \sigma,\tau\in\stackrel{*}{W}_0^{r,s} \big\}\ .
\endaligned
\end{equation}

We also set
\begin{equation}\label{Z1-2ell}
Z_{j,\ell}^{r,s} = i(d\theta)^\ell Z_1^{r,s}\ ,\qquad\qquad j=1,2.
\end{equation}

Notice that
$Z_0^{r,s}$ is a subspace of $\S_0\Lambda_V^k$, where $k=r+s+1$, 
$Z_1^{r,s}\subseteq\S_0\Lambda_V^k$ with $k=r+s$,
 and $Z_2^{r,s}\subseteq\S_0\Lambda_V^k$ with $k=r+s-1$.  Therefore,
$Z_{j,\ell}^{r,s}\subseteq \S_0\Lambda_V^k$ where $k=r+s+1-j-2\ell$, $j=1,2$.
\medskip

Observe also that from \eqref{1.11} it follows that $\om\in\S_0\Lambda^k$,
$\om=\om_1+\theta\wedge\om_2$  is $d$-closed
if and only if 
$
\om_1= T\inv d_H\om_2$.  

The mapping $\stackrel{*}{\Phi}:\S_0\Lambda_V^k \rightarrow
(\S_0\Lambda^{d-k})_{d\cl}$ defined by
$$
\stackrel{*}{\Phi}(\theta\wedge\om')=T\inv d_H \om' +\theta\wedge\om'\ ,
$$
where $\om'\in\S_0\Lambda_H^k$,
is an isomorphism.

\begin{lemma}\label{*transfer}
The following properties hold true:
\begin{itemize}
\item[(i)] $*(L^2\Lambda^k)_{d\ex}=
(L^2\Lambda^{d-k})_{d^*\ex}$ and 
$*(L^2\Lambda^k)_{d^*\cl}=  (L^2\Lambda^{d-k})_{d\cl}$;\smallskip
\item[(ii)] if $\om\in\S_0\Lambda_H^k$, then
$
\ *\Phi(\om)= \stackrel{*}{\Phi}(*\om)\ . 
$
\end{itemize}
Moreover, if for given $p,q$ we put $r=n-q$ and  $s=n-p,$ then
\begin{itemize}
\item[(iii)] $*(W_0^{p,q})=Z_0^{r,s},$ hence  $\ *(V_0^{p,q})=\stackrel{*}{\Phi}(Z_0^{r,s})$;\smallskip
\item[(iv)] $*(W_{1,\ell}^{p,q})=Z_{1,\ell}^{r,s},$ hence $\ *(V_{1,\ell}^{p,q})=\stackrel{*}{\Phi}(Z_{1,\ell}^{r,s})$;\smallskip
\item[(v)] $*(W_{2,\ell}^{p,q})=Z_{2,\ell}^{r,s},$ hence $\ *(V_{2,\ell}^{p,q})=\stackrel{*}{\Phi}(Z_{2,\ell}^{r,s})$.
\end{itemize}

Finally, the spaces $Z_{j,\ell}^{r,s}$, $j=1,2$ are non-trivial, and
$Z_0^{r,s}$ are non-trivial if and only if $r+s>n$ or, if $r+s=n$, $rs=0$.
\end{lemma}

\proof
Property (i) follows from Proposition \ref{*pro} (3). 

If $\om\in \Lambda^k,$ we shall put
$$
\om':=(-1)^k*_H\om,
$$
so that according to Lemma \ref{rem-abv},  $*\omega=\theta\wedge \omega'.$  

Then
$$
\aligned
*\Phi(\om)
& = *\om +*\big(\theta\wedge T\inv d_H^* \om\big) = *\om + T\inv *_H d_H^* \om \\
& =*\om + (-1)^{(2n-k+1)(k-1)+1} T\inv d_H *_H \om \\
& = \theta\wedge\om' + T\inv d_H \om'\ ,
\endaligned
$$
which proves (ii).

Using Lemma \ref{rem-abv}, the fact that (on horizontal forms) $\de^*=-*_H\de*_H$ and the
analogous formula for $\bar\de^*$, for $\om\in W_0^{p,q}$ we obtain
$$
*\om = \theta\wedge(-1)^{p+q} *_H \om =\theta\wedge\om',
$$
where  here $\om'\in \S_0\Lambda^{r,s}$ and $\de\om' =\bar\de \om'=0$.  

This shows that $*(W_0^{p,q})\subset Z_0^{r,s},$ and in a similar way one proves that $*(Z_0^{r,s})\subset W_0^{p,q}.$ Combining this with (ii), we obtain (iii).
\medskip

Next, if $\om=\de\xi+\bar\de\eta$, with $\xi,\eta\in W_0^{p,q}$, then
\begin{equation}\label{case-ell=0}
\aligned
*\om 
& =  
(-1)^k \theta\wedge *_H(\de\xi+\bar\de\eta)\\
& = \theta\wedge (\de^* *_H\xi+\bar\de^* *_H\eta)
=: \theta\wedge  (\de^*\sigma +\bar\de^* \tau)
\endaligned
\end{equation}
where $\sigma,\tau\in \stackrel{*}{W}_0^{r,s},$ 
hence $\theta\wedge(\de^*\sigma +\bar\de^* \tau)\in Z_1^{r,s}$.
This shows that $*(W_1^{p,q})\subset Z_1^{r,s},$ and in a similar way one proves that $*(Z_1^{r,s})\subset W_1^{p,q},$ and we obtain (iv)  in the case $\ell=0$.  

For the general case, we observe that,
for all test forms $\om$ and $\sigma$,
$$
\aligned
\int\sigma\wedge *_H\overline{i(d\theta)\om}
& = \lan \sigma,\, i(d\theta)\om\ran = \lan e(d\theta)\sigma,\,
\om\ran \\
& = \int e(d\theta)\sigma\wedge *_H\overline{\om} = \int
\sigma\wedge\overline{e(d\theta)(*_H\om)}\ .
\endaligned
$$
It follows that
$$
*_H i(d\theta) = e(d\theta)  *_H\ , \qquad\text{and}\qquad
 *_H e(d\theta) = i(d\theta)  *_H\ .
$$
\smallskip

Hence, if $\om=\de\xi+\bar\de\eta$, with $\xi,\eta\in W_0^{p,q}$,
$$
\aligned
*\Phi\big( e(d\theta)^\ell \om)
& = \stackrel{*}{\Phi}\big(* e(d\theta)^\ell \om\big) \\
& = (-1)^k \theta
\wedge *_H  e(d\theta)^\ell \om + (-1)^k T\inv d_H *_H( e(d\theta)^\ell
\om)\\
& =  (-1)^k \theta \wedge i(d\theta)^\ell
 *_H  \om + (-1)^k T\inv d_H i(d\theta)^\ell
 *_H\om \\
& = \stackrel{*}{\Phi}\big( i(d\theta)^\ell \theta\wedge \om'\big)
\endaligned
$$
where $\theta\wedge\omega'=\theta\wedge(\de^*\sigma+\bar\de^*\tau)\in Z_1^{r,s}$, with
$\sigma,\tau$  as in \eqref{case-ell=0}. This shows  that
$*(W_{1,\ell}^{p,q})\subset Z_{1,\ell}^{r,s},$ and in a similar way
one proves that $*(Z_{1,\ell}^{r,s})\subset W_{1,\ell}^{p,q},$ and we
obtain (iv).   

The proof of  (v) follows along the same lines and is therefore  omitted.
\medskip

The proof about the non-triviality of these subspaces follows from
Propositions \ref{non-triviality-lemma} and \ref{non-tr}.
\qed

\begin{defn}{\rm
When $r+s\ge n$ we set
$$
Y^{r,s}_0 = \stackrel{*}{\Phi}(Z_0^{r,s})=Z_0^{r,s}\ ,\quad
Y^{r,s,\pm}_{1,\ell}  =
\stackrel{*}{\Phi}(Z_{1,\ell}^{r,s,\pm}) \ , \quad
Y^{r,s}_{2,\ell}  =
\stackrel{*}{\Phi}(Z_{2,\ell}^{r,s}) \ ,
$$
and denote by $\Upsilon^{r,s}_0, \Upsilon^{r,s,\pm}_{1,\ell}$
respectively  $\Upsilon^{r,s}_{2,\ell}$ the closures of these spaces
in $L^2\Lambda^k.$ 
}
\end{defn}

Let us finally observe that, in view of Lemmas \ref{s4.4} and
\ref{*pro},  the Riesz transforms on $H_n$ satisfy 
$$
*R_{2n-k}(\omega)=-R_{k+1}^* (*\om).
$$
Then, from Theorem \ref{dec-k-le-n}, Lemma \ref{*transfer} and
Proposition \ref{*pro} we immediately obtain the following
decomposition of  
$L^2\Lambda^k$ into $\Delta_k$-invariant subspaces  when  $n< k\le
2n+1$. 

\begin{thm}\label{dec-k>n}
Let $n< k\le 2n+1$.  Then $L^2\Lambda^k$ admits the orthogonal
decomposition
\begin{multline}\label{eq-dec-k}
L^2\Lambda^k =
\osum_{\substack{r+s=k-1>n \\ r+s=n,\, rs=0 }} \Upsilon_0^{r,s} \, \oplus\, 
\osum_{\eps=\pm}\osum_{r+s-2\ell=k} 
\Upsilon_{1,\ell}^{r,s,\eps} 
\, \oplus \osum_{r+s-2\ell=k+1}   \Upsilon_{2,\ell}^{r,s} \\
\oplus\, 
\osum_{r+s=k} R^*\,\Upsilon_0^{r,s}\,  \oplus\, 
\osum_{\eps=\pm}\osum_{r+s-2\ell=k+1}R^*\, 
\Upsilon_{1,\ell}^{r,s,\eps} 
\, \oplus \osum_{r+s-2\ell=k+2}  
R^*\, \Upsilon_{2,\ell}^{r,s}\ ,\medskip
\end{multline}
where $R^*=R^*_{k+1}.$ 

Moreover, since the $*$-Hodge operator transform the subspaces in this
decomposition into the corresponding subspaces in the decomposition
given by Theorem \ref{dec-k-le-n}, with $p:=n-s$ and $q:=n-r,$ the
unitary intertwining operators which transform $\Delta_k$ on each of
these subspaces into scalar forms are simply given by those from Table
\ref{table} at the end of Section \ref{Hodge}, composed on the right
hand 
side by the $*$-Hodge operator, and similar remarks apply to the
orthogonal projections and scalar forms. 
\end{thm}

\bigskip
\setcounter{equation}{0}
\section{$L^p$-multipliers}\label{Lp}

\bigskip

The decomposition of $L^2\Lambda^k$ presented in the previous
sections, together with the  description of the 
action of $\Delta_k$ on the various subspaces, can be used for the $L^p$- 
functional calculus of $\Delta_k$.  For this purpose, we are
going to show that $L^p\Lambda^k$
admits the same decomposition when $1<p<\infty$.  Concretely, this
means proving that the orthogonal projections on the various
invariant subspaces and the intertwining operators that reduce
$\Delta_k$ to scalar forms are $L^p$-bounded.

\subsection{The multiplier theorem}\quad
\medskip

The joint spectrum of $L$ and $i\inv T$ is the {\it
Heisenberg fan} $F\subset\R^2$ defined as follows.  If
$$
\ell_{k,\pm}=\{(\la,\xi):\xi=\pm(n+2k)\la ,\la\in\R^*_+\}\ ,
$$
then
$$
F=\overline{\bigcup_{k\in\N}(\ell_{k,+}\cup\ell_{k,-})}\ .
$$

The variable $\la$ corresponds to $i\inv T$ and $\xi$ to $L$, i.e.,
calling $dE(\la,\xi)$ the spectral measure on $F$, then 
$$
i\inv T=\int_{F}\la\,dE(\la,\xi)\ ,\qquad
L=\int_{F}\xi\,dE(\la,\xi)\ . $$

If $m$ is any bounded, continuous function on $\R\times\R^*_+,$ we can
then define the associated multiplier operator $m(i^{-1}T,L)$ by  
$$
m(i^{-1}T,L):=\int_{F}m(\la,\xi)\,dE(\la,\xi),
$$ 
which is clearly bounded on $L^2(H_n).$

It follows from Plancherel's formula that the spectral measure of
the vertical half-line $\{(0,\xi):\xi\ge0\}\subset F$ is zero. A  
spectral multiplier is therefore a function $m(\la,\xi)$ on $F$  
whose restriction to each $\ell_k$ is measurable w.r. to $d\la$ for every $k.$

We shall use the following results from \cite{MRS1,MRS2} concerning  
$L^p$-boundedness of
spectral multipliers, see also Section 5 in \cite{MPR}.

Given $\rho,\sigma>0$, we say that a measurable function $f(\la,\xi)$ is in the
mixed  Sobolev space $L^2_{\rho,\sigma }=L^2_{\rho,\sigma }(\R^2)$ if
\begin{equation}\label{L2-rho-sigma}
\aligned
\|f\|_{L^2_{\rho,\sigma
}}^2:&=\int_{\R^2}(1+|\xi'|)^{2\rho}(1+|\la'|+|\xi'|)^{2\sigma}|\hat
f(\la',\xi')|^2\,d\la'\,d\xi'\\  
&=c\|(1+|\de_\xi|)^\rho(1+|\de_\la|+|\de_\xi|)^\sigma
f\|_2^2 <\infty\ .
\endaligned
\end{equation}

Let $\eta_0\in C_0^\infty(\R)$ be a non-trivial, non-negative,  
smooth bump function supported in $\R^*_+:=(0,\infty),$  put
$\eta_1(x):=\eta_0(x)+\eta_0(-x) $ and set
$\chi:=\eta_1\otimes\eta_0.$  If  $f(\la,\xi)$ is a continuous,
bounded function on $\R\times\R^*_+,$ then we put
$f^r(\la,\xi)=f(r_1\la,r_2\xi),\ r=(r_1,r_2)\in(\R^*_+)^2$,  and say
that $f$ lies in $L^2_{\rho,\sigma,\text{\rm   
sloc}}\big(\R\times\R^*_+\big)$ if for every
$r=(r_1,r_2)\in(\R^*_+)^2$, the function
$f^r \chi$ lies in  $L^2_{\rho,\sigma }$
and 
\begin{equation}\label{rho-sigma-norm} 
\|f\|_{L^2_{\rho,\sigma,\text{\rm sloc}}}:=\sup_r\|f^r\chi\|_{
L^2_{\rho,\sigma}} <\infty \ .
\end{equation}

\begin{defn}{\rm 
A function $m$ satisfying \eqref{rho-sigma-norm} is called a
{\em Marcinkiewicz multiplier of class} $(\rho,\sigma)$. A {\em smooth Marcinkiewicz multiplier} is a Marcinkiewicz multiplier of every class $(\rho,\sigma)$, i.e., satisfying the pointwise estimates
\begin{equation}\label{pointwise}
\big|\de_\la^j\de_\xi^km(\la,\xi)\big|\le C_{jk}|\la|^{-j}|\xi|^{-k}\ ,
\end{equation}
for every $j,k$.
}
\end{defn}

\begin{thm}\label{marcin}{\rm{\bf (\cite{MRS2})}}
Let $m$ be a Marcinkiewicz multiplier of class $(\rho,\sigma)$ for some $\rho>n$  
and
$\sigma>\half$.Then $m(i\inv T,L)$ is bounded on $L^p(H_n)$ for  
$1<p<\infty$,
with norm controlled by $\|m\|_{L^2_{\rho,\sigma,\text{\rm sloc}}}$.
\end{thm}

\medskip

\subsection{Some classes of multipliers}\label{Psi-classes}\quad
\medskip

We introduce the classes $\Psi^{\rho, \sigma}_\tau$ of (possibly
unbounded) smooth multipliers, in terms of which we will understand
the behavior of the projections and intertwining operators presenteded
in the previous sections. 

These classes are defined by pointwise estimates on all derivatives,
in analogy to \eqref{pointwise}, which must be satisfied on some open
angle $\Gamma_{n-\eps}:=\{(\la,\xi)\in \R^2: \xi> 
(n-\eps)|\la|\}$ containing the Heisenberg
fan $F$ taken away the origin.

\begin{defn} 
We say that $m\in \Psi^{\rho, \sigma}_\tau$ ($\rho,\sigma,\tau\in\R$) if 
\begin{equation}\label{classes}
\big|\de^j_\la\de^k_\xi m(\la,\xi)\big|\lesssim 
\begin{cases} \xi^{\tau-j-k}&\text{ for } \xi\le 1\\
(\xi+\la^2)^{\rho-\frac j2}\xi^{\sigma-k}&\text{ for }\xi>1\ .
\end{cases}
\end{equation}
for every $j,k\in\N$.
We also say that $m\in {}^*\Psi{}^{\rho, \sigma}_\tau$ if 
$m\in \Psi^{\rho, \sigma}_\tau$ and, moreover, 
\begin{equation}\label{inverse}
m(\la,\xi)\gtrsim \begin{cases} \xi^\tau&\text{ for } \xi<1\\
(\xi+\la^2)^\rho\xi^\sigma&\text{ for }\xi>1\ .
\end{cases}
\end{equation}
\end{defn}

Prototypes are given by the smooth functions $m$ such that
$$
m(\la,\xi)=\begin{cases}(\xi+p\la+a\la^2)^\tau &\text{ for } \xi<1\\
  (\xi+\la^2)^\rho(\xi+q\la)^\sigma&\text{ for } \xi>2\ ,\end{cases} 
$$
with $|p|,|q|<n$. The following properties are easy to prove.

\begin{lemma}\label{properties}
The classes $\Psi^{\rho, \sigma}_\tau$ satisfy the following properties:
\begin{enumerate}
\item[\rm(i)] $\de_\la \Psi^{\rho, \sigma}_\tau\subset \Psi^{\rho-\half,
    \sigma}_{\tau-1}$, $\de_\xi \Psi^{\rho, \sigma}_\tau\subset
  \Psi^{\rho, \sigma-1}_{\tau-1}$; 
\item[\rm(ii)] $\Psi^{\rho, \sigma}_\tau \Psi^{\rho',
    \sigma'}_{\tau'}\subset \Psi^{\rho+\rho',
    \sigma+\sigma'}_{\tau+\tau'}$; 
\item[\rm(iii)] if $m\in {}^*\Psi{}^{\rho, \sigma}_\tau$ and 
then $m^s\in \Psi^{s\rho, s\sigma}_{s\tau}$ for every $s\in\R$ (for
$s\in\N$, $m\in \Psi^{\rho, \sigma}_\tau$ is sufficient); 
\item[\rm(iv)] if $\rho+\sigma\le\rho'+\sigma'$,
  $2\rho+\sigma\le2\rho'+\sigma'$ and $\tau\ge\tau'$, then
  $\Psi^{\rho, \sigma}_\tau\subset \Psi^{\rho', \sigma'}_{\tau'}$. 
\item[\rm(v)] 
In particular,  if $\rho+\sigma\le0$, $2\rho+\sigma\le0$ and
$\tau\ge0$, then $\Psi^{\rho, \sigma}_\tau\subset \Psi^{0,0}_0$,  and
$\Psi^{\rho, \sigma}_\tau$  
consists of Marcinkiewicz multipliers. 
\end{enumerate} \medskip
\end{lemma}

\begin{remark}\label{psi-remarks}\quad
\rm\begin{enumerate}
\item[(i)] Observe  that if $\chi$ is a smooth cut-off function on
  $\R$, compactly supported on $\R\setminus\{0\}$ and with
  $0\le\chi\le1$, then $\eta=\chi(\xi/|\la|)$ and $1-\eta$ are in
  $\Psi^{0,0}_0$. By Lemma \ref{properties} (ii), multiplication by
  $\eta$ or $1-\eta$ preserves the classes $\Psi^{\rho, \sigma}_\tau$.  
This property provides a certain amount of flexibility, of which we
give two examples.  
\item[(ii)] If we are given a multiplier $m$, which satisfies the
  inequalities \eqref{classes}, but is only defined on an angle
  $\Gamma$  leaving out a finite number of half-lines $\ell_{k,\pm}$
  of $F$, we can easily extend $m$ to a multiplier in $\Psi^{\rho,
    \sigma}_\tau$ which vanishes identically on the missing lines.  
\item[(iii)] Property (iii) in Lemma \ref{properties} also applies to
  the situation where $s>0$, \eqref{inverse} only holds on an angle
  omitting a finite number of half-lines in $F$, and $m$ vanishes
  identically on these half-lines. 
\end{enumerate}
\end{remark}

We denote by the same symbol $\Psi^{\rho, \sigma}_\tau$ the class of
operators defined by the multipliers in this  class. 
For notational convenience, we shall often use the same symbol to denote 
an operator $M\in \Psi^{\rho, \sigma}_\tau$ and (a convenient choice of) its multiplier
  $M(\la,\xi)$.

 \bigskip

\setcounter{equation}{0}
\section{Decomposition of $L^p\Lambda^k$ and boundedness of the Riesz transforms}\label{LrLambda}

\bigskip

Since the letter $p$ is already used to denote degrees of differential forms, the summability exponent will be denoted by $r$.

\medskip
If $\cV$ is any of the spaces $\cV^{p,q}_0$, $\cV^{p,q}_{1,\ell}$, $\Upsilon^{p,q}_{1,\ell}$, etc., by
$\rtrans \cV$ we shall denote  the closure of this space in
$L^{r}\Lambda^k.$ 
Our goal will be to prove the following theorem, whose parts (i) and (ii) extend Theorems~\ref
{dec-k-le-n} and  \ref{decom3}.

\begin{thm}\label{Lp-bdd}
Let $1<r<\infty.$   
\medskip
\begin{enumerate}
\item[$\rm(i)_k$]
 For $0\le k\le n$, $L^r\Lambda^k$ admits the
 direct sum  decomposition
\begin{multline}\label{eq-dec-k-r}
L^r\Lambda^k =
\osum_{\substack{p+q=k<n \\ p+q=n,\, pq=0 }} 
\rtrans\V_0^{p,q} \, \oplus\, 
\osum_{\eps=\pm}\osum_{p+q+2\ell=k-1} 
\rtrans\V_{1,\ell}^{p,q,\eps} 
\, \oplus \osum_{p+q+2\ell=k-2}   \rtrans\V_{2,\ell}^{p,q} \\
 \oplus\, 
\osum_{p+q=k-1} R_{k-1}\,\rtrans\V_0^{p,q}\,  \oplus\, 
\osum_{\eps=\pm}\osum_{p+q+2\ell=k-2}R_{k-1}\,
\rtrans\V_{1,\ell}^{p,q,\eps} 
\, \oplus \osum_{p+q+2\ell=k-3}  
R_{k-1}\,\rtrans\V_{2,\ell}^{p,q}\ ,\medskip
\end{multline}
where $R_{k-1}=d\Delta_{k-1}^{-\half}$ is the Riesz transform;
\medskip   
\item[$\rm(ii)_k$] 
For $n+1\le k\le 2n+1$, $L^r\Lambda^k$ admits the
 direct sum  decomposition
\begin{multline}\label{eq-dec-k}
L^2\Lambda^k =
\osum_{\substack{r+s=k-1>n \\ r+s=n,\, rs=0 }} \rtrans\Upsilon_0^{r,s} \, \oplus\, 
\osum_{\eps=\pm}\osum_{r+s-2\ell=k} 
\rtrans\Upsilon_{1,\ell}^{r,s,\eps} 
\, \oplus \osum_{r+s-2\ell=k+1}   \rtrans\Upsilon_{2,\ell}^{r,s} \\
\oplus\, 
\osum_{r+s=k} R^*\,\rtrans\Upsilon_0^{r,s}\,  \oplus\, 
\osum_{\eps=\pm}\osum_{r+s-2\ell=k+1}R^*\, 
\rtrans\Upsilon_{1,\ell}^{r,s,\eps} 
\, \oplus \osum_{r+s-2\ell=k+2}  
R^*\, \rtrans\Upsilon_{2,\ell}^{r,s}\ ,\medskip
\end{multline}
where $R^*=R^*_{k+1}.$    
\medskip
\item[$\rm(iii)_k$] For $0\le k\le 2n$, the Riesz transform $R_k$ is bounded from $L^r\Lambda^k$ to $L^r\Lambda^{k+1}$.
\end{enumerate}
\end{thm}
\medskip

By $L^r$-boundedness of the $*$-Hodge operator, we can restrict ourselves to the case  $0\le k\le n$.
\smallskip

The proof is based on the following lemma.

\begin{lemma}\label{Psi-intertwinings}
Let $U=\begin{pmatrix}U_{11}&U_{12}\\U_{21}&U_{22}\end{pmatrix}$ denote any of the operators $U^{p,q}_{1,\ell}$ in \eqref{U1ell-eq} or $U^{p,q}_{2,\ell}$ in \eqref{U2ell-eq}. Then each component $U_{ij}$ of $U$ consists of a multiplier operator in $\Psi^{0,0}_0$, possibly composed with powers of $e(d\theta)$ and the holomorphic and antiholomorphic Riesz transforms $\Ri$, $\bar\Ri$.

In particular, for $1<r<\infty$, all these operators are $L^r$-bounded on the spaces of differential forms of the appropriate (bi-)degrees.
\end{lemma}

This lemma will be proved in the last part of this section. Taking it for granted, we give the proof of the theorem.

\begin{proof}[Proof of Theorem \ref{Lp-bdd}]
We prove the two parts of the theorem  simultaneously, via the inductive steps $\big({\rm(i)}_{k-1}+{\rm(ii)}_{k-1}\big)\Longrightarrow {\rm(i)}_k\Longrightarrow {\rm(ii)}_k$. The statement ${\rm(i)}_0$ is trivial, and ${\rm(ii)}_0$ and ${\rm(i)}_1$ are proved in~\cite{MPR}.

Assume that ${\rm(i)}_{k-1}$ and ${\rm(ii)}_{k-1}$ hold, and consider anyone of the orthogonal  projections in the last column of Table \ref{table}.  This is a product (or a sum of products) of factors, each of which can be either $R_{k-1}$, or its adjoint $R_{k-1}^*$, or  $P=UU^*$,  $U$ being  one of the operators in Lemma \ref{Psi-intertwinings}. 
Then ${\rm(i)}_k$ follows easily.

We prove now the implication ${\rm(i)}_k\Longrightarrow {\rm(ii)}_k$.   
Factoring
$$
R_k=d\Delta_k^{-\half}=R_0\Delta_0^\half\Delta_k^{-\half}\ ,
$$
and using ${\rm(ii)}_0$, it suffices to prove the boundedness of $\Delta_0^\half\Delta_k^{-\half}$ on $L^r\Lambda^k$. 

 Referring to the decomposition \eqref{eq-dec-k-r}, we disregard the $d$-exact components of $L^r\Lambda^k$ (i.e., those with $R_{k-1}$), on which $R_k=0$, and adopt the simplified notation
$$
(L^r\Lambda^k)_{d^*{\rm-cl}}={\sum_\beta}^\oplus\,\,\, \rtrans\V_\beta\ .
$$

Denote by $U_\beta:\rtrans Z_\beta\longrightarrow \rtrans\V_\beta$ the $L^r$-closure of the unitary intertwining operator in Table \ref{table}, with $\rtrans Z_\beta$ denoting the $L^r$-closure of the appropriate space $X^{p,q}$, $Y^{p,q}$ or $Z^{p,q}$ in \eqref{Xpq-Ypq}.
Let    $P_\beta=U_\beta U_\beta^*$ be the  projection of $L^r\Lambda^k$ onto $\rtrans \V_\beta$.

Decomposing $\om\in L^r\Lambda^k$ as
$$
\om=\sum_\beta \om_\beta=\sum_\beta U_\beta\sigma_\beta\ ,
$$
with $\sigma_\beta\in \rtrans Z_\beta$, we have
$$
\Delta_0^\half\Delta_k^{-\half}\om=\sum_\beta \Delta_0^\half U_\beta D_\beta^{-\half}\sigma_\beta\ ,
$$
where $D_\beta=U_\beta^*\Delta_kU_\beta$ is the scalar operator appearing in \eqref{j=0}, \eqref{j=2}, \eqref{diagonal}. Explicitely,

$$
D_\beta=\begin{cases}
\Delta_0+i(q-p)T&\text{if }\,\rtrans\V_\beta=\rtrans \V_0^{p,q} \\
 \Delta_0 +i(q-p)T +\ell(n-k+\ell)+ m 
\pm\sqrt{\Delta_0 +i(q-p)T +m^2}&\\
\hfill\big(m=\frac{n-p-q}2\big)
&\text{if }\,\rtrans\V_\beta=\rtrans \V_{1,\ell}^{p,q,\pm}\\
\Delta_0 +i(q-p)T +(\ell+1)(n-k+\ell+1)&\text{if }\,\rtrans\V_\beta=\rtrans \V_{2,\ell}^{p,q}\text\ .
\end{cases}
$$

Denote by  $m_\beta$ be the spectral multiplier of $D_\beta$. Then, for each of the above cases,
$$
m_\beta=\begin{cases}\xi+\la^2+(p-q)\la&\in {}^*\Psi{}^{1,0}_1\\
\xi+\la^2+(p-q)\la+\ell(n-k+\ell)+ m 
+\sqrt{\xi+\la^2 +(p-q)\la +m^2}&\in{}^*\Psi{}^{1,0}_0\\
\xi+\la^2+(p-q)\la+\ell(n-k+\ell)+ m 
-\sqrt{\xi+\la^2 +(p-q)\la +m^2}\\
&\!\!\!\!\!\!\!\!\!\!\!\!\!\!\!\!\!\!\!\!\!\!\!\!\!\!\!\!\!\!\!\!\!\!\!\!\!\!\!\!\!\!\!\!\!\!\!\!\!\!\!\!\!\!\!\!\!\!\!\!\!\!\!\!\!\!\!\! \in{}^*\Psi{}^{1,0}_1\text{ if } \ell=0\ ,\  {}^*\Psi{}^{1,0}_0\text{ otherwise}\\
\xi+\la^2 +(p-q)\la +(\ell+1)(n-k+\ell+1)&\in {}^*\Psi{}^{1,0}_0\ ,
\end{cases}
$$
respectively. By Lemma \ref{properties} (iii), $D_\beta^{-\half}$ is in $\Psi^{-\half,0}_{-\half}$ or in $\Psi^{-\half,0}_0$, depending on the case. Combining together Lemma \ref{Psi-intertwinings}, Lemma \ref{properties} (ii) and (v), and the fact that the multiplier $\xi+\la^2$ of $\Delta_0$ is in ${}^*\Psi{}^{1,0}_1$, we conclude that the composition $\Delta_0^\half U_\beta D_\beta^{-\half}$ has all its components in $\Psi^{0,0}_0$.

Therefore,
$$
\begin{aligned}
\|\Delta_0^\half\Delta_k^{-\half}\om\|_r&\le\sum_\beta \|\Delta_0^\half U_\beta D_\beta^{-\half}\sigma_\beta\|_r\\
&\le C \sum_\beta \|\sigma_\beta\|_r\\
&\le C\|\om\|_r\ .
\end{aligned}
$$
\end{proof}

\subsection[lpu1]{$L^p$-  boundedness of the intertwining operators $U_{1,\ell}^\pm$}\quad
\medskip

Our next goal will be to prove
\begin{prop}\label{lpu1prop}
Assume that $p+q+1+2\ell\le n$  and $1<r<\infty.$ Then there is a constant $C_r$ so that 
\begin{eqnarray*}
\|U_{1,\ell}^{+}\,\xi\|_{L^r} \le C_r\|\xi\|_{L^r} \qquad \mbox{for every} \ \xi\in W_0^{p,q},\\
\|U_{1,\ell}^{-}\,\eta\|_{L^r} \le C_r\|\eta\|_{L^r} \qquad \mbox{for every} \ \eta\in \Xi^{p,q}.
\end{eqnarray*}
\end{prop}

\begin{proof}
According to Proposition \ref{U1ell-prop}, we have to prove the $L^r$-boundedness of the  operators:
\medskip
\begin{enumerate}
\item[(i)] $\de Q^+_- \Sigma_{11}$, $\bar\de Q^+_+\Sigma_{11}$, 
\item[(ii)] $(\bar\de\de-\de\bar\de)\Sigma_{11}$, when  $\ell\ge 1$,
\item[(iii)]  $\big[\Delta_H+(2m-\ell)(\Gamma+m)\big]\Sigma_{11}$, 
\end{enumerate}
\medskip
defined on $W_0^{p,q}$, with $p+q+2\ell+1\le n$ in (i) and (iii), and $p+q+2\ell\le n$ in (ii), and of the operators:
\medskip
\begin{enumerate}
\item[(i')] $\de Q^-_+ \Sigma_{22}$, $\bar\de Q^-_-\Sigma_{22}$,
\item[(ii')] $(\bar\de\de-\de\bar\de)\Sigma_{22}$, when  $\ell\ge 1$,
\item[(iii')]  $\big[\Delta_H-(2m-\ell)(\Gamma-m)\big]\Sigma_{22}$.
\end{enumerate}
\medskip

defined on $\Xi^{p,q}$, with $p+q+2\ell+1\le n$ in (i') and (iii'), and $p+q+2\ell\le n$ in (ii').
\medskip

Recall that if $p=0,$ then $\de=\de(I-\cC),$ and if $q=0,$ then
$\bar\de=\bar\de(I-\bar\cC),$ so that, putting again 
\begin{eqnarray*}
\Box_r=\begin{cases} 
     \Box,  & \mbox{ if } p\ge 1, \\
     \Box', &  \mbox{ if } p=0,
\end{cases}
\qquad
\Boxbar_r=\begin{cases} 
     \Boxbar,  & \mbox{ if } q\ge 1, \\
     \Boxbar', &  \mbox{ if } q=0,
\end{cases}
\end{eqnarray*}
we have
 
\begin{equation}\label{factordebar1}
\de=\Ri\Box^{\half}, \ \bar\de=\overline\Ri\,\Boxbar^{\half},
\end{equation}
where $\Ri$ and $\overline\Ri$ are the holomorphic and
antiholomorphic Riesz transforms of \eqref{4.5}, which are known to
be Calder\'on-Zygmund type singular integral operators, and
consequently  are $L^r$- bounded for $1<r<\infty$.

Moreover, observe that
$\bar\de\de-\de\bar\de=2\bar\de\de+Te(d\theta).$  Since this term
appears only when $\ell\ge 1$ and $p+q+2\ell\le n,$ we have $p+q\le
n-2,$ which easily implies that the operator $\Boxbar+iT$ is injective
on its domain in $L^2\Lambda^{p,q},$ so that we can factorize  
\begin{equation}\label{factordebar2}
\bar\de\de=\bar\Ri\Boxbar^\half\de=\bar\Ri\de(\Boxbar+iT)^{\half}=
\bar\Ri\Ri \, (\Boxbar+iT)^{\half}\Box^{\half} \mbox{  on  } W^{p,q},
\end{equation}
since, on the core, $\Boxbar\de=\de(\Boxbar+iT)$, 
hence  $\Boxbar^\half\de=\de(\Boxbar+iT)^\half.$

Observe also that $\Xi^{p,q}=(I-C_p-\bar C_q)(W_0^{p,q})$.
\medskip

Thus it will suffice to prove that the following scalar operators are in $\Psi^{0,0}_0$:
\begin{enumerate}
\item[(I)] $\Box^\half Q^+_- \Sigma_{11}$, $\Boxbar^\half
Q^+_+\Sigma_{11}$, $
\big[\Delta_H+(2m-\ell)(\Gamma+m)\big]\Sigma_{11}$, for $\ell\ge 0$; 
\item[(II)]  $(\Boxbar+iT)^{\half}\Box^{\half}\Sigma_{11}$, $i\inv T\Sigma_{11}$, for $\ell\ge 1$;
\end{enumerate}
\medskip
\begin{enumerate}
\item[(I')] $\Box^\half Q^-_+ \Sigma_{22}$, $\Boxbar^\half
Q^-_-\Sigma_{22}$, $
\big[\Delta_H-(2m-\ell)(\Gamma-m)\big]\Sigma_{22}$, for $\ell\ge 0$; 
\item[(II')]  $(\Boxbar+iT)^{\half}\Box^{\half}\Sigma_{22}$, $i\inv T\Sigma_{22}$, for $\ell\ge 1$.
\end{enumerate}
\medskip

This will be a direct consequence of the following Lemmas \ref{lpu1lemma}, \ref{lpu2lemma}, \ref{lpu3lemma}, on the basis of Lemma \ref{properties}.
\end{proof}
\medskip

Observe that $m=(n-p-q)/2\ge 1/2, \ 2m-\ell\ge 1$ in 
 (I) and (I'), and $m\ge 1,
\ 2m-\ell\ge 1$ in 
(II) and (II').

\begin{lemma}\label{lpu1lemma}
Assume that $p+q+1\le n$. Then the following hold true:
 \begin{itemize}
\item[(a)] $i\inv T\in \Psi_1^{\half,0};$
\item[(b)] 
$\Box^\half, \Boxbar^\half\in\Psi_\half^{0,\half},$
  and
  $(\Boxbar+iT)^\half\in\Psi_\half^{0,\half}$;
\item[(c)] $\Delta_H\in\Psi_1^{0,1}\subset \Psi_1^{1,0},$
  $\Delta_H-aT^2\in \Psi^{1,0}_1$ for every $a\in\C$, and
  $(\Delta_H-T^2+c)^\alpha\in\Psi_0^{\alpha,0}$ for every $c>0.$ 
 \end{itemize}
\end{lemma}
\proof 
(a) is obvious. 

As for (b), note that

$(2\Box)^\half(\la,\xi)=(\xi-(n-2p)\la)^\half$. We have  
$$
\xi-(n-2p)\la\sim\xi\ ,
$$
on an angle containing the whole fan if $p\ge1$, and, if $p=0$, on an
angle avoiding just the half-line $\xi=n\la$, $\la>0$. By Lemma
\ref{properties} (iii), 
 $\Box^\half\in\Psi_\half^{0,\half},$ and a similar argument
applies to $\Boxbar^\half$ and $(\Boxbar+iT)^\half$. Remark
\ref{psi-remarks} (iii) must be used for $\Box^\half$ when $p=0$ and
for  $(\Boxbar+iT)^\half$ when $q=n-1$.

Moreover, $(2(\Boxbar+iT))^\half(\la,\xi)=(\xi +(n-2q-2)\la)^\half$,
where $p+q+2\le n,$ hence $2(q+1)\le 2n-2,$ i.e., $|n-2(q+1)|\le n-2.$
This implies that $(\Boxbar+iT)^\half\in\Psi_\half^{0,\half}.$ 

Finally, 
\begin{equation}\label{lpu1}
\Delta_H(\la,\xi)=\xi+(p-q)\la\sim \xi\ ,
\end{equation}
on an angle containing the full fan,
which shows that $\Delta_H\in\Psi_1^{0,1}$. In combination with (a)
and Lemma \ref{properties}, this easily yields (c). 
\qed

According to \eqref{lpu1}, the quantity $\xi+(p-q)\la$, which we will also denote by $\tilde\xi$,  is comparable to $\xi$.
We then set 
\begin{equation}\label{mult-Gamma}
\Gamma(\la,\xi)=(\tilde\xi+\la^2+m^2)^\half\ .
\end{equation}

Let $R_{11}$ be as in Lemma \ref{new-matrix-R}, so that, according to
\eqref{Sigma-components}, $\Sigma_{11}=R_{11}^{-\half}$.

\begin{lemma}\label{lpu2lemma}
For $p+q+2\ell
+1\le n,$ the following hold true: 
 \begin{itemize}
\item[(a)] $\Gamma +m\,,\,Q^+_+\,,\,Q^+_-\in{}^*\Psi_0^{\half,0}$; 
\item[(b)] $R_{11}\in{}^*\Psi_0^{1,1}$, 
consequently $\Sigma_{11}=R_{11}^{-\half}\in\Psi_0^{-\half,-\half}$
\end{itemize}
\end{lemma}

\proof
We have 
$$
\Gamma(\la,\xi)=(\tilde\xi+\la^2+m^2)^\half,
$$
and, since $\tilde\xi\sim\xi$, this shows that
$\Gamma\in{}^*\Psi_0^{\half,0}$. Then (a) follows easily.

As for $R_{11},$ recall that 
\begin{equation}\label{lpu2}
\begin{aligned}
R_{11}
& =(\Gamma+m)^2\big(\Delta_H+
2m(2m-\ell)\big)+2(\Gamma+m)\big((2m-\ell)\Delta_H-2mT^2\big)\\  
&\qquad+\Delta_H(\Delta_H-T^2)+2m\ell T^2\ .
\end{aligned}
\end{equation}
By Lemma \ref{properties}, and in view of what has been shown already, we find that
$$
R_{11}\in \Psi_0^{\half,0}\Psi_0^{\half,0}\Psi_0^{0,1}+\Psi_0^{\half,0}\Psi_1^{1,0}
+\Psi_1^{0,1}\Psi_1^{1,0}+\Psi_2^{1,0}\subseteq \Psi_0^{1,1}.
$$
Moreover, since here $\Gamma\ge 0,$ we have 
\begin{eqnarray*}
R_{11}&\ge& \Gamma^2(\Delta_H+2m(2m-\ell)) + 2m(-2mT^2)+2m\ell T^2\\
&=& (\tilde\xi+\la^2+m^2)(\tilde\xi+2m(2m-\ell)) +2m(2m-\ell)|T|^2\\
&\gtrsim& (\xi+\la^2 +1)(\xi+1),
\end{eqnarray*}
which shows that also the estimates from below for $R_{11}$ hold true,
so that $R_{11}^{-\half}\in\Psi_0^{-\half,-\half}.$  This concludes
the proof of (b). 
\qed

\begin{lemma}\label{lpu3lemma}
For $p+q+2\ell
+1\le n,$ the following hold true: 
 \begin{itemize}
\item[(a)] $\Gamma -m\,,\,Q^-_+\,,\,Q^-_-\in\Psi_1^{\half,0}$; 
\item[(b)] $R_{22}\in\Psi_2^{1,1}$ and $\Sigma_{22}(I-C_p-\bar C_q)=R_{22}^{-\half}(I-C_p-\bar C_q)\in\Psi_{-1}^{-\half,-\half};$
 \end{itemize}
\end{lemma}

\proof
We have
$$
\Gamma(\la,\xi)-m=(\tilde\xi+\la^2)\big(\Gamma(\la,\xi)+m\big)\inv\in \Psi^{1,0}_1\Psi^{-\half,0}_0\subset\Psi^{\half,0}_1\ .
$$

By \eqref{R22-R11inv}, on $W_0^{p,q}$ we have the identity
$$
R_{22}=16(\Delta_H-T^2+m^2) \big(\Delta_H-T^2+\ell(2m-\ell)\big)\Box\Boxbar R_{11}\inv\ ,
$$

where
\begin{equation}\label{22}
(\Delta_H-T^2+m^2) \big(\Delta_H-T^2+\ell(2m-\ell)\big)\Box\Boxbar \in\begin{cases} \Psi^{2,2}_2&\text{ if }\ell\ne0\\
\Psi^{2,2}_3&\text{ if }\ell=0\end{cases}\ \subset \Psi^{2,2}_2\ .
\end{equation}

Applying Lemma \ref{lpu2lemma} (b), we obtain that $R_{22}\in\Psi_2^{1,1}$. 

To prove the last part of the statement, observe that the presence of
the factor $I-C_p-\bar C_q$ allows us, on the basis  of Remark
\ref{psi-remarks} (ii), to restrict, if necessary,  our analysis to an
angle omitting one of the external half-lines of $F$, where the
multipliers of $\Box$ and $\Boxbar$ are non-zero, and their reciprocal
satisfy \eqref{inverse} with $\rho=0$ and $\sigma,\tau=-1$. Each of
remaining factors in \eqref{22} is in ${}^*\Psi^{1,0}_0$, and this,
together with Lemma \ref{lpu2lemma} (b), gives the conclusion. 

\qed 

\medskip

\subsection[lpu2]{$L^p$-  boundedness of the intertwining operators $U_{2,\ell}$}\quad
\medskip

We next turn to the intertwining operator $U_{2,\ell}.$ Our goal will be to  prove
\begin{prop}\label{lpu2prop}
Assume that $p+q+2+2\ell\le n$  and $1<r<\infty.$ Then there is a constant $C_r$ so that 
\begin{eqnarray*}
\|U_{2,\ell}\,(\xi,\eta)\|_{L^r} \le C_r\|(\xi,\eta)\|_{L^r} \qquad \mbox{for every} \ 
(\xi,\eta)\in Z^{p,q}.
\end{eqnarray*}
\end{prop}

\medskip
In view of the explicit expression for  $U_{2,\ell}$ in
Proposition \ref{U2ell-prop}, it will suffice to prove that the
operators $\frac{H_{11}}{\sqrt{\Delta' \Delta''}}$ and
$\frac{H_{21}}{\sqrt{\Delta' \Delta''}}$ are $L^r$-bounded on
$X^{p,q},$ and the operators $\frac{H_{12}}{\sqrt{\Delta' \Delta''}}$
and $\frac{H_{22}}{\sqrt{\Delta' \Delta''}}$  on $Y^{p,q}$  (notice
the the multiplier $\sigma(T)$ corresponds essentially to the Hilbert
transform along the center of the Heisenberg group, which is
$L^r$-bounded). 

\medskip
We shall prove the estimates on $X^{p,q}$ only, since the estimates on
$Y^{p,q}$ 
 follow along the same lines. 
 
 Using again the factorizations \eqref{factordebar1},
 \eqref{factordebar2} by means of Riesz transforms, we see that we are
 reduced to estimating the following scalar operators on $X^{p,q}$
 with respect to the $L^r$- norm: 
 
\begin{eqnarray*}
&{\rm(III)}& \qquad
\dfrac{(\Boxbar+iT)^\half\Boxbar^\half(\Delta_H-T^2)}{\sqrt{\Delta'
    \Delta''}}, \quad
\dfrac{T(\Boxbar+iT)^\half\Boxbar^\half}{\sqrt{\Delta' \Delta''}},
\quad \dfrac{\Box_r^\half\Boxbar^\half(\Delta_H-T^2)}{\sqrt{\Delta'
    \Delta''}}, 
\\
&&\qquad  \qquad\dfrac{T \Box_r^\half(\Boxbar+iT)^\half}{\sqrt{\Delta'}},
\\
&{\rm(IV)}& \qquad \dfrac{\Boxbar^\half(\Delta_H-T^2)}{\sqrt{\Delta' \Delta''}}, \quad
 \dfrac{(\Boxbar+i(\ell+1)T)\Boxbar^\half}{\sqrt{\Delta' \Delta''}},
 \quad \dfrac{\Box_r^\half\Boxbar}{\sqrt{\Delta' \Delta''}}, \quad 
 \dfrac{\Box_r\Boxbar^\half}{\sqrt{\Delta'}},
 \\
 &&\qquad  \qquad \dfrac{(\Boxbar+i(\ell+1)T)\Box_r^\half}{\sqrt{\Delta'}}\,.
\end{eqnarray*}

\begin{lemma}\label{mul-lem-7.8}
Let $\Delta',\Delta''$ be as in Lemma \ref{ip5}.  Then, the following
properties hold:
\begin{itemize}
\item[(a)] $(\Delta'')^{\alpha}\in \Psi^{\alpha,0}_0$ for every $\alpha\in\R$;\smallskip
\item[(b)] if $q\ge 1,$ then $(\Delta')^{\alpha}\in
  \Psi^{\alpha,2\alpha}_{3\alpha}$  for every $\alpha\in\R$; \smallskip
\item[(b)] if $q\ge 1,$ then $\big(\Delta'\Delta''\big)^{-\half}\in\Psi^{-1,-1}_{-\frac 32}$.
\end{itemize}
\end{lemma}

\proof 
(a) is immediate from Lemma \ref{lpu1lemma} (c).

As for (b), we first recall that $c>0.$ Moreover, $\Delta'$ has multiplier 
\begin{eqnarray*}
&&\Big[\half
\big(\xi-(n-2p)\la\big)\big(\xi+(n-2q)\la\big)
+(\ell+1)^2\la^2\Big]\big(\xi+(p-q)\la+\la^2\big)\\ 
&&
\qquad+\la^2\Big[c\la^2+\sqrt{c\big(\xi-(n-2p)\la\big)
\big(\xi+(n-2q)\la\big)\big(\xi+(p-q)\la+\la^2+c\big)}\,. 
\end{eqnarray*}
Here,  $\xi=(n+2k)\la,\ k\in\N,$ and $k\ge 1,$ if $\la>0$ and $p=0,$
since we are acting on $X^{p,q}.$ Since we are also assuming that
$q\ge 1,$ this shows that  
\begin{equation}\label{sim2}
\xi-(n-2p)\la\sim \xi,\qquad \xi+(n-2q)\la\sim \xi\,.
\end{equation}
By means of Lemma \ref{lpu1lemma} and Lemma \ref{properties}, we thus easily see that 
$$
\Delta'\in \Psi^{0,1}_1\Psi^{0,1}_1\Psi^{1,0}_1
+ \Psi^{1,0}_2\big(\Psi^{1,0}_2+\Psi^{\half,1}_{\frac 32}\big)
\subseteq \Psi^{1,2}_3+ \Psi^{2,0}_4+\Psi^{\frac 32,1}_{\frac 72}
\subseteq \Psi^{1,2}_3\ .
$$
Moreover, the inverse estimate  \eqref{inverse} holds true for
$\rho=1,\sigma=2$ and $\tau=3$  because of \eqref{sim2}, which yields
(b). 
Finally, (c) is a direct consequence of (a) and (b).
\qed

The lemmata \ref{mul-lem-7.8} and \ref{lpu1lemma} now easily imply that
\begin{eqnarray*}
&{\rm(III)}& \
\dfrac{(\Boxbar+iT)^\half\Boxbar^\half(\Delta_H-T^2)}{\sqrt{\Delta'
    \Delta''}}\in \Psi^{0,0}_\half, \quad
\dfrac{T(\Boxbar+iT)^\half\Boxbar^\half}{\sqrt{\Delta' \Delta''}} \in
\Psi^{-\frac12,0}_\half, \quad
\dfrac{\Box_r^\half\Boxbar^\half(\Delta_H-T^2)}{\sqrt{\Delta'
    \Delta''}} \in \Psi^{0,0}_\half, 
\\
&&\qquad  \qquad\dfrac{T
  \Box_r^\half(\Boxbar+iT)^\half}{\sqrt{\Delta'}} \in \Psi^{0,0}_\half
\, ;
\\
&{\rm(IV)}& \qquad \dfrac{\Boxbar^\half(\Delta_H-T^2)}{\sqrt{\Delta'
    \Delta''}}\in \Psi^{0,-\half}_0\, , \quad 
 \dfrac{(\Boxbar+i(\ell+1)T)\Boxbar^\half}{\sqrt{\Delta' \Delta''}}
 \in \Psi^{-1,\half}_0\, , 
 \quad \dfrac{\Box_r^\half\Boxbar}{\sqrt{\Delta' \Delta''}} \in
 \Psi^{-1,\half}_0\, ,
 \\ 
 &&\qquad \qquad
 \dfrac{\Box_r\Boxbar^\half}{\sqrt{\Delta'}} \in \Psi^{-\half,\half}_0 ,
\quad  
\dfrac{(\Boxbar+i(\ell+1)T)\Box_r^\half}{\sqrt{\Delta'}} \in
\Psi^{-\half,\half}_0\, .
\end{eqnarray*}

All these classes are contained in $\Psi^{0,0}_0,$ so that all these
operators are $L^r$-bounded Marcinkiewic type operators, for
$1<r<\infty.$ This proves Proposition \ref{lpu2prop} when $q\ge 1.$  
\medskip

The situation is slightly more complicated when $q=0.$ The problem is
that the second relation in \eqref{sim2} will fail to be true in this
case on the ray  
$$
\rho:=\{(\la,\xi): \la<0\mbox {  and }  \xi=n|\la|\}\subseteq F
$$
of the Heisenberg fan, on which the multiplier of $\Boxbar$ will
vanish identically. If we remove this ray, the preceding arguments
remain valid and we get $L^r$- boundedness of the restrictions of our
operators to the orthogonal complement of the kernel of $\Boxbar,$
i.e., on $(I-\bar\cC)(X^{p,q}).$ So, what remains is the restriction
on $\bar\cC(X^{p,q}).$ This corresponds to the restrictions of our
multipliers to the ray $\rho.$ However, all of the multipliers listed
in (III) and (IV) which contain a factor $\Boxbar$ or $\Boxbar^\half$
vanish identically on this ray,  
so what remains are the operators
$$
\dfrac{T \Box_r^\half(\Boxbar+iT)^\half}{\sqrt{\Delta'}}\mbox{    and     }  \dfrac{(\Boxbar+i(\ell+1)T)\Box_r^\half}{\sqrt{\Delta'}}.
$$
On the ray $\rho,$ the multipliers of these operators are given, up to
multiplicative constants,  by 
$$
\mu_1= \dfrac{\la^2}{\sqrt{\Big(c+(\ell+1)^2\Big)\la^4+(\ell+1)^2(n-p)|\la|^3}}
$$
and 
$$
 \mu_2=\dfrac{|\la|^{\frac 32}}{\sqrt{\Big(c+(\ell+1)^2\Big)\la^4+(\ell+1)^2(n-p)|\la|^3}}\,. 
$$
It is easy to see that these are Mihlin--H\"ormander multipliers in
$\la,$ so that $\dfrac{T
  \Box_r^\half(\Boxbar+iT)^\half}{\sqrt{\Delta'}}\bar\cC$ and
$\dfrac{(\Boxbar+i(\ell+1)T)\Box_r^\half}{\sqrt{\Delta'}}\bar\cC$ are
compositions of Calder\'on-Zygmund operators acting in the central
variable of the Heisenberg group with the singular integral operator
$\bar\cC,$ which shows that they are $L^r$-bounded, for $1<r<\infty,$
too. 

\medskip
This completes the proof of Proposition \ref{lpu2prop}.
\bigskip

Finally, let us denote by $\rho^\pm$ the rays
$$
\rho^\pm:=\{(\la,\xi): \xi=\pm \la, \la >0 \subseteq F
$$
of the Heisenberg fan $F,$ and define for  $(\la,\xi)\in F$  the spaces 
\begin{eqnarray*}
X^{p,q}(\la,\xi):= \begin{cases} 
     \{0\}, &  \mbox{  if  } p=0 \mbox { and  } (\la,\xi)\in \rho^+, \\
     \C, &  \mbox{  if  } p=0 \mbox { and  } (\la,\xi)\notin \rho^+, \mbox{ or if } p>0\, ,
\end{cases}
\end{eqnarray*}
\begin{eqnarray*}
Y^{p,q}(\la,\xi):= \begin{cases} 
     \{0\}, &  \mbox{  if  } q=0 \mbox { and  } (\la,\xi)\in \rho^-, \\
     \C, &  \mbox{  if  } q=0 \mbox { and  } (\la,\xi)\notin \rho^-, \mbox{ or if } q>0\, ,
    \end{cases}
\end{eqnarray*}
and 
$$Z^{p,q}(\la,\xi)=\Big\{
\begin{pmatrix}
\mu\\
\nu
\end{pmatrix}: \mu\in X^{p,q}(\la,\xi), \nu\in Y^{p,q}(\la,\xi)\Big\}.
$$

\begin{lemma}\label{Ainv}
For $(\la,\xi)\in F$ let $E(\la,\xi)=\begin{pmatrix}
E_{11}(\la,\xi) & E_{12}(\la,\xi)\\ E_{21}(\la,\xi) & E_{22}(\la,\xi)
\end{pmatrix},$
where the $E_{ij}$ are given in Lemma \ref{ip5}.  Then, when viewed as
a linear mapping from the space $Z^{p,q}(\la,\xi)$ into itself,
$E(\la,\xi)$ is invertible for almost every $(\la,\xi)$ with respect
to the Plancherel measure on $F.$ 
\end{lemma}

\proof

When $(\la,\xi)\in F\setminus(\rho^+\cup\rho^-),$ then
$\Box(\la,\xi)\ne 0, \Boxbar(\la,\xi)\ne 0,$ and since, according to
\eqref{detE},  $\det E= cT^4\Box\Boxbar\Delta'',$  the claim is
immediate.

Assume next that $(\la,\xi)\in \rho^+.$  Then, if $p>0,$ we can argue
as before. So, assume that $p=0.$ In this case, $\Box(\la,\xi)=0, \
Z^{p,q}(\la,\xi)=\Big\{ 
\begin{pmatrix}
0\\
\nu
\end{pmatrix}: \nu\in \C\Big\},$ and 
$E(\la,\xi)=\begin{pmatrix}
0& 0\\ 0 & E_{22}(\la,\xi)
\end{pmatrix},$
where $E_{22}(\la,\xi)=-i(\ell+1)(n-q)\la(\la-n+\ell+2).$ Since
$p+q+2\ell+2\le n,$ the factor $(n-q)$ is non-zero, and the claim
follows. 

Finally, the case where $(\la,\xi)\in \rho^-$ can be dealt with in a very similar way.
\qed

\bigskip

\setcounter{equation}{0}
 \section{Applications}\label{applications}

\subsection{Multipliers of $\Delta_k$}\quad
\medskip

We are in a position now to extend Theorem 6.8 of \cite{MPR} to forms
of any degree. A function $\mu$ definied on the positive half-line is
a Mihlin--H\"ormander multiplier  of class $\rho>0$ if, given a smooth
function $\chi$ supported on $[\half,4]$ and equal to 1 on $[1,2]$,  
$$
\|\mu\|_{\rho,{\rm sloc}}:=\sup_{t>0}\big\|\mu(t\cdot)\chi\big\|_{L^2_\rho}<\infty\ .
$$

\begin{thm}\label{hodgethm}
Let $m:\R\to \C$ be a bounded, continuous function in $L^2_{\rho,{\rm
    sloc}}(\R)$ for some $\rho>(2n+1)/2$. Then, for every
$k=0,\dots,2n+1,$ the operator $m(\Delta_k)$ is bounded on
$L^p(H_n)\Lambda^k$ for $1<p<\infty,$ with norm controlled by
$\|m\|_{\rho,{\rm sloc}}$. 
\end{thm}

The proof follows the same lines as in \cite{MPR}.

\subsection[exact]{Exact $L^p$-forms}\label{subsec-exact}\quad
\medskip

As a corollary to Theorem \ref{Lp-bdd} and its proof, we can derive
the following extension of Lemma 4.2 in \cite{MPR}. 

\begin{lemma}\label{exact1}
Let $r$ be such that $1/2-1/r=1/(2n+2).$ If $\om\in L^2\Lambda^k$ is
such that $\om=du$  in the distributional sense  for some $u\in
\D'\Lambda^{k-1},$ then there is some $v\in L^r\Lambda^{k-1}$ such
that $\om=dv$ in the sense of distributions. 
Moreover, $\om \in R_{k-1}( L^2\Lambda^{k-1}).$
\end{lemma}
\proof
Define
$$
v:=L^{-\frac 12}(L^{\frac 12}\Delta_{k-1}^{-\frac 12})R_{k-1}^*\om.
$$
We have seen that the operator $\Delta_0^{\frac
  12}\Delta_{k-1}^{-\frac 12}$ is $L^p$-bounded for $1<p<\infty$, 
which implies that the same is true for $L^{\frac
  12}\Delta_{k-1}^{-\frac 12}=(L^{\frac 12}\Delta_0^{-\frac
  12})(\Delta_0^{\frac 12}\Delta_{k-1}^{-\frac 12}).$ As in the proof
of Lemma 4.2 in \cite{MPR}, we can thus conclude that $v\in
L^r\Lambda^{k-1}.$ And, if $\xi\in \S\Lambda^{k},$ then 
\begin{eqnarray*}
\lan dv,\xi\ran
=\lan \om, R_{k-1}(\Delta_{k-1}^{-\half}L^\half)L^{-\half} d^*\xi\ran
=\lan\om, R_{k-1}R_{k-1}^*\xi\ran=\lan R_{k-1}R_{k-1}^*\om,\xi\ran,
\end{eqnarray*}
so that $dv=R_{k-1}R_{k-1}^*\om\in L^2\Lambda^k.$ By Lemma \ref{s4.4},
this implies that  
$$
\om=dv+R_k^*R_k\om,
$$
and by the same lemma $R_k\om=\Delta_{k+1}^{-\half}d\om,$ where
$d\om=d^2u=0$ in the sense of distributions. This implies that
$\om=dv,$ and thus also that $\om\in R_{k-1}( L^2\Lambda^{k-1}).$ 
\endproof

\begin{cor}\label{exact2}
If $\om\in L^2\Lambda^k,$ then $\om \in R_{k-1}( L^2\Lambda^{k-1})$ if
and only if there is some $u\in \D'\Lambda^{k-1}$ such that $\om=du$
in the sense of distributions. 
\end{cor}

\proof
One implication is immediate by Lemma \ref{exact1}. To prove the
converse  implication, let us assume that $\om\in
R_{k-1}(L^2\Lambda^{k-1}).$  Then, according to Lemma \ref{s4.4} and
Proposition \ref{s4.5}, $\om=R_{k-1}R_{k-1}^*\om.$ Moreover, if we
define $v$ as in the proof of Lemma \ref{exact1}, then $v\in
L^r\Lambda^{k-1}$ and $dv=R_{k-1}R_{k-1}^*\om,$  hence $dv=\om.$ We
may  thus choose $u=v$. 
\endproof

 \subsection{The Dirac operator}\label{dirac}\quad
\medskip
 
Let us denote by $\Lambda=\osum\limits_{k=0}^{2n+1}\Lambda^k$  the
Grassmann algebra of $\h_n^*,$ and by
$L^p\Lambda=L^p(H_n)\Lambda=\osum\limits_{k=0}^{2n+1} L^p\Lambda^k,
\S\Lambda$ etc. the space of $L^p$-section,  $\S$-sections etc. of the
corresponding bundle over $H_n.$  \medskip

The {\it Dirac operator} acting on $\S\Lambda$ is given by 
\begin{equation}\label{Dirac-op}
D:=d+d^*\ .
\end{equation}

Notice that
$D^2=\Delta$
on $\dom(\Delta)$,
that is,
the Dirac operator  $D$ and the Hodge Laplacian $\Delta$ commute as
differential operators.

However, in order to reduce the spectral theory of $D$  to that one of
$\Delta$ 
we need to show that $D$ and $\Delta$ {\em strongly commute}, in the
sense that the all spectral projections in the spectral decompositions of
$D$ and $\Delta$ commute.

\begin{prop}\label{Dirac-Delta}
We have that
$\overline{D}^2=\Delta$. In particular, 
$D$ and $\Delta$  strongly commute.
\end{prop}
\proof
Recall from the previous section that the Riesz transform
$R=d\Delta^{-\half}$ and its adjoint
$R^*=\Delta^{-\half}d^*=d^*\Delta^{-\half}$ are $L^p$-bounded for
$1<p<\infty.$ Let us put  
$$
P_\pm:=
\frac1{\sqrt{2}}\Big( I\pm D\Delta^{-\half}\Big)=\frac 1{\sqrt{2}}\Big( I\pm (R+R^*)\Big).
$$
One easily verifies that $P_\pm ^2=P_\pm$ and $P_\pm^*=P_\pm,$ so that
$P_+$ and $P_-$ are orthogonal projections, which are in fact
$L^p$-bounded for $1<p<\infty.$ Moreover, 
\begin{equation}\label{dirac1}
DP_\pm=\pm \Delta^\half P_\pm, 
\end{equation}
i.e., 
$$
D=\Delta^\half P_+ - \Delta^\half P_-\ .
$$

Let $\Delta=\int_0^{+\infty}\la\, dE(\la)$, so that
$$
\Delta^\half=\int_0^{+\infty}\sqrt\la\, dE(\la)
= \int_0^{+\infty}s\, d\tilde E(s)\ ,
$$ 
where $\tilde E$ denotes the image of the spectral measure $E$ under
the mapping $\la\mapsto\sqrt\la$.  Therefore,
$$
\begin{aligned}
D
& =\Delta^\half P_+ - \Delta^\half P_-
& = \int_0^{+\infty}s\, d(\tilde E P_+)(s) - 
\int_0^{+\infty}s\, d(\tilde E_s P_-)(s) 
\end{aligned}
$$

Now, if $A\subset\R$ is a Borel set, let us put
\begin{equation}\label{F(A)}
F(A):= \tilde E(A_+)P_+ +\tilde E (-A_-)P_-\ ,
\end{equation} 
where $A_+:=A\cap[0,+\infty)$ and $A_-:=A\cap(-\infty,0).$
Then   $F$ is a spectral measure on $\R,$ and 
$$
D=\int_{-\infty}^{+\infty} s\, dF(s)\qquad \mbox{on}\  \ \S\Lambda\ .
$$
Indeed, notice that, since the operators $R,R^*$ are bounded and commute with $\Delta$
on the core, they also commute with the spectral projections $\tilde
E(B)$, i.e.,
\begin{equation}\label{E(B)}
\tilde E(B)P_\pm = P_\pm \tilde E(B)\ .
\end{equation}
Moreover, we clearly have 
$$
F(A)=F(A_+)+F(A_-)\ ,
$$
and $P_+P_-=P_-P_+=0.$ This implies that $F(A)$ is an orthogonal
projection, and that $F$ is a spectral measure on $\R$.

We set
\begin{equation}\label{D-tilde}
\tilde D:= \int_{-\infty}^{+\infty} s\, dF(s)\ ,
\end{equation}
as a closed operator, and claim that indeed $\tilde D=\overline D.$ 

To verify this, denote by $D_0$ the restriction of $D$ to $\S_0\Lambda.$ Since   $D=\tilde D$ on  $\S\Lambda,$  we then have 
 $\overline{
D_0}\subset \overline D\subset \tilde D$, and clearly 
$$
\tilde D^2 = \int_{-\infty}^{+\infty} s^2\, dF(s)
= \int_0^{+\infty} \la\, dE(\la) =\Delta\ .
$$
Thus, it remains to show that $\tilde D\subset \overline{D_0}$.

Let $\xi\in\dom\tilde D$.  It suffices
to assume that there is an interval $I=[a,b]$, with $0<a<b$, so that
$F(\R\setminus K)\xi=0$, where $K=I\cup(-I)$; hence
$\tilde D\xi=\int_K s\, dF(s) \xi$.

Let $\varphi$ be a smooth cut-off function, even, identically $1$ on
$K$ and with support contained in $K'=I'\cup(-I')$, where $I'=[a',b']$
with
$0<a'<a<b<b'$. Then, 
$$
\Delta \xi=\tilde D^2 \xi =\int_K  s^2\, dF(s) \xi =
\int_K  s^2\varphi(s) \, dF(s) \xi
=\int_0^\infty  \la\psi(\la) \, dE(\la) \xi\ ,
$$
where $\psi$ is a smooth function with compact support in
${K'}^2=\{s^2:\, s\in K'\}$.  

Hence, if $Q:=\psi(\Delta)=\int_0^\infty \psi(\la) dE(\la)$, we have that
$ \Delta \xi= \Delta Q \xi$.
It is clear  that $Q$ is given by right-convolution with a Schwartz function, and that $Q$ is $U(n)$-equivariant,  hence it preserves the core $\S_0\Lambda$ for $\Delta.$

\smallskip
Choose a sequence $\{\xi_n\}\subset\S_0\Lambda$ such that $\xi_n\rightarrow\xi$
and $\Delta\xi_n \rightarrow\Delta\xi$. 

Then  $\{Q\xi_n\}\subset\S_0$, $Q\xi_n\rightarrow Q\xi$ and
$\Delta Q\xi_n \rightarrow\Delta \xi$. Therefore,
we may assume that $\xi_n=Q\xi_n$ and $\xi=Q\xi$.  Then
$$
\begin{aligned}
 \Delta^\half \xi
 & = \Delta^\half Q\xi = \int_0^\infty  \la^\half \psi(\la)\, dE(\la) \xi\\
& = \lim_{n\rightarrow+\infty} \int_0^\infty  \la^\half \psi(\la)\, dE(\la) \xi_n \\
& = \lim_{n\rightarrow+\infty} \int_0^{\infty} s\varphi(s) \, dF(s)
\xi_n\ .
\end{aligned}
$$

Hence, by \eqref{F(A)} it follows that
$$
\tilde D\xi = \lim_{n\rightarrow+\infty} \tilde D\xi_n=\lim_{n\rightarrow+\infty} D_0\xi_n\ .
$$
This implies that $\S_0\Lambda$ is a core also for $\tilde D$; hence $\tilde
D=\overline {D_0}.$ We have thus seen that 
$$
\overline D = \int_{-\infty}^{+\infty} s \, dF(s)\ .
$$

By \eqref{F(A)} and \eqref{E(B)} $F(A)$ and $E(B)$ commute; hence
$\overline D$ and $\Delta$ strongly commute.
\qed

Moreover, if $m$ is a bounded, Borel measurable spectral
multiplier of $\R,$ then  
\begin{equation}\label{dirac2}
m(D)=m(\Delta^\half)P_+ +m(-\Delta^\half)P_-.
\end{equation}

As an immediate consequence of Theorem \ref{hodgethm} and Theorem \ref{Lp-bdd}
we therefore obtain
\begin{cor}\label{diracthm}
Let $m:\R\to \C$ be a bounded, continuous function in $L^2_{\rho,{\rm
    sloc}}(\R)$ for some $\rho>(2n+1)/2.$ Then $m(D)$ is bounded on
$L^p(H_n)\Lambda$ for $1<p<\infty,$ with norm controlled by
$\|m\|_{\rho,{\rm sloc}}.$ 
\end{cor}

\setcounter{equation}{0}
\section{Appendix}\label{appendix}

In this final section we collect some technical facts and proofs that
we have previously set aside.

We need some preliminary
computations. \medskip

Recall first from Lemma \ref{de*} that 
\begin{equation}\label{commute-e}
[\de^*,e(d\theta)^\ell]=i\ell\bar\de e(d\theta)^{\ell-1}\
,\qquad[\bar\de^*,e(d\theta)^\ell]=-i\ell\de e(d\theta)^{\ell-1}\
. 
\end{equation}

Taking adjoints, this implies
\begin{equation}\label{commute-i}
[\de,i(d\theta)^\ell]=i\ell\bar\de^* i(d\theta)^{\ell-1}\ , \qquad
[\bar\de,i(d\theta)^\ell]=-i\ell\de^* i(d\theta)^{\ell-1}\ . 
\end{equation}

\begin{lemma}\label{ip2} If $\sigma\in\ker
  i(d\theta)\subset\Lambda_H^s$ and $s+2j\le n$,
  then
$$
i(d\theta)^j e(d\theta)^j\sigma=c_{s,j}\sigma \ ,
$$
where the coefficients $c_{s,j}$ are defined in \eqref{csj}, i.e., 
$$
c_{s,j}=\frac{j!(n-s)!}{(n-s-j)!}\, .
$$
Moreover,  the following relations hold:
$$
j ^2c_{s+1,j-1}= c_{s,j}- c_{s+1,j},\qquad
c_{s,j}(n-s-j)= c_{s+1,j}(n-s), \qquad jc_{s+1,j-1}(n-s-j)=c_{s+1,j}\,  .
$$
\end{lemma}

\proof Use formula (2.8) in \cite{MPR} to compute
  $i(d\theta)e(d\theta)^j\sigma$. 
Observe that $\omega_j$ there corresponds to our $\sigma$ and the value
$k=p+q$ there is our $s+2j$. Therefore, 
$$
i(d\theta)e(d\theta)^j\sigma=j(n-s-j+1)e(d\theta)^{j-1}\sigma\ .
$$

Consequentely,
$$
i(d\theta)^j
e(d\theta)^j\sigma=j(n-s-j+1)i(d\theta)^{j-1}e(d\theta)^{j-1}\sigma\ ,
$$
and the statement follows inductively.
\endproof

If $\xi\in W^s_0$, then
$\xi,\de\xi,\bar\de\xi$ are in $\ker i(d\theta)$, and consequently we see that

\begin{equation}\label{ipn1}
i(d\theta)^j e(d\theta)^j\xi=c_{s+1,j}\xi\ ,\quad i(d\theta)^j
e(d\theta)^j\de\xi=c_{s+1,j}\de\xi\ ,\quad i(d\theta)^j
e(d\theta)^j\bar\de\xi=c_{s+1,j}\bar\de\xi\ . 
\end{equation}

This is not necessarily the case for $\de\bar\de\xi$,
$\bar\de\de\xi$. We therefore need some more computations to simplify
the expressions 
$$
\de^*\bar\de^*i(d\theta)^\ell e(d\theta)^\ell \bar\de\de\ ,\
i(d\theta)^\ell e(d\theta)^{\ell-1}\bar\de\de\ ,\ \text{ etc.} \ .
$$

\begin{lemma} \label{ip3}
For $\xi\in W_0^s$,
$$
\begin{aligned}
\de^*\bar\de^*i(d\theta)^j
e(d\theta)^j\bar\de\de\xi&=c_{s+1,j}\big(\barBox+i(j+1)
T\big)\Box\xi\ ,\\ 
\bar\de^*\de^*i(d\theta)^j
e(d\theta)^j\de\bar\de\xi&=c_{s+1,j}\big(\Box-i(j+1)
T\big)\barBox\xi\ ,\\ 
\de^*\bar\de^*i(d\theta)^j
e(d\theta)^j\de\bar\de\xi&=\bar\de^*\de^*i(d\theta)^j
e(d\theta)^j\bar\de\de\xi=-c_{s+1,j}\barBox\Box\xi\ . 
\end{aligned}
$$
\end{lemma}

\proof We have, by \eqref{commute-i}
$$
\begin{aligned}
\de^*\bar\de^*i(d\theta)^j
e(d\theta)^j\bar\de\de\xi&=\de^*\bar\de^*[i(d\theta)^j,\bar\de
]e(d\theta)^j\de\xi
+\de^*\bar\de^*\bar\de i(d\theta)^j e(d\theta)^j\de\xi\\ 
&=ij\de^*\bar\de^*\de^*i(d\theta)^{j-1}e(d\theta)^j\de\xi
+c_{s+1,j}\de^*\bar\de^*\bar\de\de\xi\\
&=ij T\de^*i(d\theta)^j e(d\theta)^j\de\xi
+c_{s+1,j}\de^*\bar\de^*\bar\de\de\xi\\
&=c_{s+1,j}\de^*(\barBox+ij T)\de\xi\\
&=c_{s+1,j}\big(\barBox+i(j+1) T\big)\Box\xi\ .
\end{aligned}
$$
The second identity follows in a similar way. Finally, we have 
$$
\begin{aligned}
\de^*\bar\de^*i(d\theta)^j e(d\theta)^j\de\bar\de\xi
&=\de^*\bar\de^*[i(d\theta)^j,\de ]e(d\theta)^j\bar\de\xi
+\de^*\bar\de^*\de i(d\theta)^j e(d\theta)^j\bar\de\xi\\
&=c_{s+1,j}\de^*\bar\de^*\de\bar\de\xi\\
&=-c_{s+1,j}\de^*\de\bar\de^*\bar\de\xi\\
&=-c_{s+1,j}\barBox\Box\xi\ .
\end{aligned}
$$
This proves the lemma.
\endproof

\begin{lemma}\label{boxmap}
The following properties hold true.
\begin{itemize}
\item[(i)] For all $p,q$, $\Box W^{p,q}_0=  X^{p,q}$ and
$\Boxbar W^{p,q}_0=  Y^{p,q}$.
\smallskip
\item[(ii)] In particular, $\Box$ maps $X^{p,q}$ into itself, and  $\Boxbar$ maps $Y^{p,q}$
into itself. We therefore set 
\begin{equation}\label{BoxYBoxbarX}
\Box_X =\Box_{|_{X^{p,q}}}:X^{p,q}\rightarrow X^{p,q}\, ,\qquad 
\Boxbar_Y =\Boxbar_{|_{Y^{p,q}}}:Y^{p,q}\rightarrow Y^{p,q}\, .
\end{equation}
Then $\Box_X\inv$ and $\Boxbar_Y\inv$ are well defined on $X^{p,q}$
and $Y^{p,q}$, respectively.
 \end{itemize}
\end{lemma}

\proof
(ii) is immediate from (i). To prove (i),  we verify the statements
concerning the spaces $X^{p,q},$ the discussion of the spaces
$Y^{p,q}$ being similar. Observe first that $\Box$  leaves $W^{p,q}_0$
invariant because of \eqref{commutations}. 

Thus, in  view of Remark \ref{projpar}, if $1\le p\le n-1,$ then
$X^{p,q} =W^{p,q}_0,$ and we are done.  And, if $p=0$ and $ \xi\in
W^{0,q}_0,$ then $\Box\xi=\de^*\de\xi,$ so that by Lemma \ref{s4.8}
$C_0\Box\xi=0,$ hence $\Box\xi\in X^{p,q}.$
Finally, if $p=n$  and $ \xi\in W^{n,q}_0,$ then $\Box\xi=\de\de^*\xi=0.$ 
\qed

\proof[Proof of Lemma \ref{ip4}]
Define $B_{1,\ell}$ to be the unbounded operator
from $L^2\Lambda^{p,q}$ to $(L^2\Lambda^k)^2$ 
defined by the matrix on the right-hand side of \eqref{i4}, with core $\S_0\Lambda^{p,q},$  and set
\begin{equation}\label{iB}
\begin{aligned}
B_{1,\ell}^*{B_{1,\ell}}_{\displaystyle|_{(W^{p,q}_0)^2}}=: B
& = \bpm B_{11}& B_{12}\\ B_{21} &B_{22} \epm \ .
\end{aligned}
\end{equation}

We will use Lemma \ref{rmrk} to see that
$$
A_{1,\ell}^*A_{1,\ell}=B_{1,\ell}^*{B_{1,\ell}}_{\displaystyle|_{(W^{p,q}_0)^2}}\ .
$$

Then, the matrix entries of $B$ are:
$$
\begin{aligned}
B_{11}
&=\de^*i(d\theta)^\ell e(d\theta)^\ell\de\\
& \qquad +\Big(i\ell T\inv
\de^*\bar\de^*i(d\theta)^{\ell-1}-T\inv \Box i(d\theta)^\ell\Big)\Big(i\ell
T\inv e(d\theta)^{\ell-1}\bar\de\de+T\inv e(d\theta)^\ell\Box\Big)\\ 
B_{22}
&=\bar\de^*i(d\theta)^\ell e(d\theta)^\ell\bar\de\\
& \qquad
+\Big(-i\ell
T\inv
\bar\de^*\de^*i(d\theta)^{\ell-1}-T\inv \barBox i(d\theta)^\ell\Big)\Big(-i\ell
T\inv e(d\theta)^{\ell-1}\de\bar\de+T\inv e(d\theta)^\ell\barBox\Big)\\ 
B_{12}&=B_{21}^*
=\de^*i(d\theta)^\ell e(d\theta)^\ell\bar\de\\
& \qquad -\Big(i\ell T\inv
\de^*\bar\de^*i(d\theta)^{\ell-1}-T\inv \Box i(d\theta)^\ell\Big)\Big(i\ell
T\inv e(d\theta)^{\ell-1}\de\bar\de-T\inv e(d\theta)^\ell\barBox\Big)\ . 
\end{aligned}
$$

By \eqref{ipn1} and Lemma \ref{ip3}  we have that
\begin{equation}\label{i5}
\begin{aligned}
B_{11}
&=c_{s+1,\ell}\Box +T^{-2} \Big[ 
-\ell^2 \de^* \bar\de^* i(d\theta)^{\ell-1} e(d\theta)^{\ell-1} \bar\de\de
+i\ell \de^* \bar\de^* i(d\theta)^{\ell-1}e(d\theta)^\ell \Box \\
& \qquad\qquad
-i\ell \Box i(d\theta)^\ell  e(d\theta)^{\ell-1} \bar\de\de
-\Box i(d\theta)^\ell e(d\theta)^\ell \Box \Big] \\ 
& = c_{s+1,\ell}\Box +T^{-2} \Big[ -\ell^2 c_{s+1,\ell-1}
(\Boxbar+i\ell T)\Box \\
& \qquad\qquad
-i\ell T\inv \de^* \bar\de^* i(d\theta)^{\ell-1}e(d\theta)^{\ell-1}(\de\bar\de+\bar\de\de) \Box \\
& \qquad\qquad
-i\ell T\inv \Box (\de^*\bar\de^*+\bar\de^*\de^*)i(d\theta)^{\ell-1}  e(d\theta)^{\ell-1} \bar\de\de
-c_{s,\ell}\Box^2 \Big] \ . 
\end{aligned}
\end{equation}

Now notice that by Lemma \ref{ip3}
\begin{eqnarray*}
\de^* \bar\de^*
i(d\theta)^{\ell-1}e(d\theta)^{\ell-1}(\de\bar\de+\bar\de\de)
&=&-c_{s+1,\ell-1}\Boxbar\Box+c_{s+1,\ell-1}(\barBox+i\ell T)\Box \\
&=& i\ell c_{s+1,\ell-1}T \Box \ ,
\end{eqnarray*}

and, by taking adjoints, 
$$
(\de^*\bar\de^*+\bar\de^*\de^*)i(d\theta)^{\ell-1}
e(d\theta)^{\ell-1} \bar\de\de
= 
i\ell c_{s+1,\ell-1}T \Box \ .
$$
Therefore, substituting into \eqref{i5} and applying the identities from Lemma \ref{ip2}, we obtain that
$$
\begin{aligned}
B_{11}
& =  c_{s+1,\ell}\Box +T^{-2} \Big[ -\ell^2 c_{s+1,\ell-1}
(\Boxbar+i\ell T)\Box + 2\ell^2 c_{s+1,\ell-1} \Box^2 -c_{s,\ell}\Box^2 \Big] \\
& = T^{-2}\Box \Big[ c_{s+1,\ell}T^2 
 +\ell^2 c_{s+1,\ell-1}(\Box-\Boxbar-i\ell T)
+(-c_{s,\ell} +\ell^2 c_{s-1,\ell-1})\Box \Big]\\
& = T^{-2}\Box \Big[ c_{s+1,\ell}T^2 
 +i\ell^2 c_{s+1,\ell-1}(n-s-\ell)T
-c_{s+1,\ell}\Box \Big]\\
& = c_{s+1,\ell}T^{-2}\Box \Big[ T^2 +i\ell T -\Box\Big]\ .
\end{aligned}
$$

By conjugation, we also get
$$
B_{22}
= c_{s+1,\ell}T^{-2}\Boxbar \Big[T^2 -i\ell T -\Boxbar\Big]\ .
$$

Finally, arguing in  a similar way, we find that 
$$
\begin{aligned}
B_{12}
& =  T^{-2}\Big[ \ell^2\de^*\bar\de^* i(d\theta)^{\ell-1}
e(d\theta)^{\ell-1} \de\bar\de -c_{s,\ell}  \Box\Boxbar
+i\ell \de^*\bar\de^*  i(d\theta)^{\ell-1} e(d\theta)^\ell \Boxbar \\
& \qquad 
+i\ell \Box  i(d\theta)^\ell e(d\theta)^{\ell-1} \de\bar\de \Big]\\
& = T^{-2}\Big[ -(\ell^2 c_{s+1,\ell-1} +c_{s,\ell})  \Box\Boxbar
-i\ell T\inv \de^*\bar\de^*  i(d\theta)^{\ell-1}
 e(d\theta)^{\ell-1}(\de\bar\de+\bar\de\de) \Boxbar \\
& \qquad \qquad
+i\ell T\inv \Box(\de^*\bar\de^*+\bar\de^*\de^*)  
i(d\theta)^{\ell-1} e(d\theta)^{\ell-1} \de\bar\de \Big]\\ 
& = T^{-2}\Box\Boxbar \Big[  -\ell^2 c_{s+1,\ell-1} -c_{s,\ell}
+i\ell c_{s+1,\ell-1} T\inv  \Boxbar 
-i\ell c_{s+1,\ell-1} T\inv  (\Boxbar +i\ell T)\\
& \qquad 
-i\ell c_{s+1,\ell-1} T\inv \Box +i\ell c_{s+1,\ell-1} T\inv
(\Box-i\ell T) \Big] \\
& = T^{-2}\Box\Boxbar \Big[  -\ell^2 c_{s+1,\ell-1} -c_{s,\ell}
+2\ell^2 c_{s+1,\ell-1}  \Big]\\
& = -c_{s+1,\ell} T^{-2}\Box\Boxbar \ .
\end{aligned}
$$

These computations show that
$$
B= B_{1,\ell}^*{B_{1,\ell}}_{\displaystyle|_{(W^{p,q}_0)^2}} =
-c_{s+1,\ell}T^{-2} 
\bpm
\Box(\Box-i\ell T -T^2) & \Box\Boxbar \\
\Box\Boxbar & \Boxbar(\Boxbar+i\ell T- T^2) 
\epm\ .
$$
It is obvious that $B$ maps $(W^{p,q}_0)^2$ into itself, so
that,
by Lemma \ref{rmrk}, 
$ A_{1,\ell}^*A_{1,\ell}=
B_{1,\ell}^*{B_{1,\ell}}_{\displaystyle|_{(W^{p,q}_0)^2}}$.
This proves the lemma. 
\endproof

\begin{proof}[
Completion of the proof of Lemma \ref{new-matrix-R}]
Since $Q$ maps the subspace $(W_0^{p,q})^2$ into itself, we have that
$$
\begin{aligned}
R&=\bpm -Q^\Piu_-&Q^\Piu_+\\ -Q^\Meno_+&Q^\Meno_-\epm   \bpm
\Box(\Box-i\ell T -T^2) & \Box\Boxbar \\
\Box\Boxbar & \Boxbar(\Boxbar+i\ell T- T^2) 
\epm
  \bpm -Q^\Piu_-&-Q^\Meno_+\\ Q^\Piu_+&Q^\Meno_-\epm\\
  &=\bpm R_{11}& R_{12}\\R_{21}&R_{22}\epm\ .
  \end{aligned}
  $$

We compute first $R_{11}$. Using the identity $\Box-\barBox=2imT$ and
\eqref{Q-iden}, we have
$$
\begin{aligned}
-T^2R_{11}&=(Q^*NQ)_{11}\\
&=\big(-Q^\Piu_-\Box(\Box-i\ell T -T^2)+Q^\Piu_+\Box\barBox\big)(-Q^\Piu_-)+
\big(-Q^\Piu_-\Box\barBox+Q^\Piu_+ \Boxbar(\Boxbar+i\ell T- T^2)
\big)Q^\Piu_+\\ 
&=(Q^\Piu_-)^2\Box(\Box-i\ell T
-T^2)-2Q^\Piu_+Q^\Piu_-\Box\barBox+(Q^\Piu_+)^2\Boxbar(\Boxbar+i\ell T- T^2)\\ 
&=(\Gamma+m+iT)^2\Box(\Box-i\ell T
-T^2)-2(\Gamma+m-iT)(\Gamma+m+iT)\Box\barBox\\ 
&\qquad+(\Gamma+m-iT)^2\Boxbar(\Boxbar+i\ell T- T^2)\\ 
&=(\Gamma+m)^2\Big[\Box(\Box-i\ell T
-T^2)-2\Box\barBox+\Boxbar(\Boxbar+i\ell T- T^2)\Big]\\ 
&\qquad+2iT(\Gamma+m)\Big[\Box(\Box-i\ell T
-T^2)-\Boxbar(\Boxbar+i\ell T- T^2)\Big]\\ 
&\qquad-T^2\Big[\Box(\Box-i\ell T
-T^2)+2\Box\barBox+\Boxbar(\Boxbar+i\ell T- T^2)\Big]\\ 
&=(\Gamma+m)^2\Big[(\Box-\barBox)^2-i\ell
T(\Box-\barBox)-T^2(\Box+\barBox)\Big]\\ 
&\qquad+2iT(\Gamma+m)\Big[(\Box+\barBox)(\Box-\barBox)-i\ell
T(\Box+\barBox)-T^2(\Box-\barBox)\Big]\\ 
&\qquad-T^2\Big[(\Box+\barBox)^2-i\ell
T(\Box-\barBox)-T^2(\Box+\barBox)\Big]\\ 
&=-T^2\Big[(\Gamma+m)^2\big(\Delta_H+
2m(2m-\ell)\big)+2(\Gamma+m)\big((2m-\ell)\Delta_H-2mT^2\big)\\  
&\qquad+\big(\Delta_H(\Delta_H-T^2)+2m\ell T^2\big)\Big]\ .
\end{aligned}
$$

In the same way one finds that
$$
\begin{aligned}
-T^2R_{22}&=(Q^*NQ)_{22}\\
&=-T^2\Big[(\Gamma-m)^2\big(\Delta_H+2m(2m-\ell)\big)
-2(\Gamma-m)\big((2m-\ell)\Delta_H-2mT^2\big)\\
&\qquad+\big(\Delta_H(\Delta_H-T^2)+2m\ell T^2\big)\Big]\ .
\end{aligned}
$$

Finally, using the formulas in \eqref{Q-iden} we see that
$$
\begin{aligned}
-T^2R_{12}&=-T^2R_{21}=(Q^*NQ)_{12}\\
&=\big(-Q^\Piu_-\Box(\Box-i\ell T -T^2)+Q^\Piu_+\Box\barBox\big)(-Q^\Meno_+)\\
&\qquad+ \big(-Q^\Piu_-\Box\barBox+Q^\Piu_+ \Boxbar(\Boxbar+i\ell T- T^2) \big)Q^\Meno_-\\
&=Q^\Piu_-Q^\Meno_+\Box(\Box-i\ell T -T^2)-Q^\Piu_+Q^\Meno_+\Box\barBox-Q^\Piu_-Q^\Meno_-\Box\barBox\\
&\qquad+Q^\Piu_+Q^\Meno_-\Boxbar(\Boxbar+i\ell T- T^2)\\
&=2\Box\Boxbar(\Box-i\ell T -T^2)-(\Delta_H-2T^2-2iT\Gamma)\Box\barBox\\
&\qquad-(\Delta_H-2T^2+2iT\Gamma)\Box\barBox+2\Box\Boxbar(\barBox+i\ell T -T^2)\\
&=0\ . 
\end{aligned}
$$
\end{proof}

\bigskip

\proof[
Proof of Lemma \ref{ip5} (i)]

In order to compute $A_{2,\ell}^*A_{2,\ell}$ we apply again Lemma
\ref{rmrk}. We first define $B_{2,\ell}$ as the term on the right 
hand side of \eqref{iB2} acting as an unbounded operator from $\bigl(L^2
\Lambda^{p,q}\bigr)^2$ to $\bigl(L^2
\Lambda^{p,q}\bigr)^2,$ with core $\bigl(\S_0
\Lambda^{p,q}\bigr)^2,$ and then we compute $B:=
{B_{2,\ell}^*B_{2,\ell}}_{\displaystyle|_{(W^{p,q}_0)^2}}$.  If we
then show that $B$ maps $Z^{p,q}$ into itself, the equality
$A_{2,\ell}^*A_{2,\ell}=B_{\displaystyle|_{Z^{p,q}}}$, which in turn equals $-c_{s+1,\ell}T^{-2}E$,
will follow. \medskip

We have that
\begin{multline*}
B_{2,\ell}^*B_{2,\ell}
= 
\begin{pmatrix}  \de^*\bar\de^* & -T\inv
\big[ \de^*(\Boxbar+i\ell T) -\bar\de^* (\Box+iT)\big] \\
\bar\de^* \de^* & 
-T\inv \big[ \bar\de^*(\Box-i\ell T) -\de^*(\Boxbar-iT)\big]
\end{pmatrix}
\\
\times i(d\theta)^\ell
e(d\theta)^\ell \begin{pmatrix}  \bar\de\de &  \de\bar\de\\
T\inv\big[ (\Boxbar+i\ell T)\de -(\Box+iT)\bar\de\big] &
T\inv \big[ (\Box-i\ell T)\bar\de -(\Boxbar-iT)\de\big]
\end{pmatrix}\ .
\end{multline*}

Using Lemmas \ref{ip3} and \ref{ip2}, and recalling that we are acting
on elements in $W^{p,q}_0$, we see  that 
the matrix entries of $B_{2,\ell}^*B_{2,\ell}$ are
$$
\begin{aligned}
B_{11}
&=\de^* \bar\de^*i(d\theta)^\ell e(d\theta)^\ell\bar\de\de\\
& \qquad\quad -T^{-2} 
\big[ \de^*(\Boxbar+i\ell T) -\bar\de^* (\Box+iT)\big]i(d\theta)^\ell
e(d\theta)^\ell 
\big[ (\Boxbar+i\ell T)\de -(\Box+iT)\bar\de\big] \\
& = c_{s+1,\ell} \big[\Boxbar+i(\ell+1)T\big]\Box\\
& \qquad\quad - c_{s+1,\ell} T^{-2} 
\big[ \de^*(\Boxbar+i\ell T) -\bar\de^* (\Box+iT)\big]
\big[ (\Boxbar+i\ell T)\de -(\Box+iT)\bar\de\big] \\
& = c_{s+1,\ell} \big[\Boxbar+i(\ell+1)T\big]\Box \\
& \qquad\quad - c_{s+1,\ell} T^{-2} 
\big[ (\Boxbar+i(\ell+1) T)\de^* -\Box\bar\de^* \big]
\big[ (\Boxbar+i\ell T)\de -(\Box+iT)\bar\de\big] \\
& = c_{s+1,\ell}T^{-2} 
\Big[\big[\Boxbar+i(\ell+1)T\big]\Box T^2 
- (\Boxbar+i(\ell+1) T)^2 \Box
- \Box^2\Boxbar\Big] \\
& = c_{s+1,\ell}T^{-2} \Box
\Big[ \big[\Boxbar+i(\ell+1)T\big] T^2 
- \big(\Boxbar+i(\ell+1) T\big)^2 - \Box\Boxbar\Big] \ .
\end{aligned}
$$
From this it follows that
$$
\begin{aligned}
E_{11}
& = -\Box
\Big[ \big(\Boxbar+i(\ell+1)T\big) T^2 
-\big(\Boxbar+i(\ell+1) T\big)^2
-
\Box\Boxbar\Big] \\
& = \Box\Boxbar (\Delta_H-T^2 )  
+i(\ell+1) T\Box\big[2\Boxbar -T^2+i(\ell+1)T \big] \\
& = \Box\Boxbar (\Delta_H-T^2 )  
+i(\ell+1) T\Box\big[\Delta_H -T^2-i(n-s-\ell-1)T \big] \ ,
\end{aligned}
$$
where we have used the the equality \eqref{1.10}.

Thus, the statement  
for $E_{11}$ follows.  The term $E_{22}$ is
its complex conjugate and thus it follows as well.

Finally, we compute $E_{12}$, and hence $E_{21}$ too.  We have that
$$
\begin{aligned}
B_{12} 
&  =\de^* \bar\de^* i(d\theta)^\ell e(d\theta)^\ell\de\bar\de \\
& \quad 
- T^{-2} \Big[
\de^* (\Boxbar +i\ell T) -\bar\de^*(\Box+iT)  \Big] i(d\theta)^\ell e(d\theta)^\ell
\Big[(\Box -i\ell T)\bar\de -(\Boxbar-iT)\de\Big]\ 
 .
\end{aligned}
$$
Therefore, using Lemmas \ref{ip3} and \ref{ip2},
$$
\begin{aligned}
B_{12}
&=-c_{s+1,\ell}  \Box\Boxbar
+c_{s+1,\ell} T^{-2} \Big[
(\Boxbar+i(\ell+1) T)\de^* (\Boxbar-iT)\de +\Box\de^* (\Box-i\ell
T)\bar\de \Big]  \\
& =-c_{s+1,\ell}  \Box\Boxbar
+c_{s+1,\ell} T^{-2} \Big[
\big(\Boxbar+i(\ell+1) T\big)\Box\Boxbar +\big(\Box-i(\ell+1) T\big)\Box\Boxbar\Big]  \\
& = c_{s+1,\ell} T^{-2} \Box\Boxbar(\Delta_H -T^2)\ .
\end{aligned}
$$
Recalling that $B_{12}=-c_{s+1,\ell}T^{-2}E_{12}$, the assertion follows.
\medskip

It is now easy to check that $B$ maps $Z^{p,q}$ into itself.  For, 
suppose that $\xi\in X^{p,q}$ and  $\eta\in Y^{p,q}, $ and put
$\sigma= B_{11}\xi +B_{12}\eta.$  
Observe that  both $B_{11}$ and $B_{22}$ factor as $B_{1j}=\Box
D_{1j}, \ j=1,2,$ where $D_{1j} $ leaves $W_0^{p,q}$ invariant.
Therefore Lemma \ref{boxmap} shows that  
 $\sigma \in X^{p,q}.$  In a similar way, one shows that $B_{21}\xi
 +B_{22}\eta\in Y^{p,q,}$ which concludes the proof. 
\qed

To compute the  square root of a matrix,  
we shall make use of the following formula, which is an application of Cayley-Hamilton's
theorem: 
For a positive definite $2\times2$ matrix $A$ we have that
\begin{equation}\label{cayley-hamilton}
A^{\half} =\frac{A+\sqrt{\det A}\, I}{\sqrt{\tr A+2\sqrt{\det A}}} \ .
\medskip
\end{equation}

\proof[Proof of Lemma \ref{ip5} (ii)]
We have
$$
(A_{2,\ell}^*A_{2,\ell})^{\half}=\frac{\sqrt{c_{s+1,\ell}}}{|T|}E^\half, \quad 
(A_{2,\ell}^*A_{2,\ell})^{-\half}
= \frac{|T|}{\sqrt{c_{s+1,\ell}}} E^{-\half},
$$
where the matrix $E$ is as in Lemma \ref{ip5}.  

Then, recalling that we set $c=(\ell+1)(n-s-\ell-1)$, 
$$
\begin{aligned}
\tr E
&= 2\Box\Boxbar(\Delta_H -T^2) +i(\ell+1)T(\Box-\Boxbar)(\Delta_H-T^2)
+ cT^2\Delta_H \\
& = 2\Box\Boxbar(\Delta_H -T^2) -(\ell+1)(n-s)T^2(\Delta_H-T^2)
+ cT^2\Delta_H\\
& = \big[2\Box\Boxbar-(\ell+1)^2T^2\big] (\Delta_H -T^2) 
+ cT^4\ .
\end{aligned}
$$

Moreover, using \eqref{1.10} and recalling that the operators are
acting on $s$-forms, we have
$$
\begin{aligned}
\det E
& = \Box\Boxbar(\Delta_H -T^2) \Big[i(\ell+1)T\Box
(\Delta_H-T^2 -i(n-s-\ell-1)T)\\
& \qquad\qquad\qquad\qquad
-i(\ell+1)T\Boxbar (\Delta_H-T^2 +i(n-s-\ell-1)T) \Big] \\
& \qquad\qquad
+(\ell+1)^2T^2\Box\Boxbar\Big[ (\Delta_H-T^2)^2+(n-s-\ell-1)^2T^2
\Big] \\
& = \Box\Boxbar(\Delta_H -T^2) \Big[i(\ell+1)T(\Delta_H -T^2)(\Box
-\Boxbar)+c T^2\Delta_H \Big] \\
& \qquad\qquad
+(\ell+1)^2T^2\Box\Boxbar\Big[ (\Delta_H-T^2)^2+(n-s-\ell-1)^2T^2
\Big] \\
& = \Box\Boxbar(\Delta_H -T^2) T^2\Big[-(\ell+1)(n-s) (\Delta_H -T^2)
+c \Delta_H \Big] \\
& \qquad\qquad
+(\ell+1)^2T^2\Box\Boxbar\Big[ (\Delta_H-T^2)^2+(n-s-\ell-1)^2T^2
\Big] \\
& = c\Box\Boxbar T^4 \Big[
\Delta_H-T^2 +c\Big]\\
& = c\Box\Boxbar T^4 \Delta'' \ .
\end{aligned}
$$
This implies in particular that 
$$
\begin{aligned}
\Delta'  & := \tr E +2\sqrt{\det E} \\
& =  \big[2\Box\Boxbar-(\ell+1)^2T^2\big] (\Delta_H -T^2) 
+ cT^4 + 2\sqrt{c\Box\Boxbar T^4 \Delta''}\\
& = \big[2\Box\Boxbar-(\ell+1)^2T^2\big] (\Delta_H -T^2) 
-T^2\Big( -cT^2 + 2\sqrt{c\Box\Boxbar \Delta''}\Big)\ .
\end{aligned}
$$
The formula for $(A_{2,\ell}^*A_{2,\ell})^{\half}$ in Lemma \ref{ip5} (ii) is now immediate.

And, if we write 
$$
E\inv =\frac{1}{\det E} E^{\text{(co)}}
$$
then we have
$$
\begin{aligned}
(A_{2,\ell}^*A_{2,\ell})^{-\half}
& = \frac{|T|}{\sqrt{c_{s+1,\ell}}\sqrt{\det E}} {E^{\text{(co)}}}^{\half} \\
& = \frac{|T|}{\sqrt{c_{s+1,\ell}}\sqrt{\det E}} 
\frac{1}{\sqrt{\tr E^{\text{(co)}} +2\sqrt{\det E^{\text{(co)}}}}} \big(E^{\text{(co)}}+\sqrt{\det E^{\text{(co)}}}I\big) \\
& = \frac{|T|}{\sqrt{c_{s+1,\ell}}\sqrt{\det E}} 
\frac{1}{\sqrt{\tr E +2\sqrt{\det E}}} \big(E^{\text{(co)}}+\sqrt{\det E}I\big) \\
& =
\frac{1}{\sqrt{c_{s+1,\ell}}|T|\sqrt{c\Box\Boxbar\Delta''}}\frac{1}{\sqrt{\Delta'}} 
 \big(E^{\text{(co)}}+\sqrt{\det E}I\big) \ .
\end{aligned}
$$

The result for  $(A_{2,\ell}^*A_{2,\ell})^{-\half}$ now follows easily, since
$$
\begin{aligned}
& E^{\text{(co)}}+\sqrt{\det E}\, I \\
& =  \Box\Boxbar (\Delta_H-T^2 ) \bpm 1 & 1\\ 1& 1\epm \\
& \qquad + 
\left(
\begin{array}{cc}
 \Boxbar\big[-i(\ell+1)T\big(\Delta_H -T^2\big)-cT^2 \big] & 0 \\
0 &  \Box\big[i(\ell+1)T\big(\Delta_H -T^2\big)-cT^2 \big]
\end{array}
\right)\\
& \qquad\qquad -T^2\sqrt{c\Box\Boxbar\Delta''}\,I \\
& = \Box\Boxbar (\Delta_H-T^2 ) \bpm 1 & 1\\ 1& 1\epm + \bpm M_{11} &
\\ & M_{22} \epm
\ ,
\end{aligned}
$$
where $M_{11}$ and $M_{22}$ are as claimed in \eqref{Mmatrix}.
\qed

\end{document}